\documentclass[10pt]{article}
\usepackage{amssymb,amsmath,amsthm,epsfig}
\usepackage{txfonts}
\usepackage{latexsym, enumerate,cite}
\usepackage{eepic}
\usepackage{epic}
\usepackage{caption}
\usepackage{graphicx}
\usepackage{color}
\usepackage{ifpdf}
\usepackage{subfigure}
\usepackage{tikz}
\usepackage{dsfont}
\usepackage{multirow}
\usepackage{makecell}
\usepackage{algorithm}
\usepackage{bm}
\usepackage{multirow}
\usepackage{color}
\usepackage[colorlinks, linkcolor=blue,anchorcolor=blue,citecolor=blue,urlcolor=blue]{hyperref}
\usepackage{appendix}
\usepackage{comment}
\graphicspath{{figs/}}

\numberwithin{equation}{section}

\theoremstyle{plain}
\newtheorem{lem}{Lemma}[section]
\newtheorem{thm}[lem]{Theorem}

\newtheorem{prop}{Proposition}[section]

\theoremstyle{definition}
\newtheorem{defn}{Definition}[section]

\newtheorem{rem}{Remark}[section]

\newcommand{\p}{\partial}
\newcommand{\ds}{\displaystyle}
\newcommand{\vp}{\phi}

\newcommand{\RR}{ \mathbb{R}}
\newcommand{\pd}[2]{\frac {\p #1}{\p #2}}

\newcommand{\nm}{\noalign{\smallskip}}

\newcommand{\Scal}{\mathcal{S}}
\newcommand{\Ocal}{\mathcal{O}}
\newcommand{\Kcal}{\mathcal{K}}
\newcommand{\Dcal}{\mathcal{D}}
\renewcommand{\(}{\left(}
\renewcommand{\)}{\right)}

\def \e{\ensuremath{\mathrm{e}}}
\begin{document}

\title{Spectral theory of the Neumann-Poincar\'e operator on multi-layered structures and analysis of plasmon mode splitting}
%\title{Spectral theory of the Neumann-Poincar\'e operator in multi-layer structures with applications to plasmon hybridization}
\author{Youjun Deng \thanks{Guangxi Key Laboratory of Universities Optimization Control and Engineering Calculation, Guangxi Minzu University, Nanning, Guangxi 530006, China. School of Mathematics and Statistics, Central South University, Changsha, 410083, Hunan Province, China. \ \ Email: youjundeng@csu.edu.cn; dengyijun\_001@163.com}
\and{Lingzheng Kong}\thanks{%Corresponding author. 
	School of Mathematics and Statistics, Central South University, Changsha, 410083, Hunan
		Province, China. Email: math\_klz@csu.edu.cn; math\_klz@163.com}
\and
{Zijia Peng}
\thanks{Guangxi Key Laboratory of Universities Optimization Control and Engineering Calculation, and School of Mathematical Sciences, Guangxi Minzu University, Nanning, Guangxi 530006, China. E-mail: pengzijia@126.com}
%\and
%{Zaiyun Zhang}
%\thanks{School of Mathematics, Hunan Institute of Science and Technology, Yueyang 414006, Hunan, China. Email: zhangzaiyun1226@126.com}
\and
{Liyan Zhu}
\thanks{School of Mathematics and Statistics, Central South University, Changsha, 410083, Hunan Province, China.\ \ Email: math\_zly@csu.edu.cn}
}

\date{}%
\maketitle
% ----------------------------------------------------------------
\begin{abstract}
In this paper, we develop a general mathematical framework for analyzing electostatics within multi-layered metamaterial structures. The multi-layered structure can be designed by nesting complementary negative and regular materials together, and it can be easily  achieved  by truncating bulk metallic material in a specific configuration. Using layer potentials and symmetrization techniques, we establish the perturbation formula in terms of Neumann-Poincar\'e (NP) operator for general multi-layered medium, and obtain the spectral properties of the NP operator, which demonstrates that the number of plasmon modes increases with the number of layers. Based on Fourier series, we present an exact  matrix representation of the NP operator in an apparently unsymmetrical structure,  exemplified by multi-layered confocal ellipses. By highly intricate and delicate analysis, we establish a handy algebraic framework for studying the splitting of the plasmon modes within multi-layered structures. Moreover, the asymptotic profiles of the plasmon modes are also obtained. This framework helps reveal the effects of material truncation and rotational symmetry breaking  on the splitting of the plasmon modes, thereby inducing desired resonances and enabling  the realization of customized applications.
\end{abstract}


%	35J05  	Laplace operator, Helmholtz equation (reduced wave equation), Poisson equation
%    35B30 	Dependence of solutions to PDEs on initial and/or boundary data and/or on parameters of PDEs

% ----------------------------------------------------------------
\section{Introduction}

We are concerned with the mathematical investigation of electostatics within multi-layered structure and the peculiar plasmon resonance phenomena it induces. Plasmon resonance refers to the resonant oscillation of conduction electrons at the interface between negative and positive permittivity material caused by the background field. Negative materials can be created using high-contrast particles compared to the surrounding medium through homogenization theory in specific configurations \cite{CGSJLMS23,AmmariJMA17}. Another commonly utilized category of negative materials includes noble metallic materials \cite{KPLM_PRB2014,Maier07,YA_PNASU18}. Negative materials have been proposed for groundbreaking applications such as super-resolution imaging and invisibility cloaking. Multi-layered structures are frequently employed as the foundational components in constructing metamaterial devices for these cutting-edge applications \cite{PRHN2003SCI,YA_SIAMREV18}. It is important to note that the plasmon resonances induced by these negative structures serve as the fundamental basis for these applications.

In  \cite{Ammari2016,AMRZARMA17}, the authors studied the surface plasmon resonance of metallic nanoparticles in electromagnetic wave using the Drude model, with applications in wave imaging. These nanoparticles can be viewed as single-layer structures. In \cite{MGWNNAP_PRSA06,Ammari2013,ACKLM2,BS11}, the core-shell structures were proposed for achieving invisibility cloaking through the phenomenon known as anomalous localized resonance. In this configuration, the core contains a regular material while the shell contains a negative material, both being uniform. By precisely aligning their material properties according to specific geometric arrangements, cloaking due to anomalous localized resonance can be induced by some appropriate incident field. The core-shell device can be regarded as a two-layer structure. The aforementioned studies have also been extended in \cite{LLLWESAIMM2AN19,LLPRSA18} for macro-scale scenarios,  in \cite{PMSOPRSA19,RMSOPRB22,DLZJMPA21,DLZJDE22} for  anisotropic structures (nanorod, slender-body), and in \cite{Chung2014} for two-layer confocal ellipses. Similarly, one-layer or two-layer metamaterial structures have also been extensively used and theoretically investigated in linear elasticity, see e.g. \cite{DLL201,LL16,LLL16,AKKY17,AKKY18,DLbook2024} and the references cited therein. Multiband operation and cloaking of larger objects have been achieved using multi-layered plasmonic covers, providing more degrees of freedom to suppress a greater number of scattering orders \cite{Alu2008,Chen2012}. Consequently, three-layer or four-layer structures, known as nanomatryushkas  (a matryushka is a nested set of
dolls found in Russian folk art), have been widely employed  and theoretically investigated using the hybridization model \cite{PRHN2003SCI,FDLMMA15,PNJCP2004,KPN_NA13} and the references cited therein. Notably, the plasmon resonance of multi-layered concentric radial structures has been investigated in \cite{FangdengMMA23,KZDF}, where plasmon modes interact and hybridize, resulting in localized surface plasmons excited at each surface. This phenomenon explains the physical reason why the number of plasmon modes splitting (PMS) corresponds to the number of layers in the structure. This would enable an even greater degree of versatility in designing the cloak, accommodating multiple operating frequencies and potentially wider frequency bandwidths.

Despite the important work mentioned above, there are still many outstanding questions regarding PMS within multi-layered structures waiting to be answered. One such question deals with PMS on multi-layered structures  of general shape: most of the aforementioned works focus on radial structures. As the first step to addressing this challenging question, we consider in this paper the structure of multi-layered confocal ellipses. The number of layers can be arbitrary and the material parameters in each layer may be different, though uniform.
This can encompass many of the existing studies mentioned above, where the number of layers is one or two. However, these results cannot be easily generalized to arbitrary multi-layered structures, which poses a significant challenge.
To the authors' best knowledge, the physical reason for the PMS within multi-layered confocal ellipses has not been adequately discussed in the literature and neither has the equivalence of the multi-layered  circular and elliptical structures.  We prove that like the radial case  there is an exact perturbed field in terms of the \textit{generalized polarization matrix} (GPM)  which contains the material information, together with the geometric information of multi-layered structures.  The explicit formula for the determinant of the GPM, which is equivalent to the characteristic polynomial of matrix type Neumann-Poincar\'e (NP) operator in multi-layered structures, is a stumbling block in the analysis of PMS.

To prove the main result of this paper, we first derive the matrix type NP operator $\mathbb{K}_N^*$ (see \eqref{NPoperator}) in the $N$-layer structure with $C^{1,\eta}\, (0<\eta<1)$ smooth interface using layer potentials. Secondly,  by delicate analysis, we develop a spectral theory (see Proposition \ref{NPspectral}) of NP operator, especially the Calder\'on-type identity in multi-layered structures, which serves as a highly nontrivial extension of the corresponding study about two-layer structures  in \cite{Ammari2013}.
Based on computations of the NP operator on single-layer ellipses in \cite{AKLJCM07}, we obtain two exact matrix representation of $\mathbb{K}_N^*$: even NP matrix $\(\mathbf{K}_{N,c}^{(n)}\)^T$ and odd NP matrix $\(\mathbf{K}_{N,s}^{(n)}\)^T$ (see Corollary \ref{cor51}), which results from a breaking of rotational symmetry; however the modes specifically split into even and odd parities because the multi-layered confocal ellipses still possesse two mirror planes. After struggling against complicated structure of the matrics $\(\mathbf{K}_{N,c}^{(n)}\)^T$ and $\(\mathbf{K}_{N,s}^{(n)}\)^T$, we come up with a conjecture of the exact formulation of characteristic polynomial. The conjecture is then verified by backward induction.
By using the spectral theory of NP operator, we conclude that the even NP matrix $\(\mathbf{K}_{N,c}^{(n)}\)^T$ and odd NP matrix $\(\mathbf{K}_{N,s}^{(n)}\)^T$ have $N$ real eigenvalues that lie in the interval $[-1/2,1/2]$, respectively.  Each eigenvalue is associated with one type of plasmon modes, which means that truncating the material will result in plasmon mode splitting.
Furthermore, the asymptotic profiles of the plasmon modes in multi-layered confocal ellipses are obtained as the elliptic radii approach infinity (i.e., multi-layered concentric discs) or zero (i.e., multi-layered thin strips).
This analysis establishes an algebraic framework that offers a convenient and practical approach for studying plasmon resonances associated with multi-layered metamaterial structures. In practical applications, the multi-layered structure can serve as a fundamental building block for various material devices. Our algebraic framework will provide a powerful and general design principle that can be applied to select appropriate material parameters, both qualitatively and quantitatively, guiding the design of metallic nanostructures and predicting their resonant properties. We shall investigate along this
direction in our forthcoming work.

The remainder of this paper is organized as follows.
In Section \ref{sec2}, we first present the electrostatic scattering problem with multi-layered structures. Second, we transform the problem into an integral system using layer potentials and derive the form of the NP operator in multi-layered structures.
In Section \ref{sec3}, we develop a spectral theory for NP operator.
Section \ref{sec4} presents the representation of the perturbed field in terms of the GPM.
In Section \ref{sec5}, we establish an algebraic framework for plasmon resonances  and give the main results on the explicit formula of the characteristic polynomial and the estimation of the roots.
Section \ref{sec6} is devoted to the proofs of the main results in Section \ref{sec5}.
In Section \ref{sec7}, numerical examples are presented  to corroborate our theoretical results. 
Some conclusions are made in Section \ref{sec8}.

\section{Electrostatics with multi-layered structures}\label{sec2}
\subsection{Multi-layered structures}
In this subsection, we first introduce the material parameters and geometric configurations for subsequent analysis.
The definition below encapsulates the concept of general multi-layered structures.
\begin{defn}\label{def21}
	Let $A$ be the inclusion with a $C^{1,\eta}\, (0<\eta<1)$ smooth boundary $\Gamma_1:=\p A$ and  $A_0=\RR^d\setminus\overline{A}$.
	We say that $A $  is considered to have a partition made up of a multi-layered structure, if the interior of $A$ be divided by means of closed and nonintersecting  $C^{1,\eta}$ surfaces $\Gamma_k$ $(k = 2, 3,...,N)$  into subsets (layers)  $A_k$ $(k = 1, 2,...,N)$.  Each $\Gamma_{k-1}$ surrounds $\Gamma_k$  $(k=2,3,\ldots,N)$.
	The region $A_k$ $(k=0,1,2,\ldots,N)$ stands for homogeneous media.
\end{defn}
For the $N$-layer structure $A$  as per Definition \ref{def21}, the material parameters are characterized by conductivity
\begin{equation}\label{conductvtA&AC}
\sigma(x)=\sigma_c(x)\chi(A)+\sigma_0\chi(\mathbb{R}^d\backslash \overline{A}),
\end{equation}
where $\chi$ stands for the characteristic function of a domain; $\sigma_0$ is set to be a positive constant and hence
it defines a regular material in the  background space $\mathbb{R}^d\backslash \overline{A}$. We set $\sigma_c$ to be of the following form
\begin{equation}\label{conductvt_each_layer}
\sigma_c(x)=\sigma_j, \quad x\in A_j,\quad j=1,2,\ldots,N.
\end{equation}
where
\[
\sigma_j\in\mathbb{C}\ \ \mbox{with}\ \ \Im\sigma_j\geqslant 0,\quad j=1,2,\ldots,N.
\]
That is, $\sigma_c$ is of a layered piecewise-constant structure. If $\Re\sigma_j>0$, then it defines a regular material in $A_j$ with  $\Im\sigma_j$ signifying the lossy/dampling parameters of the material. If $\Re\sigma_j\leqslant 0$, then it defines a metamaterial in $A_j$.  In this paper, we simply refer to $\sigma_j$ with $\Re\sigma_j\leqslant 0$ as a negative  material in $A_j$.

In summary, we consider a rather general multi-layered structure in which the number of layers can be arbitrarily given and each layer can consist of either regular or negative materials. Moreover, the material parameters in each layer can be different from one another.

\subsection{Electrostatics scattering and integral reformulation}
Associated with the $N$-layer structure $A$, which configuration is described by \eqref{conductvtA&AC}--\eqref{conductvt_each_layer}, the electrostatics scattering is governed by the following conductivity transmission problem:
\begin{equation}\label{cductvt_transmi_prblm}
\left\{
\begin{array}{ll}
\nabla\cdot (\sigma(x)\nabla u) =0, & \mbox{in} \quad \RR^d,\\
u-H=\Ocal(|x|^{1-d}), & \mbox{as}\quad |x|\rightarrow \infty,
\end{array}
\right.
\end{equation}
where  $d\geqslant 2$, the background  electrical potential ${H}$ is a harmonic function in $\RR^d$, and $u$ represents the total electric potential.
It is nature that the solution $u$ to \eqref{cductvt_transmi_prblm} satisfies the transmission conditions
\begin{equation}\label{transmi_cdtion}
u|_{+}=u|_{-}\quad\mbox{and}\quad\sigma_{k-1}\frac{\p u}{\p \nu_k}|_{+}=\sigma_{k}\frac{\p u}{\p \nu_k}|_{-}\quad\mbox{on}\quad \Gamma_k,\quad k=1,2,\ldots,N,
\end{equation}
where we used the notation $\nu_k$  to indicate the outward normal on $\Gamma_k$ and $$
\left.w\right|_{\pm}(x)=\lim _{h \rightarrow 0^{+}} w(x \pm h {\nu_k}),  \quad x \in \Gamma_k,
$$
for an arbitrary function $w$.

Let  $\Gamma $ be a $C^{1,\eta}$ surface. Let $H^{s}(\Gamma)$, for $s\in\RR$,  be the usual $L^2$-Sobolev space and let
$$H^{s}_0 (\Gamma):=\left\{\phi \in H^{s}(\Gamma) : \int_{\Gamma} \phi \mathrm{d}s=0\right\}.$$
For $s=0$, we use the notation $L^2_0(\Gamma)$.
Let  \[
G(x) =\left\{
\begin{array}{ll}
\frac{1}{2 \pi} \ln |x|, & d=2,\\
\frac{1}{(2-d) \omega_{d}}|x|^{2-d}, & d \geqslant 3,
\end{array}
\right.
\] be the outgoing fundamental solution to the Laplacian in $\RR^d$,
where $\omega_{d}$ is the area of the unit sphere in $\RR^d$.
The single and double layer potentials can then be defined by
\[
\Scal_{\Gamma}[\varphi](x):=\int_{\Gamma}G(x-y)\varphi(y)~\mathrm{d} s(y), \quad x\in \RR^d,
\]
\[
\mathcal{D}_{\Gamma}[\varphi](x):=\int_{\Gamma}\frac{\p}{\p \nu_y}G(x-y)\varphi(y)~\mathrm{d} s(y), \quad x\in \RR^d\setminus\Gamma,
\]
where $\varphi\in L^2(\Gamma)$ is the density function.
There hold the following jump relations on the surface \cite{Ammari2007}
\begin{equation} \label{singlejump}
\frac{\p}{\p\nu}\Scal_{\Gamma} [\varphi] \Big|_{\pm} = \(\pm \frac{1}{2}I+
\Kcal_{{\Gamma}}^*\)[\varphi] \quad \mbox{on } {\Gamma},
\end{equation}
\begin{equation}\label{doublejump}
\mathcal{D}_{{\Gamma}}[\varphi]\Big|_{\pm}=\(\mp \frac{1}{2}I+
\Kcal_{{\Gamma}} \)[\varphi] \quad \mbox{on } {\Gamma},
\end{equation}
where $\Kcal_{\Gamma}^*$ is the Neumann-Poincar\'e (NP) operator defined by
\[
\Kcal_\Gamma^*[\varphi]
(x) = \mbox{p.v.}\;\int_{{\Gamma}}\frac{\p G(x-y)}{\p \nu_x}\varphi(y)~\mathrm{d} s(y),
\]
and  $\Kcal_{\Gamma}$ is the adjoint operator of $\Kcal_{\Gamma}^*$, where p.v. stands for the Cauchy principle value.
We emphasize that $\Kcal_{\Gamma}^*$ is not self-adjoint in the usual inner product unless the domain is a circle or a sphere (see \cite{Lim_IJM_01}).

Define $\mathcal{H}^{s}=\prod _{k=1}^N  H^{s}(\Gamma_k)$ and $\mathcal{H}^{s}_0=\prod _{k=1}^N  H^{s}_0(\Gamma_k)$.
With the above preparations, one has the following integral representation of the conductivity transmission problem \eqref{conductvtA&AC}--\eqref{transmi_cdtion} with density function
%$\phi_k\in L^2_0(\Gamma_k)$ or
$\phi_k\in  H^{-\frac{1}{2}}_0(
\Gamma_k)$
\begin{equation}\label{potential_sol}
u(x)=H(x)+\sum_{k=1}^{N}\Scal_{\Gamma_k}[\phi_k](x).
\end{equation}
It follows from the second condition in \eqref{transmi_cdtion} that for all $k = 1,2,\ldots,N$,
\[
\sigma_{k-1}\left(\frac{\p H}{\p \nu_k}+\left.\frac{\p \Scal_{\Gamma_k}[\phi_k]}{\p \nu_k}\right|_{+}+\sum_{l\neq k}^{N}\frac{\p \Scal_{\Gamma_l}[\phi_l]}{\p \nu_k}\right)=\sigma_{k}\left(\frac{\p H}{\p \nu_k}+\left.\frac{\p \Scal_{\Gamma_k}[\phi_k]}{\p \nu_k}\right|_{-}+\sum_{l\neq k}^{N}\frac{\p \Scal_{\Gamma_l}[\phi_l]}{\p \nu_k}\right).
\]
%Note that we have used the notation $\nu_k$  to indicate the outward normal on $\Gamma_k$.
Using the jump formula \eqref{singlejump}, the above equations can be rewritten as
\begin{equation}\label{potential_matrix}
\begin{split}
\begin{bmatrix}
\lambda_{1}I - \Kcal_{\Gamma_1}^* & -\nu_1\cdot\nabla\Scal_{\Gamma_2} & \cdots & -\nu_1\cdot\nabla\Scal_{\Gamma_{N}} \\
-\nu_2\cdot\nabla\Scal_{\Gamma_1} & \lambda_{2}I-\Kcal_{\Gamma_2}^* & \cdots & -\nu_2\cdot\nabla\Scal_{\Gamma_{N}}\\
\vdots & \vdots &\ddots &\vdots \\
-\nu_{N}\cdot\nabla\Scal_{\Gamma_1} & -\nu_{N}\cdot\nabla\Scal_{\Gamma_2} & \cdots & \lambda_{N}I-\Kcal_{\Gamma_{N}}^*
\end{bmatrix}
\begin{bmatrix}
\phi_1  \\
\phi_2 \\
\vdots  \\
\phi_N
\end{bmatrix} = \begin{bmatrix}
\nu_1\cdot\nabla H  \\
\nu_2\cdot\nabla H \\
\vdots  \\
\nu_N\cdot\nabla H
\end{bmatrix},
\end{split}
\end{equation}
on $\mathcal{H}^{-1/2}_0$,
where
\begin{equation}\label{lamdk}
\lambda_k=\frac{\sigma_{k}+\sigma_{k-1}}{2(\sigma_{k}-\sigma_{k-1})},\quad k=1,2,\ldots,N.
\end{equation}
The problem \eqref{potential_sol} is solvable for $|\lambda_{k}|>\frac{1}{2}$ as shown in \cite[Theorem 2.1]{KDZIP24}, with the positivity assumption on the conductivities $\sigma_k$, $k=0,1,\ldots,N$.
However, in plasmon modes, it contains negative values in the multi-layered structure, that is $\sigma_k < 0$ holds for some $k = 1, 2, \ldots,N.$
This causes a serious difficulty in dealing with \eqref{potential_matrix}.

It is known that plasmon resonance is usually associated with some eigenvalue problem generated by the PDE system. We shall also explore the related eigenvalue problem for multi-layered structure. To simplify the analysis, we suppose that
\begin{equation}\label{eq:vepdef01}
\sigma_k
=\left\{
\begin{array}{ll}
\sigma_1, & k \quad \mbox{is odd},\\
\sigma_0, & k \quad \mbox{is even}.
\end{array}
\right.
\end{equation}
This is not a very restrictive assumption, and this case contains effects we are interested in.
In this setup, one can readily obtain that $\lambda_k = (-1)^{(k-1)}\lambda$, $k = 1,2,\ldots,N$, i.e.,
\[
\lambda_k=
\left\{
\begin{array}{ll}
\lambda, & k \quad \mbox{is odd},\\
-\lambda, &k \quad \mbox{is even},
\end{array}
\right.
\]
where
\begin{equation}\label{eq:deflamb01}
\lambda:=\frac{\sigma_{1}+\sigma_{0}}{2(\sigma_{1}-\sigma_{0})}.
\end{equation}
 We  apply the
operator
\begin{equation*}
\begin{bmatrix}
I & 0 & \cdots & 0 \\
0 & -I & \cdots & 0\\
\vdots & \vdots &\ddots &\vdots \\
0 & 0 & \cdots & (-1)^{N-1}I
\end{bmatrix} :\ \mathcal{H}^{-1/2}_0 \to \mathcal{H}^{-1/2}_0
\end{equation*}
to \eqref{potential_matrix}. Then we have
\begin{equation}\label{potential_matrix1}
\begin{split}
\begin{bmatrix}
\lambda I - \Kcal_{\Gamma_1}^* & -\nu_1\cdot\nabla\Scal_{\Gamma_2} & \cdots & -\nu_1\cdot\nabla\Scal_{\Gamma_{N}} \\
\nu_2\cdot\nabla\Scal_{\Gamma_1} & \lambda I+\Kcal_{\Gamma_2}^* & \cdots & \nu_2\cdot\nabla\Scal_{\Gamma_{N}}\\
\vdots & \vdots &\ddots &\vdots \\
(-1)^{N}\nu_{N}\cdot\nabla\Scal_{\Gamma_1} & (-1)^{N}\nu_{N}\cdot\nabla\Scal_{\Gamma_2} & \cdots & \lambda I+(-1)^{N}\Kcal_{\Gamma_{N}}^*
\end{bmatrix}
\begin{bmatrix}
\phi_1  \\
\phi_2 \\
\vdots  \\
\phi_N
\end{bmatrix} = \begin{bmatrix}
\nu_1\cdot\nabla H  \\
-\nu_2\cdot\nabla H \\
\vdots  \\
(-1)^{N-1}\nu_N\cdot\nabla H
\end{bmatrix},
\end{split}
\end{equation}
Let
$\mathbb{K}_N^*$ be an $N$-by-$N$ matrix type NP operator  on $\mathcal{H}^{-1/2}_0$ defined by
\begin{equation}\label{NPoperator}
\begin{split}
\mathbb{K}_N^*:=
\begin{bmatrix}
-\Kcal_{\Gamma_1}^* & -\nu_1\cdot\nabla\Scal_{\Gamma_2} & \cdots & -\nu_1\cdot\nabla\Scal_{\Gamma_{N}} \\
\nu_2\cdot\nabla\Scal_{\Gamma_1} & \Kcal_{\Gamma_2}^* & \cdots & \nu_2\cdot\nabla\Scal_{\Gamma_{N}}\\
\vdots & \vdots &\ddots &\vdots \\
(-1)^{N}\nu_{N}\cdot\nabla\Scal_{\Gamma_1} & (-1)^{N}\nu_{N}\cdot\nabla\Scal_{\Gamma_2} & \cdots & (-1)^{N}\Kcal_{\Gamma_{N}}^*
\end{bmatrix},
\end{split}
\end{equation}
and let  $\bm\phi := (\phi_1,\phi_2,\ldots,\phi_{N})^T$, $\bm{g}:=\left(\nu_1\cdot\nabla H,-\nu_2\cdot\nabla H,\ldots,(-1)^{N-1}\nu_N\cdot\nabla H\right)^T$.
Then, \eqref{potential_matrix} can be rewritten in the form
\begin{equation}\label{IRE}
\({\lambda}\,\mathbb{I}+\mathbb{K}_N^*\)\bm\phi = \bm{g},
\end{equation}
where $\mathbb{I}$ is given by
\[
\mathbb{I}:=
\begin{bmatrix}
I & 0 & \cdots & 0\\
0 & I & \cdots & 0\\
\vdots & \vdots &\ddots &\vdots \\
0 & 0 & \cdots & I
\end{bmatrix}.
\]
The spectral property of the operator $\mathbb{K}_N^*$ plays an very important role in solvability of the operator equation (\ref{IRE}), which is the purpose of our subsequent analysis.

\section{Properties of $\mathbb{K}_N^*$}\label{sec3}

In this section, we provide some properties of $\mathbb{K}_N^*$. More precisely, we shall analyze the adjoint operator $\mathbb{K}_N$ of $\mathbb{K}_N^*$, the spectrum of $\mathbb{K}_N^*$, and symmetrization of $\mathbb{K}_N^*$ under a certain twisted inner product by using Plemelj's symmetrization principle \cite{KPSARMA2007}.

%%%%%%%%%%%%%%%%%%%%%%%%%%%%%%%%%%%%%%%%%%%%%%%%%%%%
%\subsection{Adjoint operator of $\mathbb{K}_N^*$}
%%%%%%%%%%%%%%%%%%%%%%%%%%%%%%%%%%%%%%%%%%%%%%%%%%%%\newcommand{\vp}{\phi}
We first compute the adjoint of $\mathbb{K}_N^*$.
Denote by $\langle \cdot,\cdot \rangle_{L^2(\Gamma)}$ the Hermitian product on $L^{2}(\Gamma)$ with  $\Gamma = \Gamma_k$, for some $k=1,2,\ldots,N.$ By interchange orders of integration, it is easy to see that for $l\neq k $,
\[
\left\langle \pd{\Scal_{\Gamma_l}[\vp_l]}{\nu_k} , \vp_k
\right\rangle_{L^2(\Gamma_k)} = \left\langle \vp_l, \Dcal_{\Gamma_k}[\vp_k]
\right\rangle_{L^2(\Gamma_l)}.
\]
Thus the $L^2$-adjoint of $\mathbb{K}_N^*$, $\mathbb{K}_N$, is given by
\begin{equation} \label{eq:K*}
\mathbb{K}_N := \begin{bmatrix}
-\Kcal_{\Gamma_1} & \Dcal_{\Gamma_2} & \cdots & (-1)^{N}\Dcal_{\Gamma_{N}} \\
-\Dcal_{\Gamma_1} & \Kcal_{\Gamma_2} & \cdots & (-1)^{N}\Dcal_{\Gamma_{N}}\\
\vdots & \vdots &\ddots &\vdots \\
-\Dcal_{\Gamma_1} & \Dcal_{\Gamma_2} & \cdots & (-1)^{N}\Kcal_{\Gamma_{N}}
\end{bmatrix}.
\end{equation}
It is worth emphasizing that for $k = 1,2,\ldots,N$ and $l\neq k$, the operator $\Dcal_{\Gamma_k}$  in the $l$-th row of $\mathbb{K}_N$ is the one from $L^2(\Gamma_k)$ into $L^2(\Gamma_l)$.

%As in the single-layer or two-layer structures, $\mathbb{K}_N^*$ is not self-adjoint on $\mathcal{H}^{-1/2}_0$, but it can be symmetrized by introducing a new inner product on $\mathcal{H}^{-1/2}_0$.
To prove that $\mathbb{K}_N^*$ is symmetrizable, we shall make use of the following two Lemmas \ref{simplyconnectSD}--\ref{doublyconnectSD}  which can be proved by Green's formulas.

\begin{lem}\textup{\cite[Lemma 3.3]{Ammari2013}} \label{simplyconnectSD}
	Let $B \subset \RR^d$ be a bounded simply connected domain.
	\begin{itemize}
		\item[{\bf (i)}] If $u$ is a solution of $\Delta u =0$ in $B$, then
		\[
		\Scal_{\p B} \Big[\frac{\p u}{\p \nu} \Big{|}_- \Big] (x) = \Dcal_{\p B} \Big[u\big{|}_- \Big] (x),\quad x\in \RR^d \setminus
		\overline{B}.
		\]
				
		\item[{\bf (ii)}] If $u$ is a solution of
		\[
		\begin{cases}
		\Delta u = 0 \quad & \mbox{in } \RR^d \setminus \overline{B}, \\
		u(x) \to 0,  & |x| \to \infty,
		\end{cases}
		\]
		then
		\begin{equation*}
		\Scal_{\p B} \Big[\frac{\p u}{\p \nu} \big{|}_+ \Big] (x) = \Dcal_{\p B} \Big[u\big{|}_+ \Big] (x),\quad x\in B.
		\end{equation*}
	\end{itemize}
\end{lem}

\begin{lem}\label{doublyconnectSD}
	Let $B \subset \RR^d$ be a bounded doubly connected domain  lying between  the  interior boundary $S_i$ and the exterior boundary $S_e$.
	If $u$ is a solution of $\Delta u =0$ in $B$, then
	\[
	\Scal_{S_e} \Big[\frac{\p u}{\p \nu} \Big{|}_- \Big] (x)-\Scal_{S_i} \Big[\frac{\p u}{\p \nu} \Big{|}_+ \Big] (x)  = \Dcal_{S_e} \Big[u\big{|}_- \Big] (x)-\Dcal_{S_i} \Big[u\big{|}_+ \Big] (x),\quad x\in \RR^d \setminus
	\overline{B}.
	\]
\end{lem}
\begin{rem}
	If we take $u=\Scal_{\p B}[\vp]$ in Lemma \ref{simplyconnectSD} {\bf(i)}, we have
	\begin{equation*}
	- \frac{1}{2} \Scal_{\p B}[ \vp ] (x) +  \Scal_{\p B} \Kcal_{\p
		B}^*[\vp](x) = \Dcal_{\p B}  \Scal_{\p B} [\vp] (x), \quad x \in
	\RR^d \setminus \overline{B}.
	\end{equation*}
	By taking the limit as $x \to \p B$ from outside $B$ and using the jump relation \eqref{doublejump}, we obtain the well-known Calder\'on's identity (also known as Plemelj's symmetrization principle)
	\begin{equation} \label{SK*=KS_SL}
	\Scal_{\p B} \Kcal_{\p B}^* = \Kcal_{\p B} \Scal_{\p B}.
	\end{equation}
	For a two-layer structure, it is sufficient to use Lemma \ref{simplyconnectSD} to get the Calder\'on-type identity. We refer to \cite{Ammari2013} for more detail. In order to obtain Calder\'on-type identity for structures of layers $\geqslant 3$, the Lemma \ref{doublyconnectSD} becomes essential.
\end{rem}

Define
\begin{equation} \label{eq:S}
\mathbb{S}_N := \begin{bmatrix}
\Scal_{\Gamma_1} & \Scal_{\Gamma_2} & \cdots & \Scal_{\Gamma_{N}} \\
\Scal_{\Gamma_1} & \Scal_{\Gamma_2} & \cdots & \Scal_{\Gamma_{N}}\\
\vdots & \vdots &\ddots &\vdots \\
\Scal_{\Gamma_1} & \Scal_{\Gamma_2} & \cdots & \Scal_{\Gamma_{N}}
\end{bmatrix}.
\end{equation}
Again we emphasize that for $k = 1,2,\ldots,N$ and $l\neq k$, the operator $\Scal_{\Gamma_k}$ in the $l$-th row of $\mathbb{S}_N$ is the one from $L^2(\Gamma_k)$  into $L^2(\Gamma_l)$.
In the following, we give the main result in this section. %symmetrization of $\mathbb{K}_N^*$.

\begin{prop}\label{NPspectral}
	For $\bm{\phi}, \bm{\psi}\in \mathcal{H}^{-\frac{1}{2}}_0$, let $\mathcal{H}$ be the space $\mathcal{H}^{-\frac{1}{2}}_0$ equipped with the inner product defined by
	\begin{equation}\label{newip}
	\langle\bm{\phi}, \bm{\psi} \rangle_{\mathcal{H}} :=  \langle \bm\phi, -\mathbb{S}_N [\bm\psi] \rangle,
	\end{equation}
	where $ \langle\cdot,\cdot \rangle$ is the pair between $\mathcal{H}^{-\frac{1}{2}}$ and $\mathcal{H}^{\frac{1}{2}}$ pairing. Then we have
	\begin{itemize}
	\item[\bf(i)] $\langle\bm{\phi}, \bm{\psi} \rangle_{\mathcal{H}}$ is an inner product on $\mathcal{H}^{-\frac{1}{2}}_0$.
	\item[\bf(ii)] The  Calder\'on-type identity  holds:
	\begin{equation}\label{calderonN}
	\mathbb{S}_N\mathbb{K}_N^* = \mathbb{K}_N \mathbb{S}_N,
	\end{equation}
	i.e.,  $\mathbb{S}_N\mathbb{K}_N^*$ is self-adjoint.
	\item[\bf(iii)] The spectrum of $\mathbb{K}_N^*$ on $\mathcal{H}^{-\frac{1}{2}}$ lies in the interval $[-1/2, 1/2]$.
	\end{itemize}
\end{prop}
\begin{rem}
	Then by using the symmetrization principle \eqref{calderonN}, one can find that
	\[
	\langle\mathbb{K}_N^*[\bm{\phi}], \bm{\psi} \rangle_{\mathcal{H}} = \langle \mathbb{K}_N^*[\bm{\phi}], -\mathbb{S}_N [\bm\psi] \rangle = \langle \bm{\phi}, -\mathbb{K}_N\mathbb{S}_N [\bm\psi] \rangle = \langle \bm{\phi}, -\mathbb{S}_N\mathbb{K}_N^* [\bm\psi] \rangle = \langle \bm{\phi}, \mathbb{K}_N^* [\bm\psi] \rangle_{\mathcal{H}},
	\]
	which show that $\mathbb{K}_N^*$ is self-adjoint on $\mathcal{H}$.
	 Moreover, the the norm $\|\cdot\|_{\mathcal{H}}$ induced by the inner product \eqref{newip} is equivalent to the $\mathcal{H}^{-\frac{1}{2}}_0$ norm. We refer to \cite{KKLSY_JLMS2016} for more detail.
	 Since the interfaces $\Gamma_k$, $k =1,2,\ldots,N$, are of class $C^{1,\eta}$ for $0<\eta<1$, the NP operator $\mathbb{K}_N^*$ is compact as well as self-adjoint on $\mathcal{H}$.
	 Hence the  spectrum $\sigma(\mathbb{K}_N^*)$ is real and has only point spectrum accumulating to 0.
\end{rem}
\begin{proof}[\bf Proof of Proposition \ref{NPspectral}]
	$\bf(i)$ This can be proved similarly by following the proof  on two-layer structures in \cite{Ammari2013,EBFTARMA13},  so we omit it here.

	$\bf(ii)$ 	First by direct computations one has
	\begin{equation*}
	\mathbb{S}_N \mathbb{K}_N^* = \begin{bmatrix}
	-\Scal_{\Gamma_1}\Kcal^*_{\Gamma_1} +\sum\limits_{l\neq1}\limits^{N}(-1)^l\Scal_{\Gamma_l}\pd{\Scal_{\Gamma_1}}{\nu_l} & \Scal_{\Gamma_2}\Kcal^*_{\Gamma_2} +\sum\limits_{l\neq 2}\limits^{N}(-1)^l\Scal_{\Gamma_l}\pd{\Scal_{\Gamma_2}}{\nu_l} & \cdots & (-1)^N\Scal_{\Gamma_N}\Kcal^*_{\Gamma_N} +\sum\limits_{l=1}\limits^{N-1}(-1)^l\Scal_{\Gamma_l}\pd{\Scal_{\Gamma_N}}{\nu_l} \\
	-\Scal_{\Gamma_1}\Kcal^*_{\Gamma_1} +\sum\limits_{l\neq1}\limits^{N}(-1)^l\Scal_{\Gamma_l}\pd{\Scal_{\Gamma_1}}{\nu_l} & \Scal_{\Gamma_2}\Kcal^*_{\Gamma_2} +\sum\limits_{l\neq 2}\limits^{N}(-1)^l\Scal_{\Gamma_l}\pd{\Scal_{\Gamma_2}}{\nu_l} & \cdots & (-1)^N\Scal_{\Gamma_N}\Kcal^*_{\Gamma_N} +\sum\limits_{l=1}\limits^{N-1}(-1)^l\Scal_{\Gamma_l}\pd{\Scal_{\Gamma_N}}{\nu_l} \\
	\vdots & \vdots &\ddots &\vdots \\
	-\Scal_{\Gamma_1}\Kcal^*_{\Gamma_1} +\sum\limits_{l\neq1}\limits^{N}(-1)^l\Scal_{\Gamma_l}\pd{\Scal_{\Gamma_1}}{\nu_l} & \Scal_{\Gamma_2}\Kcal^*_{\Gamma_2} +\sum\limits_{l\neq 2}\limits^{N}(-1)^l\Scal_{\Gamma_l}\pd{\Scal_{\Gamma_2}}{\nu_l} & \cdots & (-1)^N\Scal_{\Gamma_N}\Kcal^*_{\Gamma_N} +\sum\limits_{l=1}\limits^{N-1}(-1)^l\Scal_{\Gamma_l}\pd{\Scal_{\Gamma_N}}{\nu_l}
	\end{bmatrix},
	\end{equation*}
	and
	\[
	\small{\begin{aligned}
	\mathbb{K}_N \mathbb{S}_N=
	  \begin{bmatrix}
	-\Kcal_{\Gamma_1} \Scal_{\Gamma_1} +\sum\limits_{l\neq 1}\limits^{N} (-1)^l\Dcal_{\Gamma_l} \Scal_{\Gamma_1} & -\Kcal_{\Gamma_1} \Scal_{\Gamma_2} +\sum\limits_{l\neq 1}\limits^{N} (-1)^l\Dcal_{\Gamma_l} \Scal_{\Gamma_2} & \cdots & -\Kcal_{\Gamma_1} \Scal_{\Gamma_N} +\sum\limits_{l\neq 1}\limits^{N} (-1)^l\Dcal_{\Gamma_l} \Scal_{\Gamma_N} \\
	\Kcal_{\Gamma_2} \Scal_{\Gamma_1} +\sum\limits_{l\neq 2}\limits^{N} (-1)^l\Dcal_{\Gamma_l} \Scal_{\Gamma_1} & \Kcal_{\Gamma_2} \Scal_{\Gamma_2} +\sum\limits_{l\neq 2}\limits^{N} (-1)^l\Dcal_{\Gamma_l} \Scal_{\Gamma_2} & \cdots & \Kcal_{\Gamma_2} \Scal_{\Gamma_N} +\sum\limits_{l\neq 2}\limits^{N} (-1)^l\Dcal_{\Gamma_l} \Scal_{\Gamma_N}  \\
	\vdots & \vdots &\ddots &\vdots \\
	(-1)^N\Kcal_{\Gamma_N} \Scal_{\Gamma_1} +\sum\limits_{l=1}\limits^{N-1} (-1)^l\Dcal_{\Gamma_l} \Scal_{\Gamma_1} & (-1)^N\Kcal_{\Gamma_N} \Scal_{\Gamma_2} +\sum\limits_{l=1}\limits^{N-1} (-1)^l\Dcal_{\Gamma_l} \Scal_{\Gamma_2} & \cdots & (-1)^N\Kcal_{\Gamma_N} \Scal_{\Gamma_N} +\sum\limits_{l=1}\limits^{N-1} (-1)^l\Dcal_{\Gamma_l} \Scal_{\Gamma_N}
	\end{bmatrix}.
	\end{aligned}}
	\]
	We now check the on-diagonal elements, the above-diagonal elements, and the below-diagonal elements, respectively.
	
\textbf{Case I} $(\mathbb{S}_N\mathbb{K}_N^*)_{mm} = (\mathbb{K}_N \mathbb{S}_N)_{mm}$ for $m=1,2,\ldots,N$. By \eqref{SK*=KS_SL} it follows that
	$\Scal_{\Gamma_m} \Kcal_{\Gamma_m}^* = \Kcal_{\Gamma_m}
	\Scal_{\Gamma_m}$ on $\Gamma_m$, $m=1,2,\ldots,N$.
	Thus, we have only to prove that for $m = 1,2,\ldots,N$ and $l\neq m$,
	\begin{equation}\label{equ315}
	\Scal_{\Gamma_l}\pd{\Scal_{\Gamma_m}}{\nu_l} [\vp_m] = \Dcal_{\Gamma_l} \Scal_{\Gamma_m} [\vp_m] \quad\mbox{on } \Gamma_m.
	\end{equation}
	If we set $u(x) = \Scal_{\Gamma_m}[\vp_m](x)$
	and $B=\cup_{k = l}^{N} A_k
	$ in Lemma \ref{simplyconnectSD}, we can deduce that for the case $l<m$ %(r_l>r_m)
	of \eqref{equ315} follows from Lemma \ref{simplyconnectSD} (ii), and the case $l>m$  of that follows from Lemma \ref{simplyconnectSD} (i), respectively.
	Hence, 	 $(\mathbb{S}_N\mathbb{K}_N^*)_{mm} = (\mathbb{K}_N \mathbb{S}_N)_{mm}$ holds for all $m=1,2,\ldots,N$.
	%%%%%%%%%%%%%%%%%%%%%%%%%%%%%%%%%%%%%%%
	%%%%%%%%%%%%%%%%%%%%%%%%%%%%%%%%%%%%%%%%
	
	\textbf{Case II} $(\mathbb{S}_N\mathbb{K}_N^*)_{mn} = (\mathbb{K}_N \mathbb{S}_N)_{mn}$ for $m,n=1,2,\ldots,N$ and $m<n$. 	We want to make sure that the following holds:
	\begin{equation}\label{equ316}
	(-1)^n\Scal_{\Gamma_n}\Kcal^*_{\Gamma_n}[\vp_n] +\sum\limits_{l\neq n}\limits^{N}(-1)^l\Scal_{\Gamma_l}\pd{}{\nu_l}\Scal_{\Gamma_n}[\vp_n] = (-1)^m\Kcal_{\Gamma_m} \Scal_{\Gamma_n}[\vp_n] +\sum\limits_{l\neq m}\limits^{N} (-1)^l\Dcal_{\Gamma_l} \Scal_{\Gamma_n}[\vp_n],\quad \mbox{on } \Gamma_m.
	\end{equation}
	Next we divide $l=1,2,\ldots,N$ into three cases.
	
	\textbf{Case II.i} $1\leqslant l\leqslant m$.
    We set $u=\Scal_{\Gamma_n}[\vp_n]$ and $B=\cup_{k = l}^{N} A_k$. By using Lemma \ref{simplyconnectSD} (ii), we have
	\[
	\Scal_{\Gamma_l} \Big[\frac{\p \Scal_{\Gamma_n} [\vp_n]}{\p \nu_l}  \Big] (x) = \Dcal_{\Gamma_l} \Scal_{\Gamma_n} [\vp_n](x) \quad \textmd{for } x \in {B}.
	\]
	Thus for $l<m$, we have
	\begin{equation}\label{l<m<n}
	\Scal_{\Gamma_l} \Big[\frac{\p \Scal_{\Gamma_n} [\vp_n]}{\p \nu_l}  \Big] = \Dcal_{\Gamma_l} \Scal_{\Gamma_n} [\vp_n] \quad\mbox{on } \Gamma_m,
	\end{equation}
	and for  $l = m$, by taking the limit as $x \to \Gamma_m|_-$, we further have
	\begin{equation}\label{l=m<n}
	\Scal_{\Gamma_m} \dfrac{\p}{\p\nu_m}\Scal_{\Gamma_n} [\vp_n] = \frac{1}{2} \Scal_{\Gamma_n} [\vp_n] +  \Kcal_{\Gamma_m} \Scal_{\Gamma_n} [\vp_n],\quad \mbox{on } \Gamma_m.
	\end{equation}

	%%%%%%%%%%%%%%%%%%%%%%%%%%%%5
	%%%%%%%%%%%%%%%%%%%%%%%%%%%%%
	\textbf{Case II.ii} $m<l<n$. We  set $u=\Scal_{\Gamma_n}[\vp_n]$ and $B=\cup_{k = l}^{n} A_k$. By using Lemma \ref{doublyconnectSD}, we have
	\[
	\begin{aligned}
	&\quad \Scal_{\Gamma_l} \Big[\frac{\p\Scal_{\Gamma_n}[\vp_n]}{\p \nu_l}  \Big] (x)-\Scal_{\Gamma_n} \Big[\frac{\p\Scal_{\Gamma_n}[\vp_n]}{\p \nu_n} \Big{|}_+ \Big] (x)\\&  = \Dcal_{\Gamma_l} \Big[\Scal_{\Gamma_n}[\vp_n] \Big] (x)-\Dcal_{\Gamma_n} \Big[\Scal_{\Gamma_n}[\vp_n]\Big] (x),\quad x\in \RR^d \setminus
	\overline{B},
	\end{aligned}
	\]
	and thus we get
	\begin{equation}\label{m<l<n}
	\begin{aligned}
	&\quad \Scal_{\Gamma_l} \frac{\p \Scal_{\Gamma_n}[\vp_n]}{\p \nu_l}   -\frac{1}{2}\Scal_{\Gamma_n} [\vp_n] - \Scal_{\Gamma_n} \Kcal_{\Gamma_n}^* [\vp_n]\\&  = \Dcal_{\Gamma_l} \Scal_{\Gamma_n}[\vp_n]  -\Dcal_{\Gamma_n} \Scal_{\Gamma_n}[\vp_n] ,\quad \mbox{on } \Gamma_m.
	\end{aligned}
	\end{equation}

	\textbf{Case II.iii} $l\geqslant n$. We set $u=\Scal_{\Gamma_n}[\vp_n]$ and $B=\cup_{k = l}^{N} A_k$. By using Lemma \ref{simplyconnectSD} (i), we have
	\[
	\Scal_{\Gamma_l} \Big[\frac{\p \Scal_{\Gamma_n} [\vp_n]}{\p \nu_l}\big{|}_-  \Big] (x) = \Dcal_{\Gamma_l} \Scal_{\Gamma_n} [\vp_n](x) \quad \textmd{for } \quad x\in \RR^d \setminus
	\overline{B},
	\]
	Thus, we have for $l=n$,
	\begin{equation} \label{m<l=n}
	-\frac{1}{2}\Scal_{\Gamma_n} [\vp_n] + \Scal_{\Gamma_n} \Kcal_{\Gamma_n}^* [\vp_n] = \Dcal_{\Gamma_n} \Scal_{\Gamma_n} [\vp_n] \quad \textmd{on } \Gamma_m,
	\end{equation}
	and for $l>n$,
	\begin{equation}\label{l>n>m}
	\Scal_{\Gamma_l} \frac{\p \Scal_{\Gamma_n} [\vp_n]}{\p \nu_l} = \Dcal_{\Gamma_l} \Scal_{\Gamma_n} [\vp_n]  \quad \textmd{on } \Gamma_m.
	\end{equation}

	From \eqref{l<m<n} and \eqref{l>n>m},  it suffices to
	prove that
	\begin{equation}\label{equ321}
	(-1)^n\Scal_{\Gamma_n}\Kcal^*_{\Gamma_n}[\vp_n] +\sum\limits_{m\leqslant l< n}(-1)^l\Scal_{\Gamma_l}\pd{}{\nu_l}\Scal_{\Gamma_n}[\vp_n] = (-1)^m\Kcal_{\Gamma_m} \Scal_{\Gamma_n}[\vp_n] +\sum\limits_{m< l\leqslant n}\limits^{N} (-1)^l\Dcal_{\Gamma_l} \Scal_{\Gamma_n}[\vp_n],\quad \mbox{on } \Gamma_m.
	\end{equation} 	
	%Let $A_L$ and $A_R$ denote, respectively, the LHS and RHS terms in \eqref{equ321}.
	In view of \eqref{l=m<n}, \eqref{m<l<n} and  \eqref{m<l=n}, we can obtain that
	\[
	\begin{aligned}
	&\quad (-1)^n\Scal_{\Gamma_n}\Kcal^*_{\Gamma_n}[\vp_n] +\sum\limits_{m\leqslant l< n}(-1)^l\Scal_{\Gamma_l}\pd{}{\nu_l}\Scal_{\Gamma_n}[\vp_n]\\
	& = (-1)^n\(\frac{1}{2}\Scal_{\Gamma_n} [\vp_n] + \Dcal_{\Gamma_n} \Scal_{\Gamma_n} [\vp_n]\) + (-1)^m\(\frac{1}{2} \Scal_{\Gamma_n} [\vp_n] +  \Kcal_{\Gamma_m} \Scal_{\Gamma_n} [\vp_n]\)\\
	& \quad + \sum\limits_{m< l< n}(-1)^l\(\Dcal_{\Gamma_l} \Scal_{\Gamma_n}[\vp_n]  -\Dcal_{\Gamma_n} \Scal_{\Gamma_n}[\vp_n] + \frac{1}{2}\Scal_{\Gamma_n} [\vp_n] + \Scal_{\Gamma_n} \Kcal_{\Gamma_n}^* [\vp_n]\)\\
	& = (-1)^m\Kcal_{\Gamma_m} \Scal_{\Gamma_n} [\vp_n] +  \sum\limits_{m< l\leqslant  n}(-1)^l\Dcal_{\Gamma_l} \Scal_{\Gamma_n}[\vp_n]   +\(\frac{(-1)^n}{2}+ \frac{(-1)^m}{2} + \sum\limits_{m< l< n}(-1)^l\)\Scal_{\Gamma_n} [\vp_n] \\
	& = (-1)^m\Kcal_{\Gamma_m} \Scal_{\Gamma_n} [\vp_n] +  \sum\limits_{m< l\leqslant  n}(-1)^l\Dcal_{\Gamma_l} \Scal_{\Gamma_n}[\vp_n],
	\end{aligned}
	\]
	where the last equality follows from $\sum\limits_{m< l< n}(-1)^l = -\dfrac{(-1)^m+(-1)^n}{2}$.
	Hence,  $(\mathbb{S}_N\mathbb{K}_N^*)_{mn} = (\mathbb{K}_N \mathbb{S}_N)_{mn}$ for $m,n=1,2,\ldots,N$ and $m<n$.

	\textbf{Case III} $(\mathbb{S}_N\mathbb{K}_N^*)_{mn} = (\mathbb{K}_N \mathbb{S}_N)_{mn}$ for $m,n=1,2,\ldots,N$ and $m>n$.
	We also want to make sure that the equality \eqref{equ316} holds. 	Similarly, we divide $l=1,2,\ldots,N$ into three cases.

	\textbf{Case III.i} $l\leqslant n$. We set $u=\Scal_{\Gamma_n}[\vp_n]$ and $B=\cup_{k = l}^{N} A_k$. By using Lemma \ref{simplyconnectSD} (ii), we have
	\[
	\Scal_{\Gamma_l} \Big[\frac{\p \Scal_{\Gamma_n} [\vp_n]}{\p \nu_l}\big{|}_+  \Big] (x) = \Dcal_{\Gamma_l} \Scal_{\Gamma_n} [\vp_n](x) \quad \textmd{for } x \in {B}.
	\]
	Thus for $l<n$, we have
	\begin{equation}\label{l<n<m}
	\Scal_{\Gamma_l} \Big[\frac{\p \Scal_{\Gamma_n} [\vp_n]}{\p \nu_l}  \Big] = \Dcal_{\Gamma_l} \Scal_{\Gamma_n} [\vp_n] \quad\mbox{on } \Gamma_m,
	\end{equation}
	and for  $l = n$,  we further have
	\begin{equation}\label{l=n<m}
	\frac{1}{2} \Scal_{\Gamma_n} [\vp_n] +  \Scal_{\Gamma_n} \Kcal^*_{\Gamma_n} [\vp_n] = \Dcal_{\Gamma_n} \Scal_{\Gamma_n} [\vp_n],\quad \mbox{on } \Gamma_m.
	\end{equation}
	
	\textbf{Case III.ii} $n<l<m$. We  set $u=\Scal_{\Gamma_n}[\vp_n]$ and $B=\cup_{k = n}^{l} A_k$. By using Lemma \ref{doublyconnectSD}, we have
	\[
	\begin{aligned}
	&\quad \Scal_{\Gamma_n} \Big[\frac{\p \Scal_{\Gamma_n}[\vp_n]}{\p \nu_n} \Big{|}_- \Big] (x)-\Scal_{\Gamma_l} \Big[\frac{\p \Scal_{\Gamma_n}[\vp_n]}{\p \nu_l}  \Big] (x)\\&  = \Dcal_{\Gamma_n} \Big[\Scal_{\Gamma_n}[\vp_n]\big{|}_- \Big] (x)-\Dcal_{\Gamma_l} \Big[\Scal_{\Gamma_n}[\vp_n]\big{|}_+ \Big] (x),\quad x\in \RR^d \setminus
	\overline{B},
	\end{aligned}
	\]
	and thus we get
	\begin{equation}\label{n<l<m}
	\begin{aligned}
	&\quad-\frac{1}{2}\Scal_{\Gamma_n} [\vp_n] + \Scal_{\Gamma_n} \Kcal_{\Gamma_n}^* [\vp_n] - \Scal_{\Gamma_l} \frac{\p \Scal_{\Gamma_n}[\vp_n]}{\p \nu_l}   \\&  =\Dcal_{\Gamma_n} \Scal_{\Gamma_n}[\vp_n] -\Dcal_{\Gamma_l} \Scal_{\Gamma_n}[\vp_n]   ,\quad \mbox{on } \Gamma_m.
	\end{aligned}
	\end{equation}
	
	\textbf{Case III.iii} $l\geqslant m$. We set $u=\Scal_{\Gamma_n}[\vp_n]$ and $B=\cup_{k = l}^{N} A_k$. By using Lemma \ref{simplyconnectSD} (i), we have
	\[
	\Scal_{\Gamma_l} \Big[\frac{\p \Scal_{\Gamma_n} [\vp_n]}{\p \nu_l}  \Big] (x) = \Dcal_{\Gamma_l} \Scal_{\Gamma_n} [\vp_n](x) \quad \textmd{for } \quad x\in \RR^d \setminus
	\overline{B},
	\]
	Thus, for $l=m$,  by taking the limit as $x \to \Gamma_m|_+$, we further have
	\begin{equation} \label{n<l=m}
	\Scal_{\Gamma_m} \Big[\frac{\p \Scal_{\Gamma_n} [\vp_n]}{\p \nu_m}  \Big] = -\frac{1}{2}\Scal_{\Gamma_n} [\vp_n] + \Kcal_{\Gamma_m}\Scal_{\Gamma_n}  [\vp_n]  \quad \textmd{on } \Gamma_m,
	\end{equation}
	and for $l>m$,
	\begin{equation}\label{l>m>n}
	\Scal_{\Gamma_l} \frac{\p \Scal_{\Gamma_n} [\vp_n]}{\p \nu_l} = \Dcal_{\Gamma_l} \Scal_{\Gamma_n} [\vp_n](x)  \quad \textmd{on } \Gamma_m.
	\end{equation}
	Therefore, from \eqref{l<n<m}--\eqref{l>m>n}, we get for $m,n=1,2,\ldots,N$ and $m>n$,
	\[
	\begin{aligned}
	(\mathbb{S}_N\mathbb{K}_N^*)_{mm} &= 	(-1)^n\Scal_{\Gamma_n}\Kcal^*_{\Gamma_n}[\vp_n] +\sum\limits_{l\neq n}\limits^{N}(-1)^l\Scal_{\Gamma_l}\pd{}{\nu_l}\Scal_{\Gamma_n}[\vp_n]\\
	&  =\sum\limits_{l< n}(-1)^l\Dcal_{\Gamma_l} \Scal_{\Gamma_n} [\vp_n]  + (-1)^n\(-\frac{1}{2}\Scal_{\Gamma_n} [\vp_n] + \Dcal_{\Gamma_n} \Scal_{\Gamma_n} [\vp_n]\)\\
	&\quad  + \sum\limits_{ n<l<m}\limits^{N}(-1)^l\(-\frac{1}{2}\Scal_{\Gamma_n} [\vp_n] + \Scal_{\Gamma_n} \Kcal_{\Gamma_n}^* [\vp_n] - \Dcal_{\Gamma_n} \Scal_{\Gamma_n}[\vp_n] +\Dcal_{\Gamma_l} \Scal_{\Gamma_n}[\vp_n]\)\\
	&\quad  +
	(-1)^m\(-\frac{1}{2} \Scal_{\Gamma_n} [\vp_n] +  \Kcal_{\Gamma_m} \Scal_{\Gamma_n} [\vp_n]\) + \sum\limits_{ l>m}(-1)^l\Dcal_{\Gamma_l} \Scal_{\Gamma_n} [\vp_n]\\
	& = (-1)^m\Kcal_{\Gamma_m} \Scal_{\Gamma_n} [\vp_n] +  \sum\limits_{ l\neq m}(-1)^l\Dcal_{\Gamma_l} \Scal_{\Gamma_n}[\vp_n]   -\(\frac{(-1)^n}{2}+ \frac{(-1)^m}{2} + \sum\limits_{n< l< m}(-1)^l\)\Scal_{\Gamma_n} [\vp_n] \\
	& = (-1)^m\Kcal_{\Gamma_m} \Scal_{\Gamma_n} [\vp_n] +  \sum\limits_{ l\neq m}(-1)^l\Dcal_{\Gamma_l} \Scal_{\Gamma_n}[\vp_n] \\
	& = (\mathbb{K}_N \mathbb{S}_N)_{mn}.
	\end{aligned}
	\]
	The proof of $\bf(ii)$ is complete.
	
	$\bf(iii)$ The proof is essentially similar to the proof of \cite[Lemma 2.1]{KDZIP24}, we omit it here.	
	
	The proof is complete.
\end{proof}

\section{Explicit formulae for the integral equation in multi-layered confocal ellipses}\label{sec4}
%In order to give the explicit formulae of the eigenvalues of the NP operator $\mathbb{K}_N^*$,
In this section, we explicitly compute the solution $\phi_k$ of the integral equation \eqref{potential_matrix} on the non-radial case, more specifically, we consider  the conductivity transmission problem \eqref{conductvtA&AC}--\eqref{transmi_cdtion} with $A$ being a multi-layered confocal ellipse. We introduce the elliptic coordinates $(\xi, \eta)$ so that $x=(x_1,x_2)$ in Cartesian coordinates are defined by
\begin{align}\label{ellipcoor}
x_1=R \cosh \xi  \cos \eta, \quad x_2=R \sinh \xi  \sin \eta,\quad \xi >   0, \quad 0\leqslant   \eta \leqslant   2\pi,
\end{align}
where $2R$ is the focal distance. Precisely, we give a sequence of layers by
\begin{equation}\label{eq:aj}
A_{0}:=\{(\xi, \eta): \xi>\xi_{1}\},\; A_k:=\{(\xi, \eta):\xi_{k+1}<\xi\leqslant   \xi_{k}\},\;  k=1,2,\ldots, N-1,\; A_{N}:=\{(\xi, \eta):\xi\leqslant   \xi_{N}\},
\end{equation}
and the interfaces between the adjacent layers can be rewritten by
\begin{equation}\label{interface}
\Gamma_k:=\left\{(\xi, \eta) : \xi = \xi_k\right\}, \quad  k=1,2,\ldots,N,
\end{equation}
where $N\in \mathbb{N}$. We assume $\Gamma_k,$ $k=1,2,\ldots,N$, are confocal ellipses whose common foci are denoted
by $(\pm R, 0)$.
So the confocal ellipses have the same elliptic coordinates.
The illustration of $N$-layer confocal ellipses is shown in Figure \ref{fig:1}.
\begin{figure}[!h]
	\begin{center}
		{\includegraphics[width=2.5in]{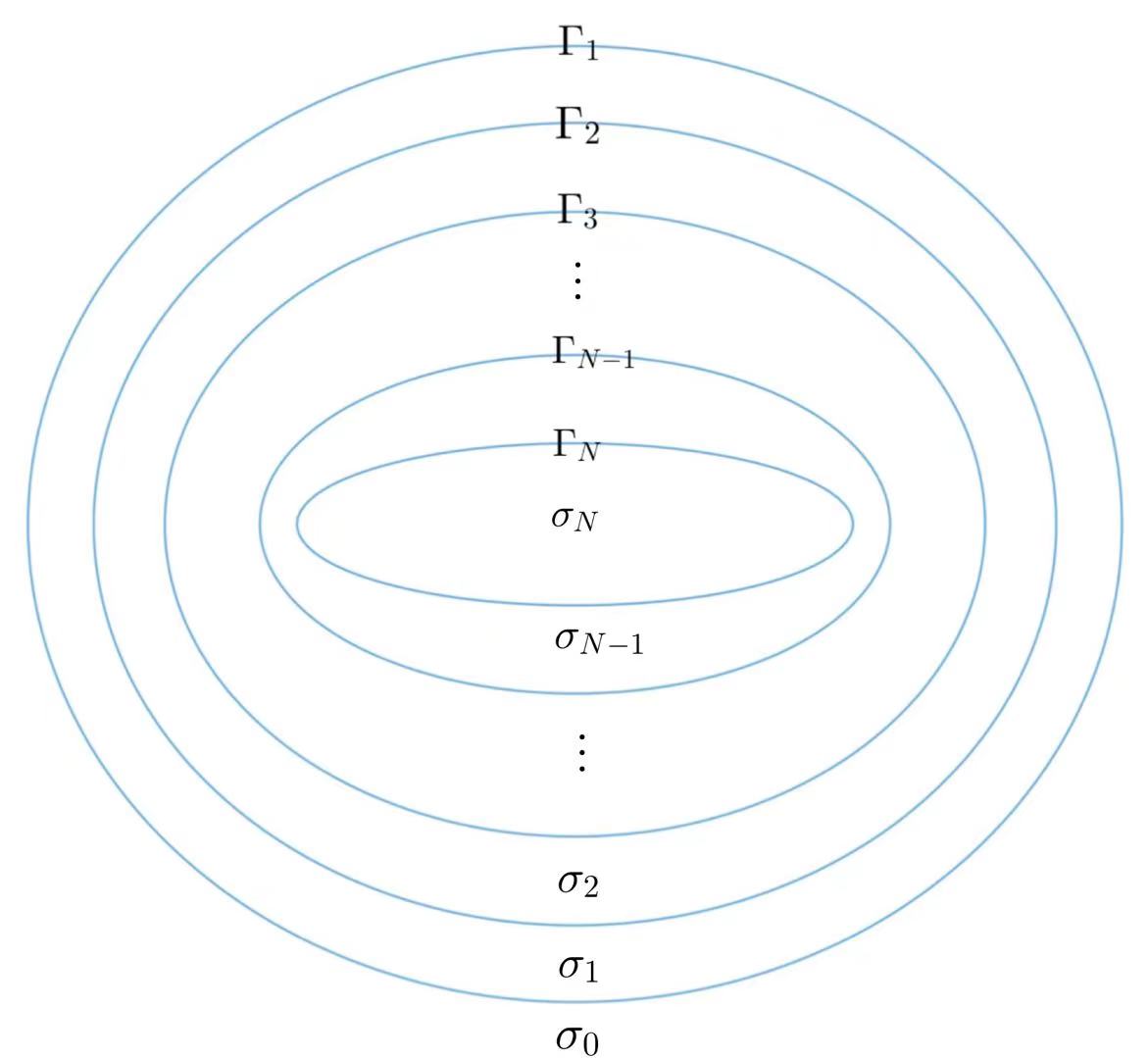}}
	\end{center}
	\caption{Schematic illustration of an $N$-layer confocal ellipse.
	}\label{fig:1}
\end{figure}

For any positive elliptic radius $\xi_0$, the set $\Gamma_0 = \{ (\xi, \eta) : \xi = \xi_0\}$ is an ellipse whose foci are $(\pm R, 0)$. One can see easily that the length element $\mathrm{d} s$ and the outward normal derivative $\frac{\p}{\p \nu}$ on $\Gamma_0$ are given in terms of the elliptic coordinates by
\begin{equation*}
\mathrm{d} s = \gamma \mathrm{d} \eta \quad \mbox{and} \quad \frac{\p}{\p \nu} = \gamma^{-1}\frac{\p}{\p \xi},
\end{equation*}
where
\begin{equation*}
\gamma =  \gamma (\xi_0, \eta) = R \sqrt{\sinh^2\xi_0+\sin^2\eta}.
\end{equation*}
To proceed, it is convenient to use the following notations: for  $n=1,2,\dots$,
\begin{align*}
\beta_n^{c,0} := \gamma (\xi_0, \eta)^{-1} \cos (n\eta) \quad \mbox{and} \quad \beta_n^{s,0} := \gamma (\xi_0, \eta)^{-1} \sin (n\eta).
\end{align*}
For a nonnegative integer $n$, it is proved in \cite{Chung2014, Ando2016} that
\begin{align}\label{S-ellipse-cos}
\mathcal{S}_{\Gamma_0}[ \beta_n^{c,0}](x) = \begin{cases}
\ds -\frac{\cosh (n\xi)}{n\e^{n\xi_0}}\cos (n\eta),\quad & \xi  \leqslant
 \xi_0,\\
\nm
\ds -\frac{\cosh (n\xi_0)}{n\e^{n\xi}}\cos (n\eta),\quad & \xi  > \xi_0,
\end{cases}
\end{align}
and
\begin{align}\label{S-ellipse-sin}
\mathcal{S}_{\Gamma_0}[\beta_n^{s,0}](x) = \begin{cases}
\ds -\frac{\sinh (n\xi)}{n\e^{n\xi_0}}\sin (n\eta),\quad & \xi  \leqslant \xi_0,\\\nm
\ds -\frac{\sinh (n\xi_0)}{n\e^{n\xi}}\sin (n\eta),\quad & \xi  > \xi_0.
\end{cases}
\end{align}
Hence, there hold
\begin{align}\label{dS-ellipse-cos}
\frac{\p}{\p \nu}\mathcal{S}_{\Gamma_0}[ \beta_n^{c,0}](x) = \begin{cases}
\ds -\frac{\sinh (n\xi)}{\e^{n\xi_0}}\gamma (\xi, \eta)^{-1}\cos (n\eta),\quad & \xi  < \xi_0,\\\nm
\ds \frac{\cosh (n\xi_0)}{\e^{n\xi}}\gamma (\xi, \eta)^{-1}\cos (n\eta),\quad & \xi  > \xi_0,
\end{cases}
\end{align}
and
\begin{align}\label{dS-ellipse-sin}
\frac{\p}{\p \nu}\mathcal{S}_{\Gamma_0}[\beta_n^{s,0}](x) = \begin{cases}
\ds -\frac{\cosh (n\xi)}{\e^{n\xi_0}}\gamma (\xi, \eta)^{-1}\sin (n\eta),\quad & \xi  < \xi_0,\\\nm
\ds \frac{\sinh (n\xi_0)}{\e^{n\xi}}\gamma (\xi, \eta)^{-1}\sin (n\eta),\quad & \xi  > \xi_0.
\end{cases}
\end{align}
Moreover, we also have
\begin{align}\label{K-ellipse}
\mathcal{K}^*_{\Gamma_0} [\beta_n^{c,0}] = \frac{1}{2 \e^{2n\xi_0}}\beta_n^{c,0}\quad \mbox{and} \quad \mathcal{K}^*_{\Gamma_0} [\beta_n^{s,0}] = -\frac{1}{2 \e^{2n\xi_0}}\beta_n^{s,0}.
\end{align}

In what follows, we define $\bm{s}:=(\sinh (n\xi_1),\sinh (n\xi_2),\ldots,\sinh (n\xi_N))^T$ and $\bm{c}:=(\cosh (n\xi_1),\cosh (n\xi_2),\ldots,\cosh (n\xi_N))^T$. Let $\bm{e}_k:=(0,0,\ldots,1,0,\ldots,0)^T$ be the $N$-dimensional vector with the $k$-th entrance being one. We shall also define $n$-order even \textit{GPM} $\mathbf{M}_{N,c}^{(n)}$ as follows,
	\begin{equation}\label{MNc}
	\begin{split}
	\mathbf{M}_{N,c}^{(n)}:=
	\begin{bmatrix}
	\lambda_{1}- \frac{1}{2 \e^{2n\xi_1}} & -\frac{\cosh (n\xi_2)}{\e^{n\xi_1}} &  \cdots & -\frac{\cosh (n\xi_{N-1})}{\e^{n\xi_1}}& -\frac{\cosh (n\xi_{N})}{\e^{n\xi_1}} \\
	\nm
	\frac{\sinh (n\xi_2)}{ \e^{n\xi_1}} & 	\lambda_{2}- \frac{1}{2 \e^{2n\xi_2}}& \cdots & -\frac{\cosh (n\xi_{N-1})}{\e^{n\xi_2}}& -\frac{\cosh (n\xi_N)}{\e^{n\xi_2}}\\
	\nm
	\vdots & \vdots &\ddots& \vdots &\vdots \\
	\nm
	\frac{\sinh (n\xi_{N-1})}{ \e^{n\xi_1}} &  \frac{\sinh (n\xi_{N-1})}{ \e^{n\xi_2}} & \cdots &\lambda_{N-1}- \frac{1}{2 \e^{2n\xi_{N-1}}}&-\frac{\cosh (n\xi_N)}{\e^{n\xi_{N-1}}}\\
	\nm
	\frac{\sinh (n\xi_N)}{ \e^{n\xi_1}} &  \frac{\sinh (n\xi_N)}{ \e^{n\xi_2}}& \cdots & \frac{\sinh (n\xi_N)}{ \e^{n\xi_{N-1}}}& 	\lambda_{N}- \frac{1}{2 \e^{2n\xi_N}}
	\end{bmatrix}.
	\end{split}
	\end{equation}
Furthermore, we define $n$-order odd \textit{GPM} $\mathbf{M}_{N,s}^{(n)}$ as follows,
	\begin{equation}\label{MNs}
	\begin{split}
	\mathbf{M}_{N,s}^{(n)}:=
	\begin{bmatrix}
	\lambda_{1}+ \frac{1}{2 \e^{2n\xi_1}} & -\frac{\sinh (n\xi_2)}{\e^{n\xi_1}} &  \cdots & -\frac{\sinh (n\xi_{N-1})}{\e^{n\xi_1}}& -\frac{\sinh (n\xi_{N})}{\e^{n\xi_1}} \\
	\nm
	\frac{\cosh (n\xi_2)}{ \e^{n\xi_1}} & 	\lambda_{2}+ \frac{1}{2 \e^{2n\xi_2}}& \cdots & -\frac{\sinh (n\xi_{N-1})}{\e^{n\xi_2}}& -\frac{\sinh (n\xi_N)}{\e^{n\xi_2}}\\
	\nm
	\vdots & \vdots &\ddots& \vdots &\vdots \\
	\nm
	\frac{\cosh (n\xi_{N-1})}{ \e^{n\xi_1}} &  \frac{\cosh (n\xi_{N-1})}{ \e^{n\xi_2}} & \cdots &\lambda_{N-1}+ \frac{1}{2 \e^{2n\xi_{N-1}}}&-\frac{\sinh (n\xi_N)}{\e^{n\xi_{N-1}}}\\
	\nm
	\frac{\cosh (n\xi_N)}{ \e^{n\xi_1}} &  \frac{\cosh (n\xi_N)}{ \e^{n\xi_2}}& \cdots & \frac{\cosh (n\xi_N)}{ \e^{n\xi_{N-1}}}& 	\lambda_{N}+ \frac{1}{2 \e^{2n\xi_N}}
	\end{bmatrix}.
	\end{split}
	\end{equation}

Our main result in this subsection is the following.

\begin{thm}\label{thm41}
	Let the multi-layered confocal ellipses $A=\cup_{k=1}^NA_k$ be given by \eqref{eq:aj}--\eqref{interface}. Assume that the background electrical potential $H$ admits the following multipolar expansion:
	\begin{equation}\label{eq:defHc}
	H= \sum_{n=1}^{+\infty} a_{n,c}\cosh (n\xi) \cos (n\eta),
	\end{equation}
	where $a_{n,c}$, $n=1,2,\ldots,$  are arbitrary constants.
	Then, the solution of \eqref{potential_matrix} is given by
	\[
	\phi_k=\sum_{n=1}^{+\infty}na_{n,c}\beta_n^{c,k}\bm{e}_k^T\(\mathbf{M}^{(n)}_{N,c}\)^{-1}\bm{s},
	\]
	provided that the corresponding $n$-order even \textit{GPM} $\mathbf{M}_{N,c}^{(n)}$  is invertible,
\end{thm}
\begin{proof}[\bf Proof]

	From \eqref{S-ellipse-cos}, \eqref{dS-ellipse-cos} and \eqref{K-ellipse}, if $\bm\phi$ is given by
	\begin{equation}\label{PHIc3}
	\bm\phi = \sum_{n=1}^{+\infty} (\phi^n_1\beta_n^{c,1},\phi^n_2\beta_n^{c,2},\ldots,\phi^n_N\beta_n^{c,N})^T,
	\end{equation}
	then the integral equations \eqref{potential_matrix} are equivalent to
	\[
	\left\{
	\begin{aligned}
	& \(\lambda_{1}- \frac{1}{2 \e^{2n\xi_1}}\) \phi^n_1-\frac{1}{\e^{n\xi_1}}\sum_{k=2}^{N}\phi^n_k\cosh (n\xi_k)=na_{n,c}\sinh (n\xi_1),\\
	&\sinh (n\xi_l)\sum_{k=1}^{l-1}\phi^n_k\frac{1}{ \e^{n\xi_k}}+\(\lambda_{l}-\frac{1}{2 \e^{2n\xi_l}}\)\phi^n_l-\frac{1}{\e^{n\xi_l}}\sum_{k=l+1}^{N}\phi^n_k\cosh (n\xi_k)=na_{n,c}\sinh (n\xi_l), l=2,3,\ldots,N-1,\\
	&\sinh (n\xi_N)\sum_{k=1}^{N-1}\phi^n_k\frac{1}{ \e^{n\xi_k}}+\(\lambda_{N}-\frac{1}{2 \e^{2n\xi_N}}\)\phi^n_N=na_{n,c}\sinh (n\xi_N).
	\end{aligned}
	\right.
	\]
	Therefore,
	we can obtain that
	\[
	\mathbf{M}_{N,c}^{(n)}\((\phi^n_1,\phi^n_2,\ldots,\phi^n_N)^T\)=na_{n,c}\bm{s},
	\]
	which implies that
	\[
	\phi_k=\sum_{n=1}^{+\infty}na_{n,c}\beta_n^{c,k}\bm{e}_k^T\(\mathbf{M}_{N,c}^{(n)}\)^{-1}\bm{s}.
	\]
	The proof is complete.
\end{proof}

As an immediate application of the above theorem we obtain the following explicit form of the perturbed electric potential in entire two-dimensional space $\RR^2$.

\begin{thm}\label{th:42}
	Suppose $u$ is the solution to the conductivity transmission problem \eqref{conductvtA&AC}--\eqref{transmi_cdtion} with $A$ being the multi-layered confocal ellipses given by \eqref{eq:aj}--\eqref{interface}. Let $H$ be given by \eqref{eq:defHc}.
	If $\mathbf{M}_{N,c}^{(n)}$ is invertible, then  the problem \eqref{cductvt_transmi_prblm} is uniquely solvable with the solution given by the following formula:
	\[
		u-H=-\sum_{n=1}^{+\infty}a_{n,c}\cos (n\eta)\(\cosh(n\xi)\bm{e}_{1:l}^T +\frac{1}{\e^{n\xi}}\bm{e}_{l+1:N}^T \)\mathbb{F}^{(n)}_{N,l,c}\(\mathbf{M}^{(n)}_{N,c}\)^{-1}\bm{s}, \quad \mbox{in } A_l,\; l =0,1,\ldots,N,
	\]
	where $\bm{e}_{i:j} = \sum_{k=i}^{j}\bm{e}_k$ if $i\leqslant j$, and  $\bm{e}_{i:j} =0$ if $i>j$, and
	\[
	\mathbb{F}^{(n)}_{N,l,c}:= \begin{bmatrix}
	\e^{-n\xi_1}&\cdots & 0 & 0 &\cdots& 0 \\
	\vdots&\ddots & \vdots & \vdots &\ddots & \vdots\\
	0&\cdots & \e^{-n\xi_l} & 0 &\cdots& 0 \\
	0&\cdots & 0 & \cosh(n\xi_{l+1}) &\cdots& 0\\
	\vdots & \ddots & \vdots & \vdots& \ddots & \vdots\\
	0&\cdots & 0 &
	0&\cdots& \cosh(n\xi_N)\\
	\end{bmatrix}.
	\]
\end{thm}
\begin{proof}[\bf Proof]
	It follows from Theorem \ref{thm41}, \eqref{potential_sol} and \eqref{S-ellipse-cos} that the perturbed electric potential $u-H$ in $A_l$, $l =0,1,\ldots,N$, can be given by
	\[
	\begin{aligned}
	u-H&=\sum_{n=1}^{+\infty}na_{n,c}\sum_{k=1}^{N}\Scal_{\Gamma_k}[\beta_n^{c,k}]\bm{e}_k^T\(\mathbf{M}^{(n)}_{N,c}\)^{-1}\bm{s}\\
	&=-\sum_{n=1}^{+\infty}a_{n,c}\cos (n\eta)\(\cosh(n\xi)\sum_{k=1}^{l}  \frac{1}{\e^{n\xi_k}} +\frac{1}{\e^{n\xi}}\sum_{k=l+1}^{N} \cosh(n\xi_k) \)\bm{e}_k^T\(\mathbf{M}^{(n)}_{N,c}\)^{-1}\bm{s}\\
	&=-\sum_{n=1}^{+\infty}a_{n,c}\cos (n\eta)\(\cosh(n\xi)\bm{e}_{1:l}^T +\frac{1}{\e^{n\xi}}\bm{e}_{l+1:N}^T \)\mathbb{F}^{(n)}_{N,l,c}\(\mathbf{M}^{(n)}_{N,c}\)^{-1}\bm{s}.
	\end{aligned}
	\]
	The proof is complete.
\end{proof}

In a similar manner, we can obtain the following results.

\begin{thm}\label{thm44}
	Let the multi-layered confocal ellipses $A=\cup_{k=1}^NA_k$ be given by \eqref{eq:aj}--\eqref{interface}. Assume that the background electrical potential $H$ admits the following multipolar expansion:
	\begin{equation}\label{eq:defHs}
	H= \sum_{n=1}^{+\infty} a_{n,s}\sinh (n\xi) \sin (n\eta),
	\end{equation}
	where $a_{n,s}$, $n=1,2,\ldots,$  are arbitrary constants.
	Then, the solution of \eqref{potential_matrix} is given by
	\[
	\phi_k=\sum_{n=1}^{+\infty}na_{n,s}\beta_n^{s,k}\bm{e}_k^T\(\mathbf{M}^{(n)}_{N,s}\)^{-1}\bm{c},	
	\]
	provided that the corresponding   $n$-order odd \textit{GPM} $\mathbf{M}_{N,s}^{(n)}$ is invertible.
\end{thm}
%\begin{proof}
%	From \eqref{S-ellipse-sin}, \eqref{dS-ellipse-sin} and \eqref{K-ellipse}, if $\bm\phi$ is given by
%	\begin{equation}\label{PHI3}
%	\bm\phi = \sum_{n=1}^{+\infty} (\phi^n_1\beta_n^{s,1},\phi^n_2\beta_n^{s,2},\ldots,\phi^n_N\beta_n^{s,N})^T,
%	\end{equation}
%	then the integral equations \eqref{potential_matrix} are equivalent to
%	\[
%	\left\{
%	\begin{aligned}
%	& \(\lambda_{1}+ \frac{1}{2 \e^{2n\xi_1}}\) \phi^n_1-\frac{1}{\e^{n\xi_1}}\sum_{k=2}^{N}\phi^n_k\sinh (n\xi_k)=na_{n,s}\cosh (n\xi_1),\\
%	&\cosh (n\xi_l)\sum_{k=1}^{l-1}\phi^n_k\frac{1}{ \e^{n\xi_k}}+\(\lambda_{l}+\frac{1}{2 \e^{2n\xi_l}}\)\phi^n_l-\frac{1}{\e^{n\xi_l}}\sum_{k=l+1}^{N}\phi^n_k\sinh (n\xi_k)=na_{n,s}\cosh (n\xi_l), l=2,3,\ldots,N-1,\\
%	&\cosh (n\xi_N)\sum_{k=1}^{N-1}\phi^n_k\frac{1}{ \e^{n\xi_k}}+\(\lambda_{N}+\frac{1}{2 \e^{2n\xi_N}}\)\phi^n_N=na_{n,s}\cosh (n\xi_N).
%	\end{aligned}
%	\right.
%	\]
%	Therefore,
%	we can obtain that
%	\[
%	\mathbf{M}_{N,s}^{(n)}\((\phi^n_1,\phi^n_2,\ldots,\phi^n_N)^T\)=na_{n,s}\bm{c}.
%	\]
%	Thus, we can deduce that
%	\begin{equation}\label{PHINK}
%	\phi_k=\sum_{n=1}^{+\infty}na_{n,s}\beta_n^{s,k}\bm{e}_k^T\(\mathbf{M}_{N,s}^{(n)}\)^{-1}\bm{c}.
%	\end{equation}
%	The proof is complete.
%\end{proof}

%As an immediate application of the above theorem we obtain the following explicit form of the perturbed electric potential in entire two-dimensional space $\RR^2$.

\begin{thm}\label{th:45}
	Suppose $u$ is the solution to the conductivity transmission problem \eqref{conductvtA&AC}--\eqref{transmi_cdtion} with $A$ being the multi-layered confocal ellipses given by \eqref{eq:aj}--\eqref{interface}. Let $H$ be given by \eqref{eq:defHs}.
	If $\mathbf{M}_{N,s}^{(n)}$ is invertible, then  the problem \eqref{cductvt_transmi_prblm} is uniquely solvable with the solution given by the following formula:
	\[
	u-H=-\sum_{n=1}^{+\infty}a_{n,s}\sin (n\eta)\(\sinh(n\xi)\bm{e}_{1:l}^T +\frac{1}{\e^{n\xi}}\bm{e}_{l+1:N}^T \)\mathbb{F}^{(n)}_{N,l,s}\(\mathbf{M}^{(n)}_{N,s}\)^{-1}\bm{c}, \quad \mbox{in } A_l,\; l =0,1,\ldots,N,
	\]
	where
	\[
	\mathbb{F}^{(n)}_{N,l,s}:= \begin{bmatrix}
	\e^{-n\xi_1}&\cdots & 0 & 0 &\cdots& 0 \\
	\vdots&\ddots & \vdots & \vdots &\ddots & \vdots\\
	0&\cdots & \e^{-n\xi_l} & 0 &\cdots& 0 \\
	0&\cdots & 0 & \sinh(n\xi_{l+1}) &\cdots& 0\\
	\vdots & \ddots & \vdots & \vdots& \ddots & \vdots\\
	0&\cdots & 0 &
	0&\cdots& \sinh(n\xi_N)\\
	\end{bmatrix}.
	\]
\end{thm}
%\begin{proof}
%		It then follows from Theorem \ref{thm44}, \eqref{potential_sol} and \eqref{S-ellipse-sin} that the perturbed electric potential $u-H$ in $A_l$, $l =0,1,\ldots,N$, can be given by
%	\[
%	\begin{aligned}
%	u-H&=\sum_{n=1}^{+\infty}na_{n,s}\sum_{k=1}^{N}\Scal_{\Gamma_k}[\beta_n^{s,k}]\bm{e}_k^T\(\mathbf{M}^{(n)}_{N,s}\)^{-1}\bm{c}\\
%	&=-\sum_{n=1}^{+\infty}a_{n,s}\sin (n\eta)\(\sinh(n\xi)\sum_{k=1}^{l}  \frac{1}{\e^{n\xi_k}} +\frac{1}{\e^{n\xi}}\sum_{k=l+1}^{N} \sinh(n\xi_k) \)\bm{e}_k^T\(\mathbf{M}^{(n)}_{N,s}\)^{-1}\bm{c}\\
%	&=-\sum_{n=1}^{+\infty}a_{n,s}\sin (n\eta)\(\sinh(n\xi)\bm{e}_{1:l}^T +\frac{1}{\e^{n\xi}}\bm{e}_{l+1:N}^T \)\mathbb{F}^{(n)}_{N,l,s}\(\mathbf{M}^{(n)}_{Ns}\)^{-1}\bm{c}.
%	\end{aligned}
%	\]
%	The proof is complete.
%\end{proof}

\begin{rem}\label{rem41}
	One can see that the problem \eqref{cductvt_transmi_prblm} is solvable if and only if the corresponding   $n$-order  GPM $\mathbf{M}_{N,c}^{(n)}$ and $\mathbf{M}_{N,s}^{(n)}$  is invertible.
	If $\Re\sigma>0$,  then $\nabla\cdot \sigma(x)\nabla$ is an elliptic PDO (Partial Differential Operator), which guarantees the well-posedness of \eqref{cductvt_transmi_prblm} and hence the invertibility of $\mathbf{M}_{N,c}^{(n)}$ and $\mathbf{M}_{N,s}^{(n)}$ . On the other hand, if for some layers the materials are negative, the ellipticity of $\nabla\cdot \sigma(x)\nabla$ might be broken. Nevertheless, if for those negative materials, say $\sigma_k$, one has that $\Im\sigma_k\neq 0$, then the invertibility of $\mathbf{M}_{N,c}^{(n)}$ and $\mathbf{M}_{N,s}^{(n)}$ can also be guaranteed. This can be seen in our subsequent analysis.
\end{rem}
%\begin{rem}\label{rem42}
%%	By the above analysis,	we establish a connection between the eigenvalues $-\lambda$ of NP operator $\mathbb{K}_N^*$  and the eigenvalues $\beta$ of NP matrix $\(\mathbf{K}_N^{(n)}\)^T$:
%%	\begin{equation}\label{betaandlambda}
%%	\beta=
%%	\left\{
%%	\begin{array}{ll}
%%	-2\lambda, &   \mbox{if } \quad d=2,\\
%%	-\frac{2n+1}{n}\lambda+\frac{1}{2n}, &   \mbox{if } \quad d=3.
%%	\end{array}
%%	\right.
%%	\end{equation}
%	From the properties of $\mathbb{K}_N^*$ obtained in section \ref{sec3}, we can conclude that the matrix $\(\mathbf{K}_N^{(n)}\)^T$ has $N$ real eigenvalues lying in the interval $[-1,1]$ in $\RR^2$ and lying in the interval $[-1,\frac{n+1}{n}]$ in $\RR^3$, respectively, which perfectly coincide with the results in $\RR^2$ and of the uniform background field (dipole mode $n=1$) in $\RR^3$ provided in \cite{DFArxiv}.  Each eigenvalue is associated with one type of plasmon mode. This means that they may appear as in-gap localized surface plasmon modes once the material is truncated to form a layered structure. In other words, truncating the material will result in mode splitting.
%\end{rem}

\section{An algebraic framework for plasmon hybridization}\label{sec5}
In this section, we shall establish an algebraic framework for polariton hybridization. We mention that in order to analyze the plasmon resonance phenomena, one should consider the situation that the material
is not regular.
\subsection{Plasmon resonance}
As mentioned above, plasmon resonance is usually associated with some eigenvalue problem generated by the PDE system. According to our earlier discussion in Remark \ref{rem41}, if the parameters $\sigma_k$, $k=1, 2, \ldots, N$,  are all positive real-valued,  the elastostatic system the conductivity transmission problem \eqref{conductvtA&AC}--\eqref{transmi_cdtion} has only trivial solution if ${H}=0$. We seek non-trivial solutions to \eqref{conductvtA&AC}--\eqref{transmi_cdtion} when the conductivity in \eqref{eq:vepdef01} is allowed to be negative valued, i.e.,
\begin{equation}\label{negmat}
\sigma_{1}=-\sigma^*+\mathrm{i}\delta\quad \mbox{ with } \sigma^*>0,
\end{equation}
where $\delta > 0$ is some small lossy parameter. $\mathrm{i} = \sqrt{-1}$. Then, $\lambda$ given in \eqref{eq:deflamb01} can be rewritten by
\[
\lambda=\frac{\sigma_{0}-\sigma^*+\mathrm{i}\delta}{2(-\sigma_{0}-\sigma^*+\mathrm{i}\delta)}.
\]

By using Theorems \ref{th:42} and \ref{th:45}, we can oatain the exact  matrix representation of the NP operator in multi-layered confocal ellipses.
\begin{prop}\label{cor51}
	Suppose $u$ is the solution to the conductivity transmission problem \eqref{conductvtA&AC}--\eqref{transmi_cdtion} with  the conductivity $\sigma $ given by \eqref{eq:vepdef01} and $A$ being the multi-layered confocal ellipses given by \eqref{eq:aj}--\eqref{interface}.
	\begin{itemize}
	\item[\bf(i)]Let $H$ be given by \eqref{eq:defHc}.
	Define the corresponding  $n$-order even {NP matrix } $\(\mathbf{K}_{N,c}^{(n)}\)^T$ as follows:
	\begin{equation}\label{KN(n)T3}
	\(\mathbf{K}_{N,c}^{(n)}\)^T:=
	\begin{bmatrix}
	-\frac{1}{2 \e^{2n\xi_1}} & -\frac{\cosh (n\xi_2)}{\e^{n\xi_1}} &  \cdots & -\frac{\cosh (n\xi_{N-1})}{\e^{n\xi_1}}& -\frac{\cosh (n\xi_{N})}{\e^{n\xi_1}} \\
	\nm
	-\frac{\sinh (n\xi_2)}{ \e^{n\xi_1}} &  \frac{1}{2 \e^{2n\xi_2}}& \cdots & \frac{\cosh (n\xi_{N-1})}{\e^{n\xi_2}}& \frac{\cosh (n\xi_N)}{\e^{n\xi_2}}\\
	\nm
	\vdots & \vdots &\ddots& \vdots &\vdots \\
	\nm
	(-1)^{N-2}\frac{\sinh (n\xi_{N-1})}{ \e^{n\xi_1}} &  (-1)^{N-2}\frac{\sinh (n\xi_{N-1})}{ \e^{n\xi_2}} & \cdots & (-1)^{N-1}\frac{1}{2 \e^{2n\xi_{N-1}}}&(-1)^{N-1}\frac{\cosh (n\xi_N)}{\e^{n\xi_{N-1}}}\\
	\nm
	(-1)^{N-1}\frac{\sinh (n\xi_N)}{ \e^{n\xi_1}} &  (-1)^{N-1}\frac{\sinh (n\xi_N)}{ \e^{n\xi_2}}& \cdots & (-1)^{N-1}\frac{\sinh (n\xi_N)}{ \e^{n\xi_{N-1}}}& 	 (-1)^{N}\frac{1}{2 \e^{2n\xi_N}}
	\end{bmatrix}.
	\end{equation}
	If $-\lambda$ is not the eigenvalue of  $\mathbf{K}_{N,c}^{(n)}$,
	then  the problem \eqref{cductvt_transmi_prblm} is uniquely solvable with the solution given by the following formula:
	\begin{equation}\label{eq:purbmn03}
	u-H=\sum_{n=1}^{+\infty}a_{n,c}\cos (n\eta)\(\cosh(n\xi)\bm{e}_{1:l}^T +\frac{1}{\e^{n\xi}}\bm{e}_{l+1:N}^T \)\mathbb{F}^{(n)}_{N,l,c} \(-\lambda\mathbf{I}_N-\(\mathbf{K}_{N,c}^{(n)}\)^T\)^{-1}\tilde{\bm{s}},
	\end{equation}
	in  $A_l,\; l =0,1,\ldots,N,$ where $\tilde{\bm{s}} = \(\sinh (n\xi_1),-\sinh (n\xi_2),\ldots,(-1)^{N-1}\sinh (n\xi_N)\)^T$, $\mathbf{I}_N$ is the identity matrix.
	\item[\bf(ii)] Let $H$ be given by \eqref{eq:defHs}.
	Define the corresponding  $n$-order odd {NP matrix } $\(\mathbf{K}_{N,s}^{(n)}\)^T$ as follows:
	%	\[
	%	\begin{split}
	%	\begin{bmatrix}
	% -\frac{1}{2 \e^{2n\xi_1}} & -\frac{\sinh (n\xi_2)}{\e^{n\xi_1}} &  \cdots & -\frac{\sinh (n\xi_{N-1})}{\e^{n\xi_1}}& -\frac{\sinh (n\xi_{N})}{\e^{n\xi_1}} \\
	%	-\frac{\cosh (n\xi_2)}{ \e^{n\xi_1}} & 	 \frac{1}{2 \e^{2n\xi_2}}& \cdots & \frac{\sinh (n\xi_{N-1})}{\e^{n\xi_2}}& \frac{\sinh (n\xi_N)}{\e^{n\xi_2}}\\
	%	\vdots & \vdots &\ddots& \vdots &\vdots \\
	%	(-1)^{N-2}\frac{\cosh (n\xi_{N-1})}{ \e^{n\xi_1}} &  (-1)^{N-2}\frac{\cosh (n\xi_{N-1})}{ \e^{n\xi_2}} & \cdots &(-1)^{N-1} \frac{1}{2 \e^{2n\xi_{N-1}}}&(-1)^{N-1}\frac{\sinh (n\xi_N)}{\e^{n\xi_{N-1}}}\\
	%	(-1)^{N-1}\frac{\cosh (n\xi_N)}{ \e^{n\xi_1}} &  (-1)^{N-1}\frac{\cosh (n\xi_N)}{ \e^{n\xi_2}}& \cdots & (-1)^{N-1}\frac{\cosh (n\xi_N)}{ \e^{n\xi_{N-1}}}& 	(-1)^{N}\frac{1}{2 \e^{2n\xi_N}}
	%	\end{bmatrix}.
	%	\end{split}
	%	\]
	
	\begin{equation}\label{KN(n)Ts3}
	\(\mathbf{K}_{N,s}^{(n)}\)^T:=
	\begin{bmatrix}
	\frac{1}{2 \e^{2n\xi_1}} & -\frac{\sinh (n\xi_2)}{\e^{n\xi_1}} &  \cdots & -\frac{\sinh (n\xi_{N-1})}{\e^{n\xi_1}}& -\frac{\sinh (n\xi_{N})}{\e^{n\xi_1}} \\
	\nm
	-\frac{\cosh (n\xi_2)}{ \e^{n\xi_1}} & 	 -\frac{1}{2 \e^{2n\xi_2}}& \cdots & \frac{\sinh (n\xi_{N-1})}{\e^{n\xi_2}}& \frac{\sinh (n\xi_N)}{\e^{n\xi_2}}\\
	\nm
	\vdots & \vdots &\ddots& \vdots &\vdots \\
	\nm
	(-1)^{N-2}\frac{\cosh (n\xi_{N-1})}{ \e^{n\xi_1}} &  (-1)^{N-2}\frac{\cosh (n\xi_{N-1})}{ \e^{n\xi_2}} & \cdots &(-1)^{N-2} \frac{1}{2 \e^{2n\xi_{N-1}}}&(-1)^{N-1}\frac{\sinh (n\xi_N)}{\e^{n\xi_{N-1}}}\\
	\nm
	(-1)^{N-1}\frac{\cosh (n\xi_N)}{ \e^{n\xi_1}} &  (-1)^{N-1}\frac{\cosh (n\xi_N)}{ \e^{n\xi_2}}& \cdots & (-1)^{N-1}\frac{\cosh (n\xi_N)}{ \e^{n\xi_{N-1}}}& 	(-1)^{N-1}\frac{1}{2 \e^{2n\xi_N}}
	\end{bmatrix}.
	\end{equation}
	If $-\lambda$ is not the eigenvalue of  $\mathbf{K}_{N,s}^{(n)}$,
	then  the problem \eqref{cductvt_transmi_prblm} is uniquely solvable with the solution given by the following formula:
	\begin{equation}\label{eq:purbmn04}
	u-H=\sum_{n=1}^{+\infty}a_{n,s}\sin (n\eta)\(\sinh(n\xi)\bm{e}_{1:l}^T +\frac{1}{\e^{n\xi}}\bm{e}_{l+1:N}^T \)\mathbb{F}^{(n)}_{N,l,s} \(-\lambda\mathbf{I}_N-\(\mathbf{K}_{N,s}^{(n)}\)^T\)^{-1}\tilde{\bm{c}},
	\end{equation}
	in  $A_l,\; l =0,1,\ldots,N,$ where $\tilde{\bm{c}} = \(\cosh (n\xi_1),-\cosh (n\xi_2),\ldots,(-1)^{N-1}\cosh (n\xi_N)\)^T$.
	\end{itemize}	
\end{prop}

%\begin{rem}\label{rem42}
%	%	By the above analysis,	we establish a connection between the eigenvalues $-\lambda$ of NP operator $\mathbb{K}_N^*$  and the eigenvalues $\beta$ of NP matrix $\(\mathbf{K}_N^{(n)}\)^T$:
%	%	\begin{equation}\label{betaandlambda}
%	%	\beta=
%	%	\left\{
%	%	\begin{array}{ll}
%	%	-2\lambda, &   \mbox{if } \quad d=2,\\
%	%	-\frac{2n+1}{n}\lambda+\frac{1}{2n}, &   \mbox{if } \quad d=3.
%	%	\end{array}
%	%	\right.
%	%	\end{equation}
%	From the properties of $\mathbb{K}_N^*$ obtained in section \ref{sec3}, we can conclude that the matrice $\(\mathbf{K}_{N,c}^{(n)}\)^T$ and $\(\mathbf{K}_{N,s}^{(n)}\)^T$  has $N$ real eigenvalues lying in the interval $[-1/2,1/2]$, respectively,   Each eigenvalue is associated with one type of plasmon mode. This means that they may appear as in-gap localized surface plasmon modes once the material is truncated to form a layered structure. In other words, truncating the material will result in mode splitting.
%\end{rem}

Next, we  give the definition of the plasmon resonance mode for our subsequent study.
\begin{defn}\label{df:def01}
	Consider the problem \eqref{cductvt_transmi_prblm} for $H=0$, associated with the $N$-layer structure $A$, where the material configuration is described in \eqref{eq:vepdef01} and \eqref{negmat}. Plasmon resonance mode occurs when there exists material configuration such that, in the limiting case as the loss parameter $\delta$ goes to zero, the corresponding PDO $\nabla\cdot \sigma(x)\nabla$  possesses non-trivial kernel.
\end{defn}
\begin{thm}\label{th:resonancedefth01}
	Suppose $ u$ is the solution to the conductivity transmission problem \eqref{conductvtA&AC}--\eqref{transmi_cdtion}, with the parameter $\sigma$ be given in \eqref{eq:vepdef01} and \eqref{negmat}. If there holds
	$$
	\lim_{\delta\rightarrow 0} \lambda = \lambda^*,
	$$
	where $-\lambda^*$ is an eigenvalue of $\(\mathbf{K}_{N,c}^{(n)}\)^T$ or $\(\mathbf{K}_{N,s}^{(n)}\)^T$, then the plasmon resonance mode occurs.
\end{thm}
\begin{proof}[\bf Proof]
	According to Definition \ref{df:def01}, one only need to verify that the corresponding PDO $\nabla\cdot \sigma(x)\nabla $  possesses a non-trivial kernel as $\delta \rightarrow 0$. To this end, let us consider the following problem:
	\begin{equation}\label{eq:lametrivi01}
	\left\{
	\begin{array}{ll}
	\nabla\cdot (\sigma(x)\nabla u) =0, & \mbox{in} \quad \RR^2,\\
	{u}=\Ocal(|x|^{-1}), & |x|\rightarrow \infty.
	\end{array}
	\right.
	\end{equation}
	%Let $\beta_i: =
	%-\frac{2n+1}{n}\lambda_i+\frac{1}{2n}$. By \eqref{eq:vepdef01}--\eqref{eq:Drude1}, we have
	%\[
	%\beta_i=
	%\left\{
	%\begin{array}{ll}
	%\beta, & i \quad \mbox{is odd},\\
	%\frac{1}{n}-\beta, & i \quad \mbox{is even},
	%\end{array}
	%\right.
	%\]
	%where
	%\[
	%\beta: =
	%-\frac{2n+1}{n}\lambda+\frac{1}{2n} = \frac{\frac{n+1}{n}\sigma_{0}+\sigma_{1}}{\sigma_{0}-\sigma_{1}},
	%\]
	By following a similar treatment as the proof of Theorem \ref{thm41}, one can show that
	\[
	\left\{
	\begin{aligned}
	& \(\lambda_{1}- \frac{1}{2 \e^{2n\xi_1}}\) \phi^n_1-\frac{1}{\e^{n\xi_1}}\sum_{k=2}^{N}\phi^n_k\cosh (n\xi_k)=0,\\
	&\sinh (n\xi_l)\sum_{k=1}^{l-1}\phi^n_k\frac{1}{ \e^{n\xi_k}}+\(\lambda_{l}-\frac{1}{2 \e^{2n\xi_l}}\)\phi^n_l-\frac{1}{\e^{n\xi_l}}\sum_{k=l+1}^{N}\phi^n_k\cosh (n\xi_k)=0, l=2,3,\ldots,N-1,\\
	&\sinh (n\xi_N)\sum_{k=1}^{N-1}\phi^n_k\frac{1}{ \e^{n\xi_k}}+\(\lambda_{N}-\frac{1}{2 \e^{2n\xi_N}}\)\phi^n_N=0.
	\end{aligned}
	\right.
	\]
	That is
	\begin{equation}\label{eq:notri01}
	\(-\lambda\mathbf{I}_N-\(\mathbf{K}_{N,c}^{(n)}\)^T\)\bm{\phi}=0,
	\end{equation}
	where $\bm{\phi} = \(\phi^n_1,\phi^n_2,\ldots,\phi^n_N\)$. If $-\lambda^*$ is an eigenvalue of $\(\mathbf{K}_{N,c}^{(n)}\)^T$, by letting $\delta \rightarrow 0$, one has $\left|-\lambda^*\mathbf{I}_N-\(\mathbf{K}_{N,c}^{(n)}\)^T\right|=0$ and thus the homogeneous linear equation \eqref{eq:notri01} has non-trivial solutions. In view of  $u(x)=\sum_{k=1}^{N}\Scal_{\Gamma_k}[\phi_k](x)$ and \eqref{PHIc3}, one immediately has  the problem \eqref{eq:lametrivi01} has non-trivial solutions.

	If $-\lambda^*$ is an eigenvalue of $\(\mathbf{K}_{N,s}^{(n)}\)^T$, one also has  the problem \eqref{eq:lametrivi01} has non-trivial solutions in a  similar manner. The proof is complete.
\end{proof}

At plasmon mode, the resonant field  demonstrates an energy blowup.
Next, similar to \cite{DLZJMPA21}, we give the formal definition of inducing of plasmon resonance.
\begin{defn}\label{df:def02}
	Consider the conductivity transmission problem \eqref{conductvtA&AC}--\eqref{transmi_cdtion} associated with the $N$-layer structure $A$, where the material configuration is described in \eqref{eq:vepdef01}--\eqref{eq:Drude1}.
	Then plasmon resonance is induced if the following condition is fulfilled:
	\[
	\lim_{\delta\rightarrow 0}
	||\nabla({u}-{H})||_{L^2(\RR^d\setminus\overline{A})}=+\infty.
	\]
\end{defn}

\begin{rem}\label{rem51}
	In view of Definition \ref{df:def02} and Corollary \ref{cor51}, the resonance condition is split into even and odd plasmon modes and reads, as $\delta\rightarrow 0$, $\left|-\lambda\mathbf{I}_N-\(\mathbf{K}_{N,c}^{(n)}\)^T\right|=0$ for even and $\left|-\lambda\mathbf{I}_N-\(\mathbf{K}_{N,s}^{(n)}\)^T\right|=0$ for odd modes.
	%The plasmon resonance mode is parallel to the material configurations which make the parameter $-\lambda$ the eigenvalue of $\(\mathbf{K}_{N,c}^{(n)}\)^T$ or $\(\mathbf{K}_{N,s}^{(n)}\)^T$ as $\delta\rightarrow 0$.
	For this reason, we shall focus on the explicit computation of eigenvalues to the matrix $\(\mathbf{K}_{N,c}^{(n)}\)^T$ or $\(\mathbf{K}_{N,s}^{(n)}\)^T$, or the exact formulae to $|\mathbf{M}_{N,c}^{(n)}|$ or $|\mathbf{M}_{N,s}^{(n)}|$. Thus, one can immediately obtain the  eigenvalues of matrix type NP operator $\mathbb{K}_N^*$.
\end{rem}

%\subsection{Resonant fields and energy blowup}
\subsection{An algebraic framework}
In this subsection, we shall derive a precise connection between the plasmon mode and the choice of conductivities in the multi-layered confocal ellipses.
To express the main results more clearly, we will defer the proofs to the subsequent section.
By some straightforward but rather lengthy and tedious calculations, we can obtain the exact formulae to $|\mathbf{M}_{N,c}^{(n)}|$ or $|\mathbf{M}_{N,s}^{(n)}|$ for the structures with small number of layers, say $2\leqslant N\leqslant 5$. By observing these results, we induce the delicate result about the explicit formulae of $|\mathbf{M}_{N,c}^{(n)}|$ or $|\mathbf{M}_{N,s}^{(n)}|$  for all $N\in \mathbb{N}$ in the following theorem.
For the convenience of description, we denote by $\mathbf{i}_{m}$ the multi-index $(i_1,i_2,\ldots,i_{m})$ and $\tau_{\mathbf{i}_{m}}$ its sign given by
\[
\tau_{\mathbf{i}_{m}}:=(-1)^{\sum_{j=1}^{m} i_j}.
\]
Moreover, we denote by $C^{i,m}_{n}$ the set of all combinations of $m$ out $n$, $m\leqslant n$, say e.g., for one combination
\[
(i_1, i_2, \ldots, i_{m})\in C^{i,m}_{n}\quad \mbox{satisfying} \quad
(i+1)\leqslant i_1,i_2,\ldots, i_{m}\leqslant (n+i),
\]  we order them in ascending way, that is, $i_1<i_2< \cdots< i_{m}$.

\begin{thm}\label{thm52}
Let $N\in \mathbb{N}$. Then it holds that
	\begin{equation}\label{cpbi-ktc}
	\left|\mathbf{M}_{N,c}^{(n)}(\lambda)\right|=
	(-1)^{\lfloor \frac{N}{2}\rfloor}\left(\sum\limits_{k=0}^{N}  \frac{\lambda^{N-k}}{2^k}\left(\sum\limits_{\mathbf{i}_{k}\in C_{N}^{0,k}}\tau_{\mathbf{i}_{k}}\e^{ 2 n\sum_{l=1}^k (-1)^{l}\xi_{i_l}}\right)\right),
	\end{equation}
	\begin{equation}\label{cpbi-kts}
	\left|\mathbf{M}_{N,s}^{(n)}(\lambda)\right|=
	(-1)^{\lfloor \frac{N}{2}\rfloor}\left(\sum\limits_{k=0}^{N}  \frac{\lambda^{N-k}}{(-2)^k}\left(\sum\limits_{\mathbf{i}_{k}\in C_{N}^{0,k}}\tau_{\mathbf{i}_{k}}\e^{ 2 n\sum_{l=1}^k (-1)^{l}\xi_{i_l}}\right)\right).
	\end{equation}
\end{thm}
%\begin{proof}
%The proof is essentially similar to that of \cite[Theorem 2.2]{DFArxiv} for dipole mode $n=1$ in $\RR^3$, so we omit it here.
%\end{proof}
\begin{rem}\label{rem52}
	%It can be seen that if the number of layers, $N$, is odd then there exists a zero root $\beta =0$ to the equation  $\left|\beta I-\(\mathbf{K}_N^{(n)}\)^T\right|=0$.
	We want to mention that the explicit formulation of characteristic polynomials \eqref{cpbi-ktc} and \eqref{cpbi-kts} establishes a handy algebraic framework for the study of analytical and numerical results on the eigenvalues of NP matrices $\(\mathbf{K}_{N,c}^{(n)}\)^T$ and $\(\mathbf{K}_{N,s}^{(n)}\)^T$, or equivalently NP operator $\mathbb{K}_N^*$. More precisely, by the Galois' Theory (cf. \cite{JPT2001}), it is possible to obtain the exact formula on the eigenvalues of $\mathbb{K}_N^*$ for $N\leqslant 4$ by using the quadratic, cubic or quartic formula for characteristic polynomials with respect to $\lambda$.
	Moreover, the explicit formulation \eqref{cpbi-ktc} and \eqref{cpbi-kts} provide an effective way  to obtain numerical results of these with a large number of layers.  Instead of manipulating explicit algebraic expressions, we shall focus on the corresponding numerical results for cases with larger number of layers in  Section \ref{numer_reslt}.
\end{rem}

%Next, analytical results of multi-layer structures with a small number of layers (actually up to nine-layer) will be presented.
%and  mathematical analysis of plasmon hybridization theory within multi-layer structures will be presented in next section.

%For the five-layer case, we shall provide an exact formula for the eigenvalues of $\mathbb{K}_N^*$. However, for cases with more than five layers, it is possible to express the eigenvalues explicitly up to the nine-layer case by utilizing the quartic formula for polynomials.
As already mentioned before, in order to find the plasmon modes, it is essential to find the roots of the polynomial
\begin{equation}\label{eq:deffq01}
f^{\pm}_N(\lambda):=\sum\limits_{k=0}^{N}  \frac{\lambda^{N-k}}{(\pm 2)^k}\left(\sum\limits_{\mathbf{i}_{k}\in C_{N}^{0,k}}\tau_{\mathbf{i}_{k}}\e^{ 2 n\sum_{l=1}^k (-1)^{l}\xi_{i_l}}\right).
\end{equation}
To this end, we have the following elementary result on the roots of $f^{\pm}_N(\lambda)$:
\begin{thm}\label{thm53}
	Let $f^{\pm}_N(\lambda)$ be defined in \eqref{eq:deffq01} and $\lambda^{\pm}$ be the root to $f^{\pm}_N(\lambda)$, respectively.
	Then there exists $N$ real values of roots to $f^{\pm}_N(\lambda)$, respectively. Moreover, we have $\lambda^+ = -\lambda^-\in \left[-1/2, 1/2\right]$.
\end{thm}

%\begin{rem}
%	Each eigenvalue is associated with one type of plasmon mode. This means that they may appear as in-gap localized surface plasmon modes once the material is truncated to form a layered structure. In other words, truncating the material will result in mode splitting.
%\end{rem}

 Next, in order to dicuss the relationship between the concentric disks case and the confocal ellipses case, we shall consider that the elliptic radius of the layers are large enough.

 \begin{thm}\label{thm54}
 	Suppose that $\xi_k = \tilde{\xi}+c_k$, where $\tilde{\xi}\gg 1$, and $c_1 >c_2> \ldots > c_{N}>0$ are regular constants. Then the polynomial  \eqref{eq:deffq01} can be rewritten as
 	\begin{equation}\label{ellip2disk}
 	f^{\pm}_N(\lambda)=\sum\limits_{k=0}^{\lfloor \frac{N}{2}\rfloor}  \frac{\lambda^{N-2k}}{( 2)^{2k}}\left(\sum\limits_{\mathbf{i}_{2k}\in C_{N}^{0,2k}}\tau_{\mathbf{i}_{2k}}\e^{ 2 n\sum_{l=1}^{2k} (-1)^{l}\xi_{i_l}}\right)+O(\e^{-2n\tilde{\xi}}):=f_{N,\textup{r}}(\lambda)+O(\e^{-2n\tilde{\xi}}).
 	\end{equation}
 \end{thm}

 \begin{rem}\label{rem53}
From \eqref{ellipcoor}, it follows that the confocal ellipse is a circle of radius $\frac{1}{2}R\e^{\xi}$ as $\xi\to\infty$. Consider the boundaries $\Gamma_k:=\left\{(\xi, \eta) : \xi = \xi_k\right\},$  $k=1,2,\ldots,N$, are confocal ellipses of elliptic radii ${\xi_k}  = \tilde{\xi}+c_k$, where $\tilde{\xi}\gg 1$  and $c_1 >c_2> \ldots > c_{N}$. We find that the polynomial $f_{N,\textup{r}}(\lambda)$ in \eqref{ellip2disk} is the characteristic polynomial of the NP operator $\mathbb{K}_N^*$ where the boundaries $\Gamma_k,$  $k=1,2,\ldots,N$, are circles of radii
$\frac{1}{2}R\e^{\xi_k}$, respectively (cf. \cite{FangdengMMA23,KZDF}).  Therefore, concentric discs are a special case where the radius of the confocal ellipse is sufficiently large.
 \end{rem}

It is known that as the curvature of that boundary point increases to a certain degree, then plasmon resonance occurs locally near the high-curvature point of the plasmonic inclusion. This is referred to as the localization and geometrization in the plasmon resonance (see \cite{ACLJFA23,BLLW_ESAIM2020}).
The corresponding curvature at a boundary point $(\xi_k, \eta)\in \Gamma_k$ can be directly calculated to be
\begin{equation}
\kappa(\eta) = \frac{\cosh\xi_k \sinh\xi_k}{R(\sinh^2\xi_k+\sin^2 \eta)^{\frac{3}{2}}}.
\end{equation}
Hence, the largest curvature $\kappa_{\textup{max}} = \frac{\cosh\xi_k}{R\sinh^2\xi_k}$ is attainable at the two vertices with $\eta = \pi$ and $\eta = 0$ respectively on the semi-major axis. It is noted that for a fixed $R,$ the curvature $\kappa_{\textup{max}}$ increases as $\xi_k$ decreases and actually one has that $\kappa_{\textup{max}} \to \infty$ as $\xi_k \to 0$. To this end, we shall consider that the elliptic radius of the layers are small enough.

\begin{thm}\label{thm55}
Suppose that $0<\xi_k\ll 1$, $k=1, 2, \ldots, N$. Then the polynomial  \eqref{eq:deffq01} can be rewritten as
\begin{equation}
f^{\pm}_N(\lambda)=\sum\limits_{k=0}^{N}  \frac{\lambda^{N-k}}{(\pm 2)^k}\left(\sum\limits_{\mathbf{i}_{k}\in C_{N}^{0,k}}\tau_{\mathbf{i}_{k}}\right)+o(1) =\(\lambda^2-\frac{1}{4}\)^{\lfloor \frac{N}{2}\rfloor} \(\lambda\mp\frac{1}{2}\)^{\frac{(-1)^{N-1}+1}{2}}+o(1).
%= \(\lambda\pm\frac{1}{2}\)^{\lfloor \frac{N}{2}\rfloor} \(\lambda\mp\frac{1}{2}\)^{\lfloor \frac{N}{2}\rfloor+\frac{(-1)^{N-1}+1}{2}}+o(1).
\end{equation}
\end{thm}
%\begin{rem}
%We want to mention that even at plasmon mode, the plasmon resonance may still not be induced. We refer to \cite{FDLMMA15} for details on inducing plasmon resonance.
%Here we show some explicit example based on Theorem \ref{thm55}. It is clearly that in the case of multi-layer thin strips in Theorem \ref{thm55}, the two possible plasmon modes are $\lambda=\pm \frac{1}{2}$. However, the plasmon resonance may not be induced for $\lambda=-\frac{1}{2}$. To better explain this, we consider a two-layered case and suppose that the background field $H=\cosh \xi  \cos \eta$, i.e., $H$ is uniform background field. It then can be seen from (\ref{eq:purbmn03}) that
%\[
%\begin{split}
%u-H=&\cos \eta\frac{1}{e^{\xi}}\bm{e}^T \mathbb{F}^{(1)}_{2,l,c} \(-\lambda\mathbf{I}_2-\(\mathbf{K}_{2,c}^{(1)}\)^T\)^{-1}\tilde{\bm{s}}\\
%=&\cos \eta\frac{\xi_2-\xi_1}{(\lambda-1/2)e^{\xi}}+o(\xi_1+\xi_2)    \quad \quad \mbox{in} \quad A_0.
%\end{split}
%\]
%It can be readily seen that the blow up on gradient of leading order term only occurs at $\lambda=1/2$ when $(\xi_1-\xi_2)\sim(\xi_1+\xi_2)$.
%\end{rem}

\section{Proofs of Theorems \ref{thm52}--\ref{thm55}}\label{sec6}

\subsection{Eigenvalue problem}\label{subsec4.1}

As mentioned above, the plasmon mode is parallel to the material configurations which make the parameter $-\lambda$ the eigenvalue of $\(\mathbf{K}_{N,c}^{(n)}\)^T$ or $\(\mathbf{K}_{N,s}^{(n)}\)^T$ as $\tau\rightarrow 0$. So, we focus on the eigenvalues of the matrices $\(\mathbf{K}_{N,c}^{(n)}\)^T$ and $\(\mathbf{K}_{N,s}^{(n)}\)^T$, or the explicit formulae  of $\left|\mathbf{M}_{N,c}^{(n)}\right|$ and $\left|\mathbf{M}_{N,c}^{(n)}\right|$  defined in \eqref{MNc} and \eqref{MNs}, respectively, in this subsection.
For this, we first define
\begin{equation}\label{eq:matP02}
\mathbf M^{(n)}_{i:N,c}:= \begin{bmatrix}
\lambda_{i}- \frac{1}{2 \e^{2n\xi_i}} & -\frac{\cosh (n\xi_{i+1})}{\e^{n\xi_i}} &  \cdots & -\frac{\cosh (n\xi_{N-1})}{\e^{n\xi_i}}& -\frac{\cosh (n\xi_{N})}{\e^{n\xi_i}} \\
\nm
\frac{\sinh (n\xi_{i+1})}{ \e^{n\xi_i}} & 	\lambda_{i+1}- \frac{1}{2 \e^{2n\xi_{i+1}}}& \cdots & -\frac{\cosh (n\xi_{N-1})}{\e^{n\xi_{i+1}}}& -\frac{\cosh (n\xi_N)}{\e^{n\xi_{i+1}}}\\
\nm
\vdots & \vdots &\ddots& \vdots &\vdots \\
\nm
\frac{\sinh (n\xi_{N-1})}{ \e^{n\xi_i}} &  \frac{\sinh (n\xi_{N-1})}{ \e^{n\xi_{i+1}}} & \cdots &\lambda_{N-1}- \frac{1}{2 \e^{2n\xi_{N-1}}}&-\frac{\cosh (n\xi_N)}{\e^{n\xi_{N-1}}}\\
\nm
\frac{\sinh (n\xi_N)}{ \e^{n\xi_i}} &  \frac{\sinh (n\xi_N)}{ \e^{n\xi_{i+1}}}& \cdots & \frac{\sinh (n\xi_N)}{ \e^{n\xi_{N-1}}}& 	\lambda_{N}- \frac{1}{2 \e^{2n\xi_N}}
\end{bmatrix},
\end{equation}
and
\begin{equation}\label{eq:matP03}
\mathbf{M}_{i:N,s}^{(n)}:=
\begin{bmatrix}
\lambda_{i}+ \frac{1}{2 \e^{2n\xi_i}} & -\frac{\sinh (n\xi_{i+1})}{\e^{n\xi_i}} &  \cdots & -\frac{\sinh (n\xi_{N-1})}{\e^{n\xi_i}}& -\frac{\sinh (n\xi_{N})}{\e^{n\xi_i}} \\
\nm
\frac{\cosh (n\xi_{i+1})}{ \e^{n\xi_i}} & 	\lambda_{i+1}+ \frac{1}{2 \e^{2n\xi_{i+1}}}& \cdots & -\frac{\sinh (n\xi_{N-1})}{\e^{n\xi_{i+1}}}& -\frac{\sinh (n\xi_N)}{\e^{n\xi_{i+1}}}\\
\nm
\vdots & \vdots &\ddots& \vdots &\vdots \\
\nm
\frac{\cosh (n\xi_{N-1})}{ \e^{n\xi_i}} &  \frac{\cosh (n\xi_{N-1})}{ \e^{n\xi_{i+1}}} & \cdots &\lambda_{N-1}+ \frac{1}{2 \e^{2n\xi_{N-1}}}&-\frac{\sinh (n\xi_N)}{\e^{n\xi_{N-1}}}\\
\nm
\frac{\cosh (n\xi_N)}{ \e^{n\xi_i}} &  \frac{\cosh (n\xi_N)}{ \e^{n\xi_{i+1}}}& \cdots & \frac{\cosh (n\xi_N)}{ \e^{n\xi_{N-1}}}& 	\lambda_{N}+ \frac{1}{2 \e^{2n\xi_N}}
\end{bmatrix}.
\end{equation}
and set $\mathbf M^{(n)}_{N+1:N,s}=\mathbf M^{(n)}_{N+1:N,c} = 1$. Obviously,  $\mathbf M^{(n)}_{N,c}=\mathbf M^{(n)}_{1:N,c}$ and $\mathbf M^{(n)}_{N,s}=\mathbf M^{(n)}_{1:N,s}$.

%\[
%\mathbf{M}_{i:N,s}^{(n)}:=
%\begin{bmatrix}
%\lambda_{i}+\lambda_{i+1}\e^{2n(\xi_{i+1}-\xi_i)} & -(\lambda_{i+1}+ \frac{1}{2})\e^{n(\xi_{i+1}-\xi_i)} &  \cdots & 0& 0 \\
%\nm
%-(\lambda_{i+1}- \frac{1}{2})\e^{n(\xi_{i+1}-\xi_i)}  & 	\lambda_{i+1}+ \frac{1}{2 \e^{2n\xi_{i+1}}}& \cdots & -\frac{\sinh (n\xi_{N-1})}{\e^{n\xi_{i+1}}}& -\frac{\sinh (n\xi_N)}{\e^{n\xi_{i+1}}}\\
%\nm
%\vdots & \vdots &\ddots& \vdots &\vdots \\
%\nm
%0 &  \frac{\cosh (n\xi_{N-1})}{ \e^{n\xi_{i+1}}} & \cdots &\lambda_{N-1}+ \frac{1}{2 \e^{2n\xi_{N-1}}}&-\frac{\sinh (n\xi_N)}{\e^{n\xi_{N-1}}}\\
%\nm
%0 &  \frac{\cosh (n\xi_N)}{ \e^{n\xi_{i+1}}}& \cdots & \frac{\cosh (n\xi_N)}{ \e^{n\xi_{N-1}}}& 	\lambda_{N}+ \frac{1}{2 \e^{2n\xi_N}}
%\end{bmatrix}.
%\]

Next, we give the recursion formulae for $\left|\mathbf{M}_{i:N,c}^{(n)}\right|$ and $\left|\mathbf{M}_{i:N,s}^{(n)}\right|$, respectively, in the following lemma.
\begin{lem}\label{le:main01}
	Let $N\geqslant 2$. Then there holds the following  recursion formulae for $i = 1,2,\ldots,N-1$:
	\begin{equation}\label{eq:deter0102}
	\left|\mathbf{M}_{i:N,f}^{(n)}\right|=\left(\lambda_{i}+\lambda_{i+1}\e^{2n(\xi_{i+1}-\xi_i)}\right)\left|\mathbf{M}_{i+1:N,f}^{(n)}\right|-\(\lambda_{i+1}^2-\frac{1}{4}\)\e^{2n(\xi_{i+1}-\xi_i)}\left|\mathbf{M}_{i+2:N,f}^{(n)}\right|, \quad f=c, s.
	\end{equation}
\end{lem}
\begin{proof}[\bf Proof]
	We only prove the case $f=c$, and the case $f=s$ can be proved in a similar manner.
	Denote by $\mathbf{I}_{i,j*k}$  the elementary matrix which is transform of identity matrix by multiplying each element of the $j$-th row of the identity matrix by $k$, and then adding it to the $i$-th row.
	Thus, by using some elementary transformation, we can obtain
	\begin{equation}\label{ETFRC}
	\begin{aligned}
	\left|\mathbf{M}_{i:N,c}^{(n)}\right|=&
	\left|\mathbf{I}_{1,2*(-\e^{n(\xi_{i+1}-\xi_i)})}\mathbf{M}_{{i:N},c}^{(n)}\mathbf{I}^T_{1,2*(-\e^{n(\xi_{i+1}-\xi_i)})}\right|\\
	=&\left|\begin{array}{ccccc}
	\lambda_{i}+\lambda_{i+1}\e^{2n(\xi_{i+1}-\xi_i)} & -(\lambda_{i+1}+ \frac{1}{2})\e^{n(\xi_{i+1}-\xi_i)} &  \cdots & 0& 0\\
	\nm
	-(\lambda_{i+1}- \frac{1}{2})\e^{n(\xi_{i+1}-\xi_i)} & 	\lambda_{i+1}- \frac{1}{2 \e^{2n\xi_{i+1}}}& \cdots & -\frac{\cosh (n\xi_{N-1})}{\e^{n\xi_{i+1}}}& -\frac{\cosh (n\xi_N)}{\e^{n\xi_{i+1}}}\\
	\nm
	\vdots & \vdots &\ddots& \vdots &\vdots \\
	\nm
	0 &  \frac{\sinh (n\xi_{N-1})}{ \e^{n\xi_{i+1}}} & \cdots &\lambda_{N-1}- \frac{1}{2 \e^{2n\xi_{N-1}}}&-\frac{\cosh (n\xi_N)}{\e^{n\xi_{N-1}}}\\
	\nm
	0 &  \frac{\sinh (n\xi_N)}{ \e^{n\xi_{i+1}}}& \cdots & \frac{\sinh (n\xi_N)}{ \e^{n\xi_{N-1}}}& 	\lambda_{N}- \frac{1}{2 \e^{2n\xi_N}}
	\end{array}
	\right|,
	\end{aligned}
	\end{equation}
	where we only changed the first row and column. By virtue of the notation \eqref{eq:matP02} and the Laplace expansion theorem for determinant one thus has that
	\begin{equation}
	\begin{aligned}
	\left|\mathbf{M}_{i:N,c}^{(n)}\right|=\left(\lambda_{i}+\lambda_{i+1}\e^{2n(\xi_{i+1}-\xi_i)}\right)\left|\mathbf{M}_{i+1:N,c}^{(n)}\right|-\(\lambda_{i+1}^2-\frac{1}{4}\)\e^{2n(\xi_{i+1}-\xi_i)}\left|\mathbf{M}_{i+2:N,c}^{(n)}\right|
	\end{aligned}
	\end{equation}
	holds for all $N\geqslant 2$, $N\in \mathbb{N}$.
\end{proof}

\begin{rem}
	Compared to the radial case \cite{FangdengMMA23,DKLZ24}, we only perform the elementary transformation on the first row and column of the GPMs, see \eqref{ETFRC}. In fact, if we perform the elementary transformation like the radial case by
	\[
	\mathbf{I}_{N,(N-1)*(-\frac{\sinh(n\xi_N)}{\sinh(n\xi_{N-1})})}\mathbf{I}_{1,2*(-\e^{n(\xi_{i+1}-\xi_i)})}\mathbf{M}_{i:N,c}^{(n)}\mathbf{I}^T_{1,2*(-\e^{n(\xi_{i+1}-\xi_i)})}\mathbf{I}^T_{N,(N-1)*(-\frac{\cosh(n\xi_N)}{\cosh(n\xi_{N-1})})},
	\]
	there would be a rather lengthy and tedious term
	\[
%	\begin{aligned}
%	&\lambda_{N}- \frac{1}{2 \e^{2n\xi_N}}+\(-\frac{\sinh(n\xi_N)}{\sinh(n\xi_{N-1})}\)\(-\frac{\cosh (n\xi_N)}{\e^{n\xi_{N-1}}}\)+\(\frac{\sinh (n\xi_N)}{ \e^{n\xi_{N-1}}}+(-\frac{\sinh(n\xi_N)}{\sinh(n\xi_{N-1})})(\lambda_{N-1}- \frac{1}{2 \e^{2n\xi_{N-1}}})\)\(-\frac{\cosh(n\xi_N)}{\cosh(n\xi_{N-1})}\)\\
%	=&
	\lambda_{N}- \frac{1}{2 \e^{2n\xi_N}}+\frac{\sinh(n\xi_N)}{\sinh(n\xi_{N-1})}\frac{\cosh (n\xi_N)}{\e^{n\xi_{N-1}}}-\(\frac{\sinh (n\xi_N)}{ \e^{n\xi_{N-1}}}-(\frac{\sinh(n\xi_N)}{\sinh(n\xi_{N-1})})(\lambda_{N-1}- \frac{1}{2 \e^{2n\xi_{N-1}}})\)\(\frac{\cosh(n\xi_N)}{\cosh(n\xi_{N-1})}\).
%	\end{aligned}
	\]
	Thus, we give a more general form of Theorem \ref{thm52} (see \eqref{MiNc} and \eqref{MiNs}) and relatively simplify the proof by backward induction.
\end{rem}

\begin{proof}[\bf Proof of Theorem \ref{thm52}]
	We shall use backward induction to prove Theorem \ref{thm52}.
	For this, we claim that for $i = 1,2,\ldots,N$,
	\begin{equation}\label{MiNc}
	\left|\mathbf{M}_{i:N,c}^{(n)}\right|=\left\{
	\begin{array}{ll}
	(-1)^{\lfloor \frac{N+1-i}{2}\rfloor}\left(\sum\limits_{k=0}^{N+1-i}  \frac{\lambda^{N+1-i-k}}{2^k}\left(\sum\limits_{\mathbf{i}_{k}\in C_{N+1-i}^{i-1,k}}\tau_{\mathbf{i}_{k}}\e^{ 2 n\sum_{l=1}^k (-1)^{l}\xi_{i_l}}\right)\right), &  N \mbox{ is odd,}\\
	(-1)^{i-1}(-1)^{\lfloor \frac{N+1-i}{2}\rfloor}\left(\sum\limits_{k=0}^{N+1-i}  \frac{\lambda^{N+1-i-k}}{2^k}\left(\sum\limits_{\mathbf{i}_{k}\in C_{N+1-i}^{i-1,k}}\tau_{\mathbf{i}_{k}}\e^{ 2 n\sum_{l=1}^k (-1)^{l}\xi_{i_l}}\right)\right), &  N \mbox{ is even},
	\end{array}
	\right.
	\end{equation}	
	and
	\begin{equation}\label{MiNs}
	\left|\mathbf{M}_{i:N,s}^{(n)}\right|=\left\{
	\begin{array}{ll}
	(-1)^{\lfloor \frac{N+1-i}{2}\rfloor}\left(\sum\limits_{k=0}^{N+1-i}  \frac{\lambda^{N+1-i-k}}{(-2)^k}\left(\sum\limits_{\mathbf{i}_{k}\in C_{N+1-i}^{i-1,k}}\tau_{\mathbf{i}_{k}}\e^{ 2 n\sum_{l=1}^k (-1)^{l}\xi_{i_l}}\right)\right), &  N \mbox{ is odd,}\\
	(-1)^{i-1}(-1)^{\lfloor \frac{N+1-i}{2}\rfloor}\left(\sum\limits_{k=0}^{N+1-i}  \frac{\lambda^{N+1-i-k}}{(-2)^k}\left(\sum\limits_{\mathbf{i}_{k}\in C_{N+1-i}^{i-1,k}}\tau_{\mathbf{i}_{k}}\e^{ 2 n\sum_{l=1}^k (-1)^{l}\xi_{i_l}}\right)\right), &  N \mbox{ is even}.
	\end{array}
	\right.
	\end{equation}
	We only prove \eqref{MiNc}, and \eqref{MiNs} can be proved in a similar manner.
	By direct computations, we deduce that  \eqref{MiNc} holds for $i=N$.
	Suppose that \eqref{MiNc} holds for all $i\geqslant  N_0$, $N_0< N$, we show that it also holds for $i=N_0-1$.
	
		\textbf{Case i} $N$ is even. Since \eqref{eq:deter0102}  and \eqref{MiNc}  hold for $i\geqslant N_0$,  we obtain
		\begin{equation}
		\begin{aligned}
		&\quad \left|\mathbf{M}_{N_0-1:N,c}^{(n)}\right|=\left(\lambda_{N_0-1}+\lambda_{N_0}\e^{2n(\xi_{N_0}-\xi_{N_0-1})}\right)\left|\mathbf{M}_{N_0:N,c}^{(n)}\right|-\(\lambda_{N_0}^2-\frac{1}{4}\)\e^{2n(\xi_{N_0}-\xi_{N_0-1})}\left|\mathbf{M}_{N_0+1:N,c}^{(n)}\right| \\
		&=\left((-1)^{N_0-2}\lambda+(-1)^{N_0-1}\lambda\e^{2n(\xi_{N_0}-\xi_{N_0-1})}\right)\left|\mathbf{M}_{N_0:N,c}^{(n)}\right|-\(\lambda^2-\frac{1}{4}\)\e^{2n(\xi_{N_0}-\xi_{N_0-1})}\left|\mathbf{M}_{N_0+1:N,c}^{(n)}\right| \\
		&=(-1)^{N_0-2}\left(1-\e^{2n(\xi_{N_0}-\xi_{N_0-1})}\right)(-1)^{N_0-1}(-1)^{\lfloor \frac{N+1-N_0}{2}\rfloor}\lambda\left(\sum\limits_{k=0}^{N+1-N_0}  \frac{\lambda^{N+1-N_0-k}}{2^k}\left(\sum\limits_{\mathbf{i}_{k}\in C_{N+1-N_0}^{N_0-1,k}}\tau_{\mathbf{i}_{k}}\e^{ 2 n\sum_{l=1}^k (-1)^{l}\xi_{i_l}}\right)\right)\\
		&\quad -\(\lambda^2-\frac{1}{4}\)\e^{2n(\xi_{N_0}-\xi_{N_0-1})}(-1)^{N_0}(-1)^{\lfloor \frac{N-N_0}{2}\rfloor}\left(\sum\limits_{k=0}^{N-N_0}  \frac{\lambda^{N-N_0-k}}{2^k}\left(\sum\limits_{\mathbf{i}_{k}\in C_{N-N_0}^{N_0,k}}\tau_{\mathbf{i}_{k}}\e^{ 2 n\sum_{l=1}^k (-1)^{l}\xi_{i_l}}\right)\right).
		\end{aligned}
		\end{equation}
		We set
		\begin{equation}\label{ABk}
		\begin{aligned}
		&\mathfrak{A}_k:=\sum_{\mathbf{i}_{k}\in C_{N-N_0+1}^{N_0-1,k}}\tau_{\mathbf{i}_{k}}\e^{ 2 n\sum_{l=1}^k (-1)^{l}\xi_{i_l}},
		\quad \mathfrak{B}_k:=\sum_{\mathbf{i}_{k}\in C_{N-N_0}^{N_0,k}}\tau_{\mathbf{i}_{k}}\e^{ 2 n\sum_{l=1}^k (-1)^{l}\xi_{i_l}}.
		\end{aligned}
		\end{equation}
		Then there holds the following relation for all $k = 1,2,\ldots,N-N_0$,
		\begin{equation}\label{eq:solpfmn03}
		\mathfrak{A}_k-\mathfrak{B}_k=(-1)^{N_0}\e^{-2n\xi_{N_0}}\sum_{\mathbf{i}_{k-1}\in C_{N-N_0}^{N_0,k-1}}\tau_{\mathbf{i}_{k-1}}\e^{ -2 n\sum_{l=1}^{k-1} (-1)^{l}\xi_{i_l}}.
		\end{equation}
		By direct computations, we deduce that
		\[
		\begin{aligned}
		\left|\mathbf{M}_{N_0-1:N,c}^{(n)}\right|&=(-1)^{N_0-2}(-1)^{\lfloor \frac{N-N_0+2}{2}\rfloor}\left(\left(1-\e^{2n(\xi_{N_0}-\xi_{N_0-1})}\right)\lambda
		\left(\sum_{k=0}^{N-N_0+1} \frac{\lambda^{N-N_0+1-k}}{2^k}\mathfrak{A}_k\right)\right.\\
		&
		\left.\quad\quad\quad\quad\quad\quad\quad\quad\quad\;\;+\(\lambda^2-\frac{1}{4}\)\e^{2n(\xi_{N_0}-\xi_{N_0-1})}\left(\sum_{k=0}^{N-N_0} \frac{\lambda^{N-N_0-k}}{2^k}\mathfrak{B}_k\right)\right)\\
		&=:(-1)^{N_0-2}(-1)^{\lfloor \frac{N-N_0+2}{2}\rfloor}\left(\sum_{k=0}^{N-N_0+2}\frac{\lambda^{N-N_0+2-k}}{2^k}\mathrm{G}_k\right),
		\end{aligned}
		\]
		where we uesd the facts that
		\[
		(-1)^{\lfloor \frac{N-N_0+1}{2}\rfloor}(-1)^{\lfloor \frac{N-N_0+2}{2}\rfloor}=(-1)^{\lfloor \frac{N-N_0}{2}\rfloor}(-1)^{\lfloor \frac{N-N_0+2}{2}\rfloor}=-1,\;\mbox{ if }N-N_0 \mbox{ is even},
		\]
		and
		\[
		(-1)^{\lfloor \frac{N-N_0+1}{2}\rfloor}(-1)^{\lfloor \frac{N-N_0+2}{2}\rfloor}=1\;\mbox{ and }(-1)^{\lfloor \frac{N-N_0}{2}\rfloor}(-1)^{\lfloor \frac{N-N_0+2}{2}\rfloor}=-1, \;\mbox{ if }N-N_0 \mbox{ is odd}.
		\]
		It can be seen that
		\begin{equation}\label{GN-N_0+2}
		\mathrm{G}_{N-N_0+2}=-\e^{2n(\xi_{N_0}-\xi_{N_0-1})}\mathfrak{B}_{N-N_0}=\sum\limits_{\mathbf{i}_{N-N_0+2}\in C_{N-N_0+2}^{N_0-2,N-N_0-2}}\tau_{\mathbf{i}_{N-N_0+2}}\e^{ 2 n\sum_{l=1}^{N-N_0+2} (-1)^{l}\xi_{i_l}},
		\end{equation}
		and
		\begin{equation}\label{GN-N_0+1}
		\begin{aligned}
		\mathrm{G}_{N-N_0+1}&=\left(1-\e^{2n(\xi_{N_0}-\xi_{N_0-1})}\right)\mathfrak{A}_{N-N_0+1}-\e^{2n(\xi_{N_0}-\xi_{N_0-1})}\mathfrak{B}_{N-N_0-1}\\&=\left(1-\e^{2n(\xi_{N_0}-\xi_{N_0-1})}\right)\(\sum_{\mathbf{i}_{N-N_0+1}\in C_{N-N_0+1}^{N_0-1,N-N_0+1}}\tau_{\mathbf{i}_{N-N_0+1}}\e^{ 2 n\sum_{l=1}^{N-N_0+1} (-1)^{l}\xi_{i_l}}\)\\
		&\quad -\e^{2n(\xi_{N_0}-\xi_{N_0-1})}\(\sum_{\mathbf{i}_{N-N_0-1}\in C_{N-N_0}^{N_0,N-N_0-1}}\tau_{\mathbf{i}_{N-N_0-1}}\e^{ 2 n\sum_{l=1}^{N-N_0-1} (-1)^{l}\xi_{i_l}}\)\\
		& =\sum\limits_{\mathbf{i}_{N-N_0+1}\in C_{N-N_0+2}^{N_0-2,N-N_0+1}}\tau_{\mathbf{i}_{N-N_0+1}}\e^{ 2 n\sum_{l=1}^{N-N_0+1} (-1)^{l}\xi_{i_l}}.
		\end{aligned}
		\end{equation}
		%		
		%		\begin{equation}\label{G1}
		%		\begin{aligned}
		%		\mathrm{G}_1&=\left(1+t^1_{2}t^{N-1}_{N}-t^1_{2}-t^{N-1}_{N}\right)\mathfrak{A}_1+\left(1-t^1_{2}\right)t^{N-1}_{N}\mathfrak{B}_1 +t^1_{2}\(1-t^{N-1}_{N}\)\mathfrak{C}_1+t^1_{2}t^{N-1}_{N}\mathfrak{D}_1\\
		%		&\quad - \(1-t^1_{2}\)t^{N-1}_{N}- t^1_{2}\(1-t^{N-1}_{N}\)-2 t^1_{2}t^{N-1}_{N}\\
		%		& =\sum_{\mathbf{i}_2\in C_{N}^{0,2}}\tau_{\mathbf{i}_2}t^{i_{1}}_{i_{2}},
		%		\end{aligned}
		%		\end{equation}
		%		where the last equality follows by using the relations
		%		\[
		%		\begin{aligned}
		%		\mathfrak{A}_1&=\mathfrak{B}_1+\sum_{j=2}^{N-2}(-1)^{j+N-1}t^j_{N-1},  &\mathfrak{A}_1&=\mathfrak{C}_1+\sum_{j=3}^{N-1}(-1)^{2+j}t^2_j,\\
		%		\mathfrak{B}_1&=\mathfrak{D}_1+\sum_{j=3}^{N-2}(-1)^{2+j}t^2_j,  &\mathfrak{C}_1&=\mathfrak{D}_1+\sum_{j=3}^{N-2}(-1)^{j+N-1}t^j_{N-1}.
		%		\end{aligned}
		%		\]
		We can also deduce that
		\begin{equation}\label{G0}
		\mathrm{G}_{0}=\left(1-\e^{2n(\xi_{N_0}-\xi_{N_0-1})}\right)\mathfrak{A}_{0}+\e^{2n(\xi_{N_0}-\xi_{N_0-1})}\mathfrak{B}_{0} = 1,
		\end{equation}
		and
		\begin{equation}\label{G1}
		\begin{aligned}
		\mathrm{G}_{1}&=\left(1-\e^{2n(\xi_{N_0}-\xi_{N_0-1})}\right)\mathfrak{A}_{1}+\e^{2n(\xi_{N_0}-\xi_{N_0-1})}\mathfrak{B}_{1}\\
		&=\mathfrak{A}_{1}-\e^{2n(\xi_{N_0}-\xi_{N_0-1})}(\mathfrak{A}_{1}-\mathfrak{B}_{1})\\
		& = \sum_{\mathbf{i}_{1}\in C_{N-N_0+1}^{N_0-1,1}}\tau_{\mathbf{i}_{1}}\e^{ -2 n\xi_{i_l}}+(-1)^{N_0-1}\e^{-2n\xi_{N_0-1}}\\
		&=\sum_{\mathbf{i}_{1}\in C_{N-N_0+2}^{N_0-2,1}}\tau_{\mathbf{i}_{1}}\e^{ -2 n\xi_{i_l}}.
		\end{aligned}
		\end{equation}
		Furthermore, one can readily obtain that
		\begin{equation}\label{gk}
		\begin{aligned}
		\mathrm{G}_k=\left(1-\e^{2n(\xi_{N_0}-\xi_{N_0-1})}\right)\mathfrak{A}_{k}+\e^{2n(\xi_{N_0}-\xi_{N_0-1})}(\mathfrak{B}_{k}-\mathfrak{B}_{k-2})
		\end{aligned}
		\end{equation}
		for all $k=2,3, \ldots, N-N_0$.
		By using \eqref{ABk},  \eqref{eq:solpfmn03}, \eqref{gk} and some proper arrangements,  we can deduce that
		\begin{equation}\label{gk2}
		\begin{aligned}
		\mathrm{G}_k&=\left(1-\e^{2n(\xi_{N_0}-\xi_{N_0-1})}\right)\mathfrak{A}_{k}+\e^{2n(\xi_{N_0}-\xi_{N_0-1})}(\mathfrak{B}_{k}-\mathfrak{B}_{k-2})\\
		&=\mathfrak{A}_{k}-\e^{2n(\xi_{N_0}-\xi_{N_0-1})}(\mathfrak{A}_{k}-\mathfrak{B}_{k})-\e^{2n(\xi_{N_0}-\xi_{N_0-1})}\mathfrak{B}_{k-2}\\
		& = \sum_{\mathbf{i}_{k}\in C_{N-N_0+1}^{N_0-1,k}}\tau_{\mathbf{i}_{k}}\e^{ 2 n\sum_{l=1}^k (-1)^{l}\xi_{i_l}}+(-1)^{N_0-1}\e^{-2n\xi_{N_0-1}}\(\sum_{\mathbf{i}_{k-1}\in C_{N-N_0}^{N_0,k-1}}\tau_{\mathbf{i}_{k-1}}\e^{ -2 n\sum_{l=1}^{k-1} (-1)^{l}\xi_{i_l}}\)\\
		&\quad +(-1)^{2N_0+1}\e^{2n(\xi_{N_0}-\xi_{N_0-1})}\(\sum_{\mathbf{i}_{k-2}\in C_{N-N_0}^{N_0,k-2}}\tau_{\mathbf{i}_{k-2}}\e^{ 2 n\sum_{l=1}^{k-2} (-1)^{l}\xi_{i_l}}\)\\
		& = \sum_{\mathbf{i}_{k}\in C_{N-N_0+2}^{N_0-2,k}}\tau_{\mathbf{i}_{k}}\e^{ 2 n\sum_{l=1}^{k} (-1)^{l}\xi_{i_l}}
		\end{aligned}
		\end{equation}
		which completes the proof for the case that both $N_0$ and $N$ are even.
		
		\textbf{Case ii} $N$ is odd. The proof follows from a similar argument to the case that $N$ is even. We shall only briefly sketch it. First, it can be derived that
		\[
		\begin{aligned}
		\left|\mathbf{M}_{N_0-1:N,c}^{(n)}\right|
		%	&=(-1)^{N_0-2}\left(1-\e^{2n(\xi_{N_0}-\xi_{N_0-1})}\right)(-1)^{\lfloor \frac{N+1-N_0}{2}\rfloor}\lambda\left(\sum\limits_{k=0}^{N+1-N_0}  \frac{\lambda^{N+1-N_0-k}}{2^k}\mathfrak{A}_k\right)\\
		%	&\quad -\(\lambda^2-\frac{1}{4}\)\e^{2n(\xi_{N_0}-\xi_{N_0-1})}(-1)^{\lfloor \frac{N-N_0}{2}\rfloor}\left(\sum\limits_{k=0}^{N-N_0}  \frac{\lambda^{N-N_0-k}}{2^k}\mathfrak{B}_k\right)\\
		&=(-1)^{\lfloor \frac{N-N_0+2}{2}\rfloor}\left(\left(1-\e^{2n(\xi_{N_0}-\xi_{N_0-1})}\right)\lambda
		\left(\sum_{k=0}^{N-N_0+1} \frac{\lambda^{N-N_0+1-k}}{2^k}\mathfrak{A}_k\right)\right.\\
		&
		\left.\quad\quad\quad\quad\quad\quad\;+\(\lambda^2-\frac{1}{4}\)\e^{2n(\xi_{N_0}-\xi_{N_0-1})}\left(\sum_{k=0}^{N-N_0} \frac{\lambda^{N-N_0-k}}{2^k}\mathfrak{B}_k\right)\right)\\
		&=:(-1)^{\lfloor \frac{N-N_0+2}{2}\rfloor}\left(\sum_{k=0}^{N-N_0+2}\frac{\lambda^{N-N_0+2-k}}{2^k}\mathrm{G}_k\right).
		\end{aligned}
		\]
		Next, by using \eqref{GN-N_0+2}--\eqref{gk2}, we complete the proof for the case that  $N$ is odd.

	Taking $i=1$ in \eqref{MiNc} and \eqref{MiNs}, respectively, one immediately has
	that \eqref{cpbi-ktc} and \eqref{cpbi-kts} hold. 	The proof is complete.
\end{proof}

\begin{proof}[\bf Proof of Theorem \ref{thm53}]
	In Proposition \ref{NPspectral}, we obtain that $\mathbb{K}_N^*$ can be symmetrizable and its spectrum lies in the interval $[-1/2, 1/2]$. This, together with the fundamental theorem of algebra, implies that the polynomial $f^{\pm}_N(\lambda)$ has $N$ real roots that lies in the interval $[-1/2, 1/2]$. By the definition \eqref{eq:deffq01} of $f_N^{\pm}(\lambda)$, it is easy to see that if $f_N^+(\lambda^{+}) = 0$, then  $f_N^-(\lambda^{-}) = 0$.
\end{proof}

\subsection{Asymptotic profiles of the plasmon modes}

In this subsection, we study the asymptotic profiles of the plasmon modes. Specifically  we consider that the elliptic radius of the layers are extreme large and extreme small, respectively. We first study the large elliptic radius case, which reveals the relationship between the concentric disks case and the confocal ellipses case.

\begin{proof}[\bf Proof of Theorem \ref{thm54}]
Since we set $\xi_k = \tilde{\xi}+c_k$, where $\tilde{\xi}\gg 1$, and $c_1 >c_2> \ldots > c_{N}>0$ are regular constants, we can obtain that
\begin{equation}
\e^{ 2 n\sum_{l=1}^k (-1)^{l}\xi_{i_l}} =  O(\e^{-2n\tilde{\xi}}) \quad \mbox{ if }\;  k\;  \mbox{ is odd}.
\end{equation}
One can immediately have
\[
\begin{aligned}
f^{\pm}_N(\lambda)&=\sum\limits_{k=0}^{N}  \frac{\lambda^{N-k}}{(\pm 2)^k}\left(\sum\limits_{\mathbf{i}_{k}\in C_{N}^{0,k}}\tau_{\mathbf{i}_{k}}\e^{ 2 n\sum_{l=1}^k (-1)^{l}\xi_{i_l}}\right)\\
& =\sum\limits_{k=0}^{\lfloor \frac{N}{2}\rfloor}  \frac{\lambda^{N-2k}}{( 2)^{2k}}\left(\sum\limits_{\mathbf{i}_{2k}\in C_{N}^{0,2k}}\tau_{\mathbf{i}_{2k}}\e^{ 2 n\sum_{l=1}^{2k} (-1)^{l}\xi_{i_l}}\right)+O(\e^{-2n\tilde{\xi}}).
\end{aligned}
\]
The proof is complete.
\end{proof}

Next, we consider the small elliptic radius case, which is closely related to the geometrization phenomenon in the plasmon resonance.
Let
\begin{equation}\label{eq:coefg01}
h_{N,k}:=\sum_{\mathbf{i}_{k}\in C_{N}^{0,k}}\tau_{\mathbf{i}_{k}}.
\end{equation}
By direct computations, we have
\[
h_{N,1}=\sum_{\mathbf{i}_{1}\in C_{N}^{0,1}}\tau_{\mathbf{i}_{1}}=((-1)^N -1)/2, \quad h_{N,2}=\sum_{\mathbf{i}_{2}\in C_{N}^{0,2}}\tau_{\mathbf{i}_{2}}=- \lfloor N/2\rfloor,
\]
and
\[
h_{N,N}=\sum_{\mathbf{i}_{N}\in C_{N}^{0,N}}\tau_{\mathbf{i}_{N}}=(-1)^{N(N+1)/2}, \quad h_{N,2\lfloor N/2\rfloor}=\sum_{\mathbf{i}_{2\lfloor N/2\rfloor}\in C_{N}^{0,2\lfloor N/2\rfloor}}\tau_{\mathbf{i}_{2\lfloor N/2\rfloor}}=(-1)^{\lfloor N/2\rfloor}.
\]
Moreover, we can also observe that
\begin{equation}\label{eq:obse01}
h_{N, k}=h_{N-1, k}+(-1)^N h_{N-1, k-1}, \quad k= 2,3, \ldots, N-1,
\end{equation}
and
\begin{equation}\label{eq:obse02}
h_{2N, 2k-1}=\sum_{\mathbf{i}_{2k-1}\in C_{2N}^{0,2k-1}}\tau_{\mathbf{i}_{2k-1}}=0, \quad k=1, 2, \ldots, N.
\end{equation}

Next, we give some recursion formulae for $h_{N,k}$ in the following lemma.
\begin{lem}\textup{\cite[Lemma 2.7]{FangdengMMA23}}\label{le:recco01}
	Let $h_{N,k}$ be given by \eqref{eq:coefg01}. Then there holds
	\begin{equation}\label{eq:lerec01}
	h_{2N+1,2k}=h_{2N,2k},  \quad k=1, 2, \ldots, N,
	\end{equation}
	\begin{equation}\label{eq:lerec02}
	h_{2N+2,2k}=h_{2N+1,2k}-h_{2N+1,2k-2},  \quad k=2, 3, \ldots, N,
	\end{equation}
	and
	\begin{equation}
	h_{2N+1,2k-1}=h_{2N-1,2k-1}-h_{2N-1,2k-3},  \quad k=2, 3, \ldots, N.
	\end{equation}
\end{lem}

In what follows, we begin with the proof of Theorem \ref{thm55}.

\begin{proof}[\bf Proof of Theorem \ref{thm55}]
By using \eqref{eq:deffq01} and the smallness assumption on the elliptic radii $\xi_k$, $k=1,2,\ldots,N$, the associated polynomial $f^{\pm}_N(\lambda)$ is then given by
\begin{equation}\label{smallelpradii}
f^{\pm}_N(\lambda)=\sum\limits_{k=0}^{N}  \frac{\lambda^{N-k}}{(\pm 2)^k}h_{N,k}+o(1).
\end{equation}

\textbf{Case i} $N$ is even. From \eqref{eq:obse02} and \eqref{smallelpradii}, we can obtain that
\begin{equation}\label{oddeq0}
f^{\pm}_N(\lambda)=\sum\limits_{k=0}^{N/2}  \frac{\lambda^{N-2k}}{(\pm 2)^{2k}}h_{N,2k}+o(1),
\end{equation}
where $h_{N,2k}$ can be calculated by using  recursion formulae in Lemma \ref{le:recco01}.
By using elementary combination theory, we can deduce that
\begin{equation}\label{elmcombin}
h_{N,2k}=(-1)^kC_{ N/2 }^k,
\end{equation}
where $C_{ N/2 }^k$ denotes the number of combinations for $k$ out of $ N/2 $.
It follows from \eqref{oddeq0}--\eqref{elmcombin} that
\[
f^{\pm}_N(\lambda) = \sum\limits_{k=0}^{N/2}  \frac{(\lambda^2)^{N/2-k}}{4^{k}}(-1)^kC_{ N/2 }^k+o(1) = \left(\lambda^2-\frac{1}{4}\right)^{N/2}+o(1).
\]

\textbf{Case ii} $N$ is odd.
From \eqref{eq:obse01}--\eqref{eq:lerec01},  it can be readily seen that
\begin{equation}\label{oddN}
h_{N,k}=\left\{
\begin{array}{ll}
-h_{N-1,k-1}, & k \quad \mbox{is odd},\\
h_{N-1,k}, &k \quad \mbox{is even}.
\end{array}
\right.
\end{equation}
By using \eqref{smallelpradii}, \eqref{elmcombin} and \eqref{oddN}, we have
\begin{equation}
\begin{aligned}
f^{\pm}_N(\lambda)&=\sum_{k=0}^{(N-1)/2}\frac{\lambda^{N-2k}}{(\pm2)^{2k}}h_{N,2k}+
\sum_{k=1}^{(N+1)/2}\frac{\lambda^{N-(2k-1)}}{(\pm2)^{2k-1}}h_{N,2k-1}+o(1)\\
&=\sum_{k=0}^{(N-1)/2}\frac{\lambda^{N-2k}}{(\pm2)^{2k}}h_{N-1,2k}-
\sum_{k=1}^{(N+1)/2}\frac{\lambda^{N-(2k-1)}}{(\pm2)^{2k-1}}h_{N-1,2k-2}+o(1)\\
&=\sum_{k=0}^{(N-1)/2}\frac{\lambda^{N-2k}}{(\pm2)^{2k}}h_{N-1,2k}-
\sum_{k=0}^{(N-1)/2}\frac{\lambda^{N-(2k+1)}}{(\pm2)^{2k+1}}h_{N-1,2k}+o(1)\\
&=\left(\lambda\mp\frac{1}{2}\right)\sum_{k=0}^{(N-1)/2}\frac{\lambda^{N-1-2k}}{(\pm2)^{2k}}h_{N-1,2k}+o(1)\\
&=\left(\lambda\mp\frac{1}{2}\right)\sum_{k=0}^{(N-1)/2}\frac{(\lambda^2)^{(N-1)/2-k}}{4^{k}}(-1)^kC_{ (N-1)/2 }^k+o(1)\\
& = \left(\lambda\mp\frac{1}{2}\right)\left(\lambda^2-\frac{1}{4}\right)^{\frac{N-1}{2}}+o(1).
\end{aligned}
\end{equation}
The proof is complete.
\end{proof}

\section{Numerical results}\label{numer_reslt}\label{sec7}
In this section, we present some numerical examples to corroborate our theoretical findings in the previous sections. From a theoretical point of view (see Theorems \ref{thm52}--\ref{thm53}), in order to find the plasmon modes for any multi-layered structure, it suffices to find the roots of the characteristic polynomial  	$f^{\pm}_N(\lambda)$ defined in \eqref{eq:deffq01}, namely, to find the eigenvalues of NP operator $\mathbb{K}_N^*$. 
Moreover,  it is also important to understand the electric field  distribution   associated with each plasmon mode.
We next investigate these scenarios numerically.

\subsection{Mode splitting}\label{ssec71}
In this subsection, we shall compute  the eigenvalues of NP operator $\mathbb{K}_N^*$.
We first consider the elliptic radius of layers are equidistance. Let $\sigma_{0}=1$.  For $N$-layer structure, set
\begin{equation}\label{eq:str01}
\xi_i=N-i+1, \quad i=1, 2, \ldots N.
\end{equation}
Note that the ellipses are of the same focus.
Define
$
\lambda_{\pm} = (\lambda_1^{\pm},\lambda_2^{\pm},\ldots,\lambda_N^{\pm}), 
$
where $\lambda_k^{\pm}$, $k=1,2,\ldots,N,$ is the root to $f_N^{\pm}(\lambda)$, respectively, we order them in descending way, that is, $\lambda_1^{\pm}>\lambda_2^{\pm}>\ldots>\lambda_N^{\pm}$.
Table \ref{tab:1111} shows the roots to the characteristic polynomials  $f^{\pm}_{15}(\lambda)$ with $n=1$,  respectively.
\begin{table}[h]
	\begin{center}
		\caption{Roots to the characteristic polynomial \eqref{eq:deffq01} in the setup \eqref{eq:str01}. \label{tab:1111}}
		\begin{tabular}{|c|cccccccc|}
			\hline
			&  & &  &$f^{+}_{15}(\lambda)$& with $n=1$   &  & & \\
			\hline
			$\lambda_+$ & 0.3205&0.3094& 0.2898& 0.2598&  0.2172& 0.1603& 0.0896& 0.0093\\
			$\sigma_1$ &  -4.5699& -4.2469& -3.7566& -3.1624&-2.5359& -1.9441&-1.4367& -1.0378\\
			\hline
			$\lambda_+$ &  -0.0725&  -0.1468& -0.2078&-0.2540& -0.2868& -0.3081   &- 0.3202&\\
			$\sigma_1$  &-0.7467&-0.5461&-0.4128&-0.3262&-0.2710&-0.2374&-0.2193&\\
			\hline
			&  & & &$f^{-}_{15}(\lambda)$& with $n=1$    &   &&  \\
			\hline
			$\lambda_-$ &0.3202&   0.3081& 0.2868& 0.2540& 0.2078&  0.1468& 0.0725& -0.0093\\
			$\sigma_1$  & -4.5605&-4.2123&-3.6894&-3.0657&-2.4225&-1.8312&-1.3391&-0.9636\\
			\hline
			$\lambda_-$ & -0.0896& -0.1603& -0.2172& -0.2598&   -0.2898&  -0.3094& -0.3205&\\
			$\sigma_1$& -0.6960&-0.5144&-0.3943&-0.3162&-0.2662&-0.2355&-0.2188& \\
			\hline
		\end{tabular}
	\end{center}
\end{table}
\begin{figure}[htbp]
	\centering
	\includegraphics[scale=0.4]{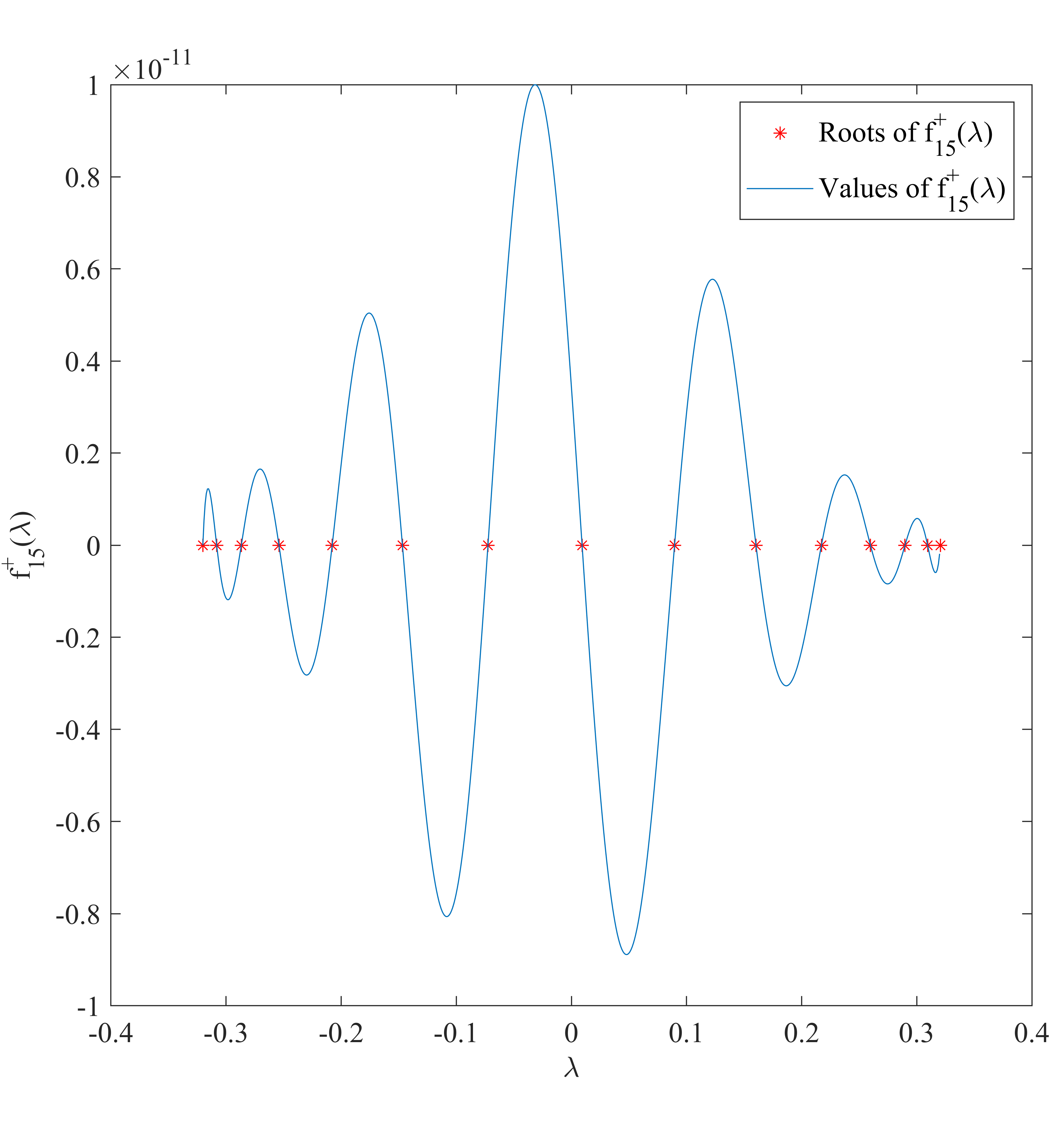}
	%\hspace{0.001in}
	\includegraphics[scale=0.4]{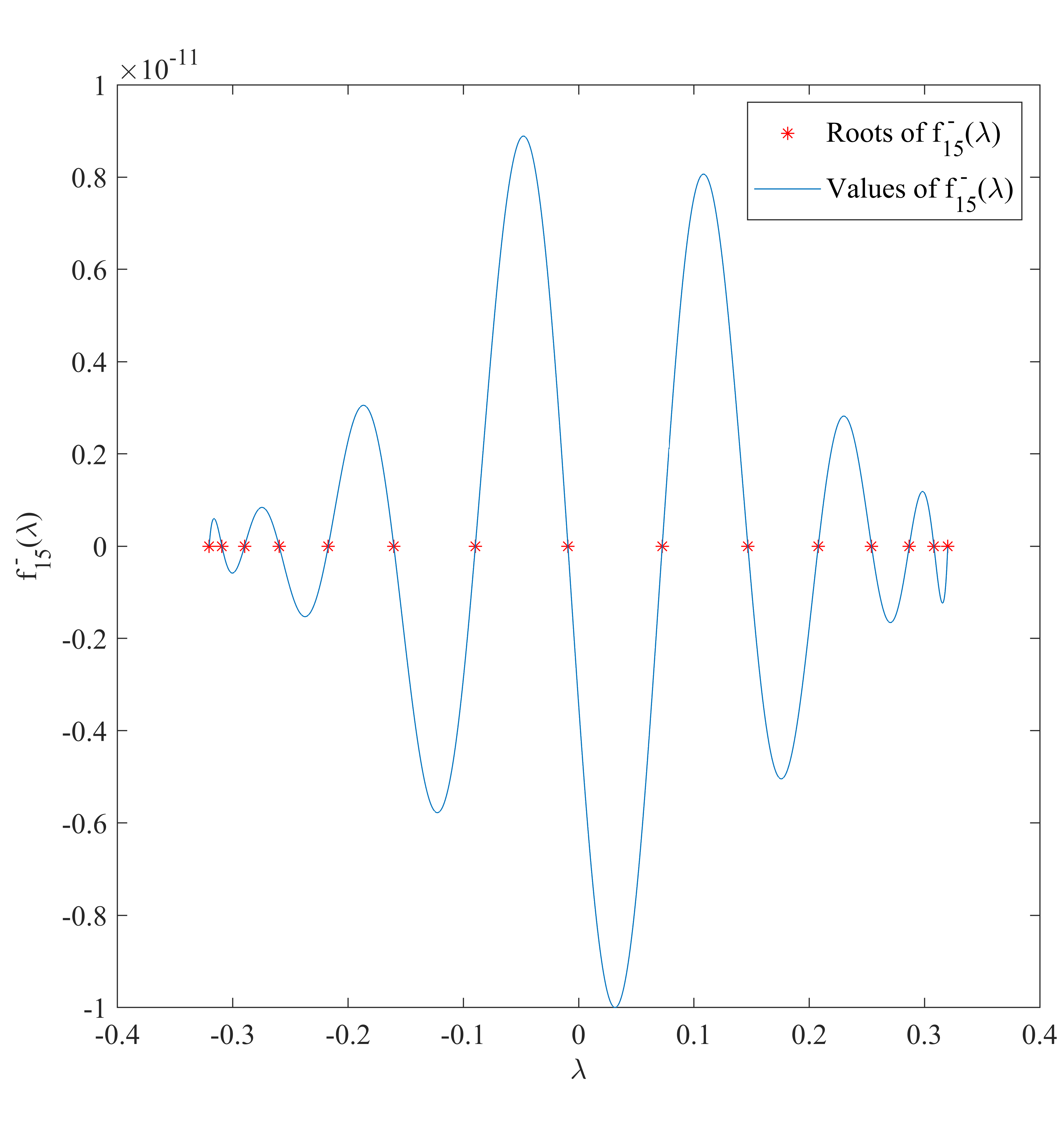}
	\captionsetup{skip=0.1pt}
	\caption{Graph on the left shows the values of $f^+_{15}(\lambda)$ and plot on
		the right  shows the values of $f^-_{15}(\lambda)$, respectively, in the setup \eqref{eq:str01} with $n=1$.}\label{fig:5}
\end{figure}

In Figure \ref{fig:5}, we plot the values of the polynomials $f^{\pm}_{15}(\lambda)$ with $n=1$ in the span  between the minimum and maximum eigenvalues, it is interesting that the values of the characteristic polynomial in the span are all very small (of the order $10^{-11}$). This is surprisingly useful in real applications, as it can be used to create a surface-plasmon-resonance-like (SPR-like) band that reduces the visibility of an object by several orders of magnitude simultaneously at multiple frequencies.

Next, we consider the radius of layers are decreasing with the same scale $s$, that is
\begin{equation}\label{eq:str02}
\xi_{i+1}=s\xi_i, \quad i=1, 2, \ldots N-1.
\end{equation}
Let $\xi_1=N$ and $s=0.8$. Table \ref{tab:222} presents all the roots to the characteristic polynomial  $f^{\pm}_{16}(\lambda)$ with $n=2$, respectively.  Similarly, in Figure \ref{fig:88}, one can also find out that the values of the polynomial $f_{16}^{\pm}(\lambda)$ with $n=2$ in the span  between the minimum and maximum eigenvalues  and the values of the polynomial $f^{\pm}_{16}(\lambda)$ with $n=2$ in the span  are all very close to zero. These  results in Tables \ref{tab:1111}-\ref{tab:222}  show excellent agreement to Theorem \ref{thm53}.
\begin{table}[h]
	\begin{center}
		\caption{Roots to the characteristic polynomial \eqref{eq:deffq01} in the setup \eqref{eq:str02}. \label{tab:222}}
		\begin{tabular}{|c|cccccccc|}
			\hline
			&  & &  &$f^{+}_{16}(\lambda)$& with $n=2$   &  & & \\
			\hline
			$\lambda_{+}$ & 0.4547&0.3831&0.2777&0.1671&0.0812&0.0296&0.0068&0.0007\\
			$\sigma_1$ &-21.0821&-7.5514&-3.4990&-2.0042&-1.3877&-1.1258&-1.0274&-1.0027\\
			\hline
			$\lambda_{+}$ &-0.0009&-0.0079&-0.0327&-0.0871&-0.1754&-0.2857&-0.3878&-0.4563\\
			$\sigma_1$  &-0.9966&-0.9690&-0.8771&-0.7033&-0.4807&-0.2728&-0.1263&-0.0457\\
			\hline
			&  & & &$f^{-}_{16}(\lambda)$& with $n=2$    &   &&  \\
			\hline
			$\lambda_-$ &0.4563&0.3878&0.2857&0.1754&0.0871&0.0327&0.0079&0.0009\\
			$\sigma_1$  & -21.9003&-7.9149&-3.6661&-2.0803&-1.4219&-1.1401&-1.0320&-1.0034\\
			\hline
			$\lambda_-$ &-0.0007&-0.0068&-0.0296&-0.0812&-0.1671&-0.2777&-0.3831&-0.4547\\
			$\sigma_1$&-0.9973&-0.9733&-0.8882&-0.7206&-0.4989&-0.2858&-0.1324&-0.0474
			\\
			\hline
		\end{tabular}
	\end{center}
\end{table}
%%%%%%%%%%%%%%%%%%%%%%%%%
%%%%%%%%%%%%%%%%%%%%%%%%%
\begin{figure}[h]
	\centering
	\includegraphics[scale=0.4]{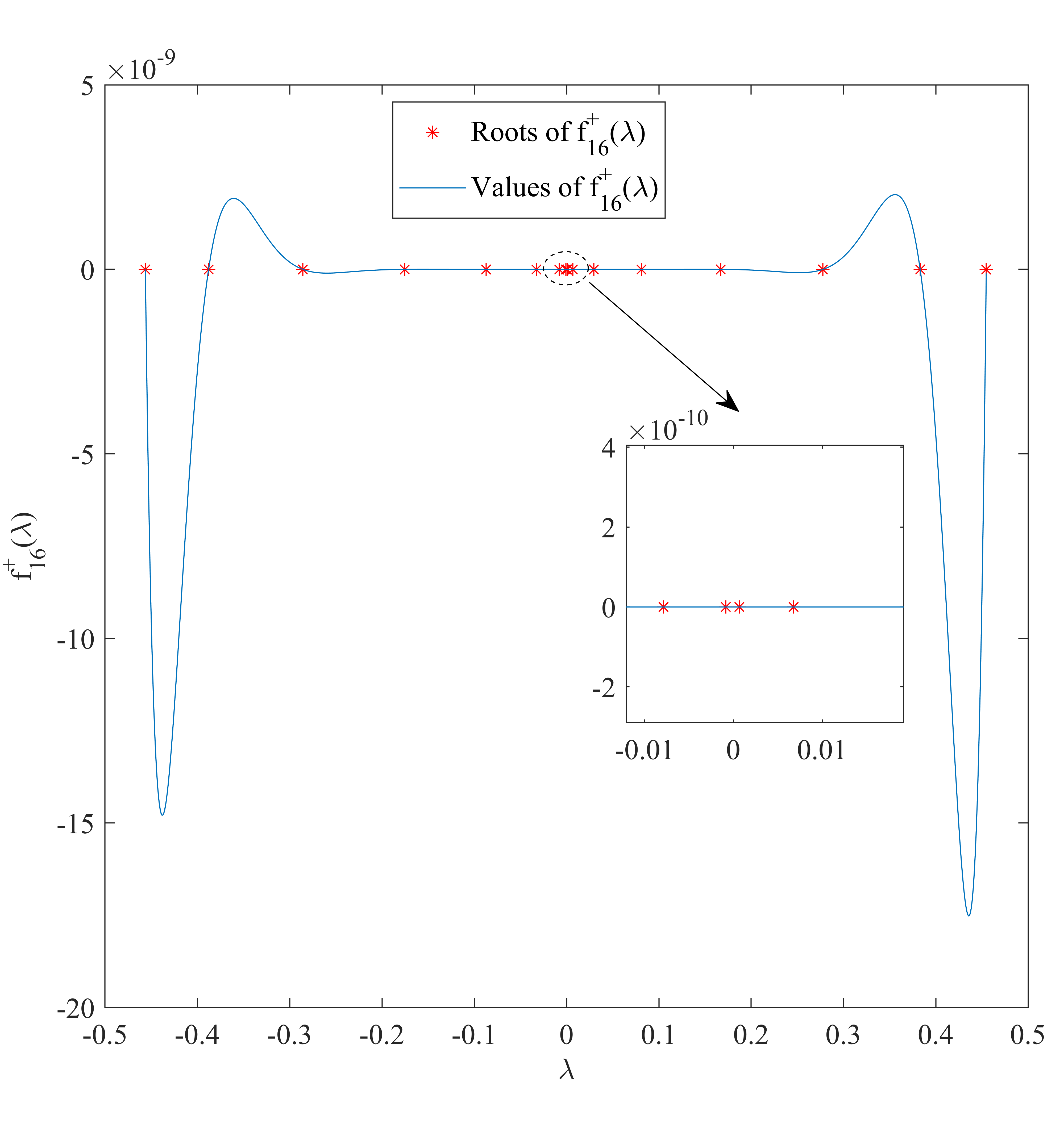}
	%\hspace{0.001in}
	\includegraphics[scale=0.4]{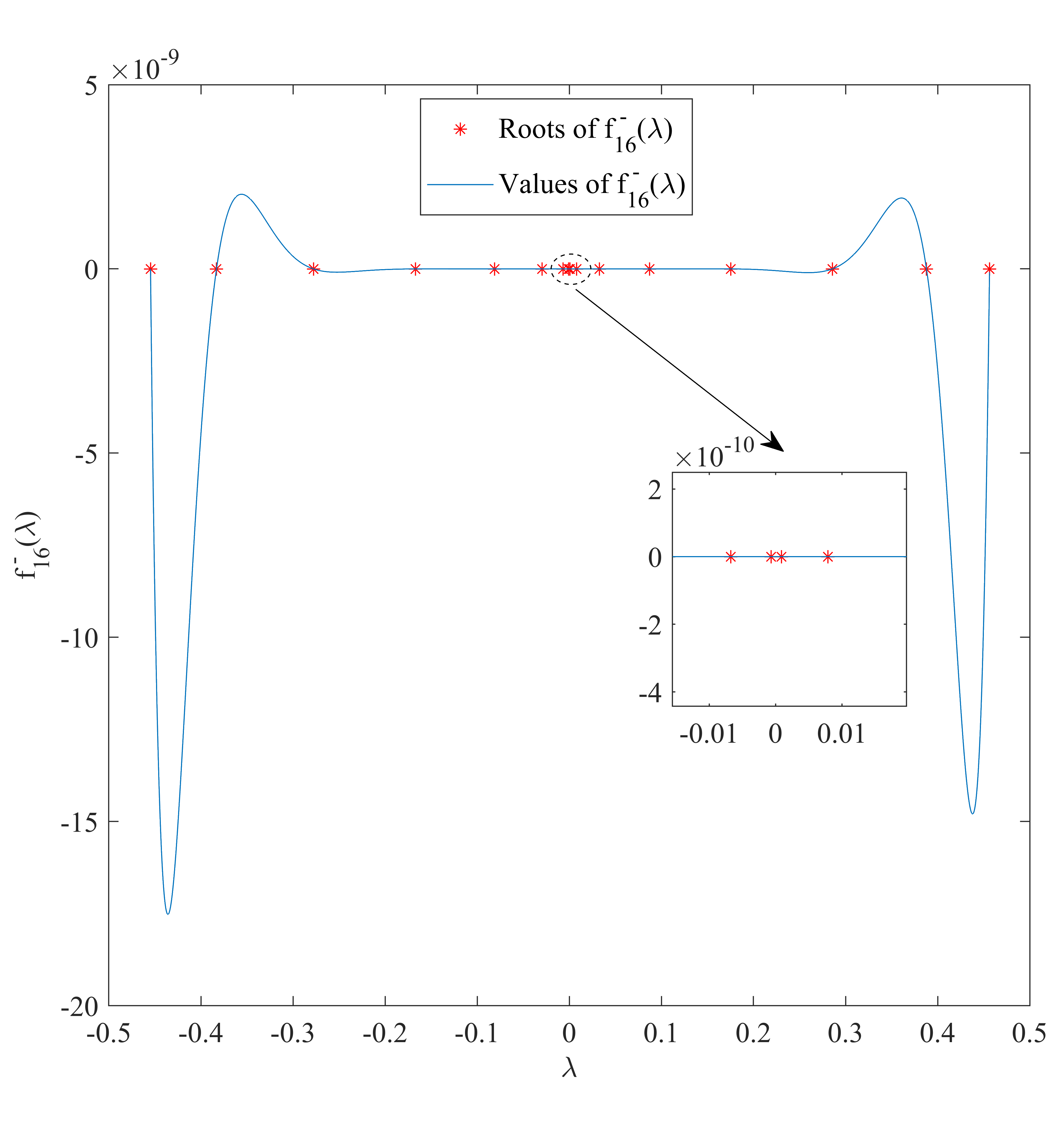}
	\captionsetup{skip=0.1pt}
	\caption{Graph on the left shows the values of $f^+_{16}(\lambda)$ and plot on the right  shows the values of $f^-_{16}(\lambda)$, respectively, in the setup \eqref{eq:str02} with $n=2$.}\label{fig:88}
\end{figure}
%\begin{table}[h]
%	\begin{center}
%		\caption{Roots to the characteristic polynomial \eqref{eq:deffq01} in the setup \eqref{eq:str02}. \label{tab:333}}
%		\begin{tabular}{|c|cccccccc|}
%			\hline
%			&  & &  &$f^{\pm}_{24}(\lambda)$& with $n=1$   &  & & \\
%			\hline
%			$\lambda$ & 0.3225&0.3179&0.3100&0.2984&0.2828&0.2626&0.2372&0.2060\\
%			$\sigma_1$ &-4.6340&-4.4915&-4.2621&-3.9613&-3.6045&-3.2123&-2.8048&-2.4017\\
%			\hline
%			$\lambda$ &0.1689&0.1259&0.0778&0.0263&-0.0263&-0.0778&-0.1259&-0.1689\\
%			$\sigma_1$  &-2.0200&-1.6728&-1.3687&-1.1113&-0.8999&-0.7306&-0.5978&-0.4951\\
%			\hline
%			$\lambda$ &-0.2060&-0.2372&-0.2626&-0.2828&-0.2984&-0.3100&-0.3179&-0.3225\\
%			$\sigma_1$  &-0.4164&-0.3565&-0.3113&-0.2774&-0.2524&-0.2346&-0.2226&-0.2158\\
%			\hline
%		\end{tabular}
%	\end{center}
%\end{table}

As previously discussed, concentric discs represent a special case where the radius of the confocal ellipse is sufficiently large (see Theorem \ref{thm54}). In order to verify this numerically, we let $\xi_1 = L*N$ and $s = 0.8$ in the setup \eqref{eq:str02}. Set $N=17$ and $n=1$.
Figure \ref{fig:9} shows that the norm of $|\lambda_{+}-\lambda_{-}|$ gradually decreases to zero with increasing $L$, which reveals the effect of rotational symmetry breaking  on the splitting of the plasmon modes, and corroborates the theoretical results \eqref{ellip2disk}.
We recover the result that the plasmon modes that split into even and odd parities in multi-layered confocal ellipses are degenerate into  plasmon modes in multi-layered concentric discs in the limiting case $\tilde{\xi}\to\infty$.

\begin{figure}[h]
	\centering
	\includegraphics[scale=0.35]{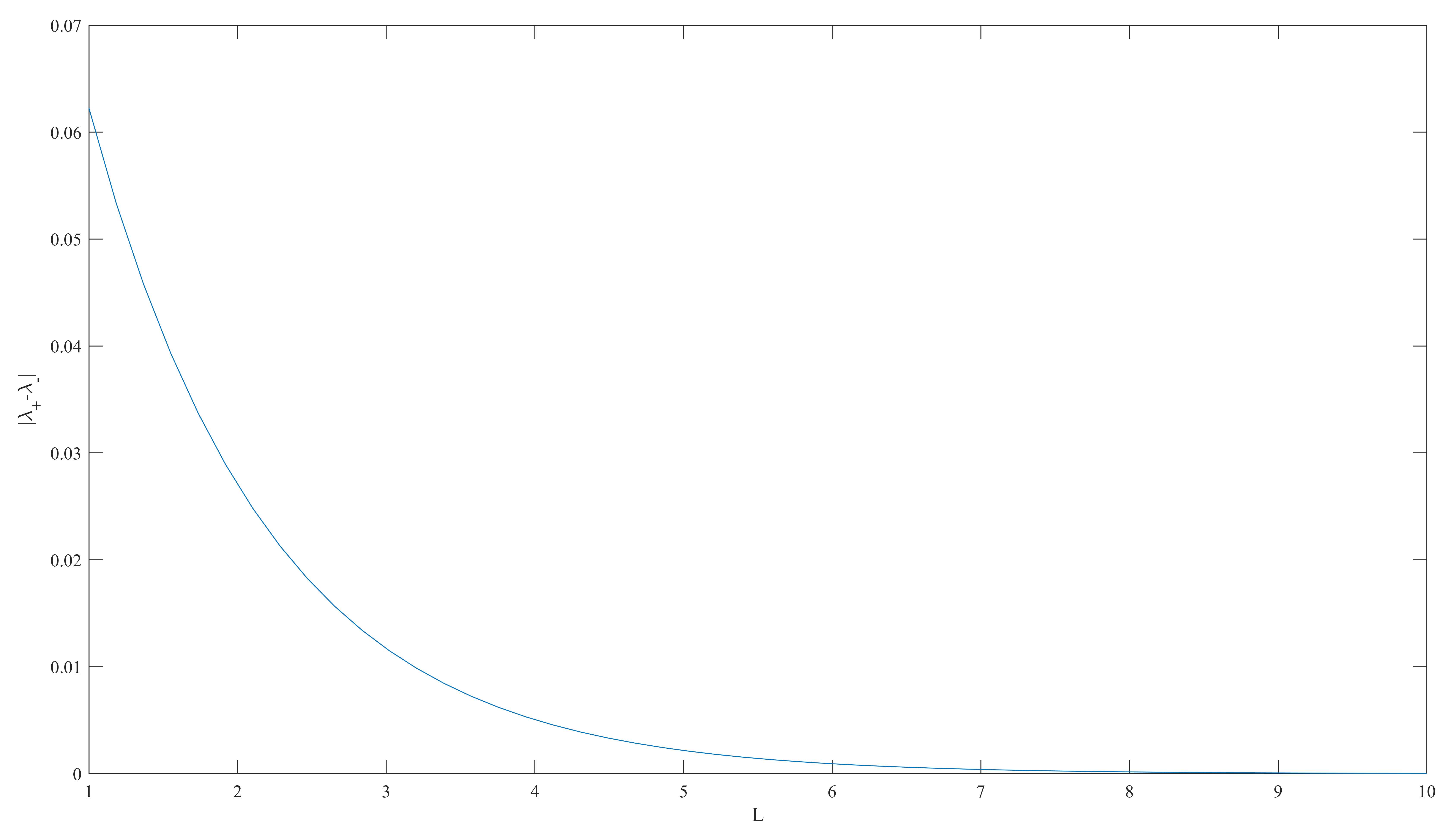}
	\captionsetup{skip=3pt}
	\caption{The value of $|\lambda_{+}-\lambda_{-}|$  in the setup \eqref{eq:str02} with $\xi_1 = L*N$, $s = 0.8$, $N=17$ and $n=1$.}\label{fig:9}
\end{figure}

\subsection{Surface localization and geometrization of plasmon resonant modes}
In subsection \ref{ssec71} above, we see that for the $N$-layer confocal ellipses and any fixed angular momentum $n\geqslant 1$, the $2N$ values of $\lambda\in [-1/2,1/2]$ such that there
exists a non-trivial solution to \eqref{cductvt_transmi_prblm}.  Then, the plasmon resonance occurs. Next, we seek to understand the distribution of the electric field associated with each plasmon mode for multi-layered structures based on Corollary \ref{cor51}.
To this end, we first introduce the Drude model for modeling the parameter $\sigma_1$ with the operating frequency $\omega$
\begin{equation}\label{eq:Drude1}
\sigma_1=\sigma_1(\omega)=\sigma'\left(1-\frac{\omega_p^2}{\omega(\omega+\mathrm{i}\tau)}\right),
\end{equation}
where we suppose that (cf. \cite{sarid_challener_2010}) 
$$\tau = 10^{14}\, s^{-1}; \sigma' = 9 \cdot 10^{-12} F \, m^{-1}; \sigma_0 = (1.33)^2 \sigma'; \omega_p = 2 \cdot 10^{15} s^{-1}.$$
\begin{figure}[htbp]
	\centering
	\includegraphics[scale=0.8]{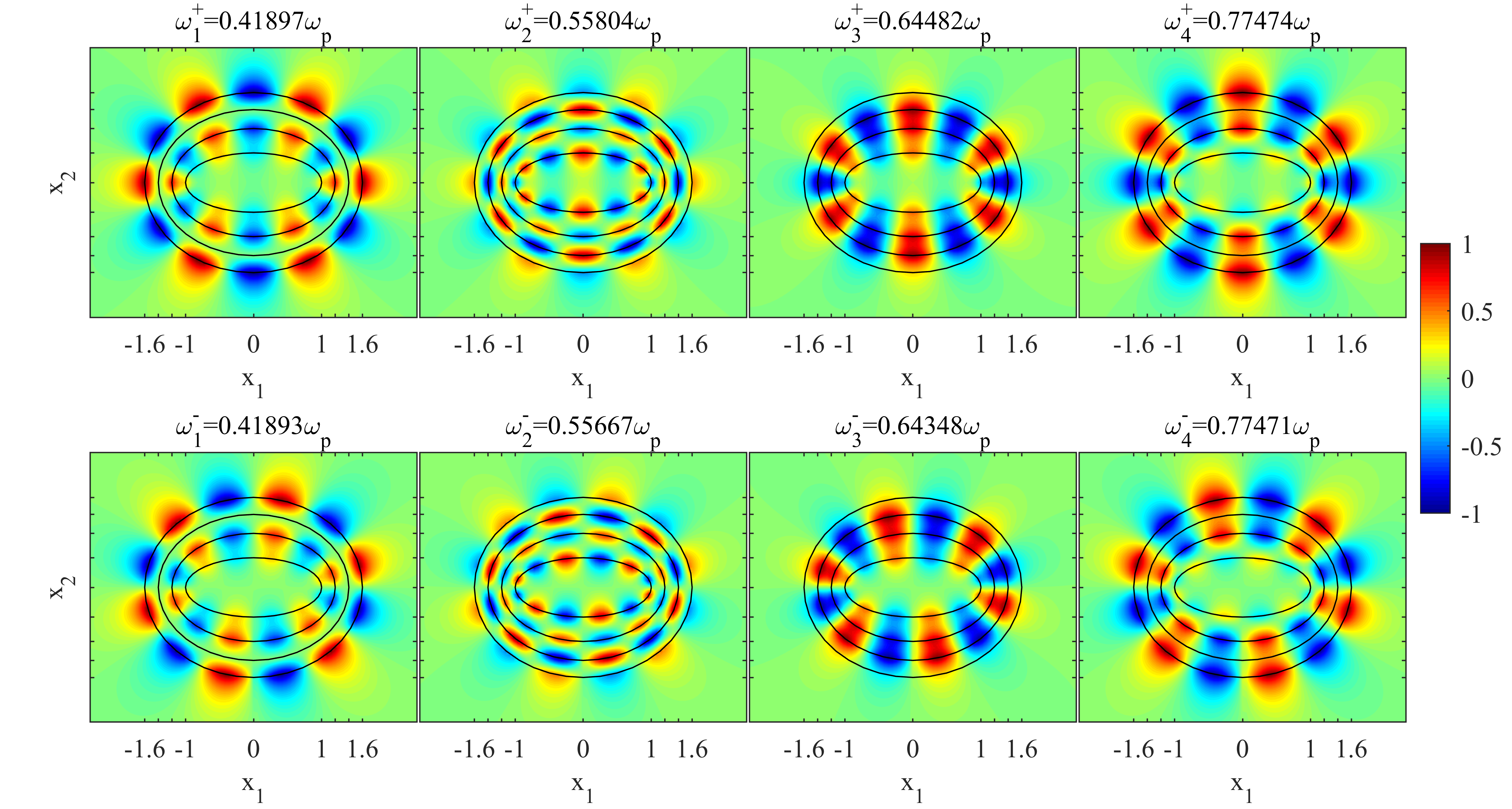}
	%\captionsetup{skip=0.1pt}
	\caption{The real part of the  perturbed electric fields  $u^{\pm}_1,u^{\pm}_2,u^{\pm}_3,u^{\pm}_4$ with $n=6$ for the four-layer confocal ellipses designed by \eqref{SMED} represented as solid black lines. Each plot corresponds to one of the eight plasmon resonance frequencies. The upper plot displays the even,  the lower plot the odd	plasmon modes.
	}\label{fig:10}
\end{figure}
\begin{figure}[htbp]
	\centering
	\includegraphics[scale=0.8]{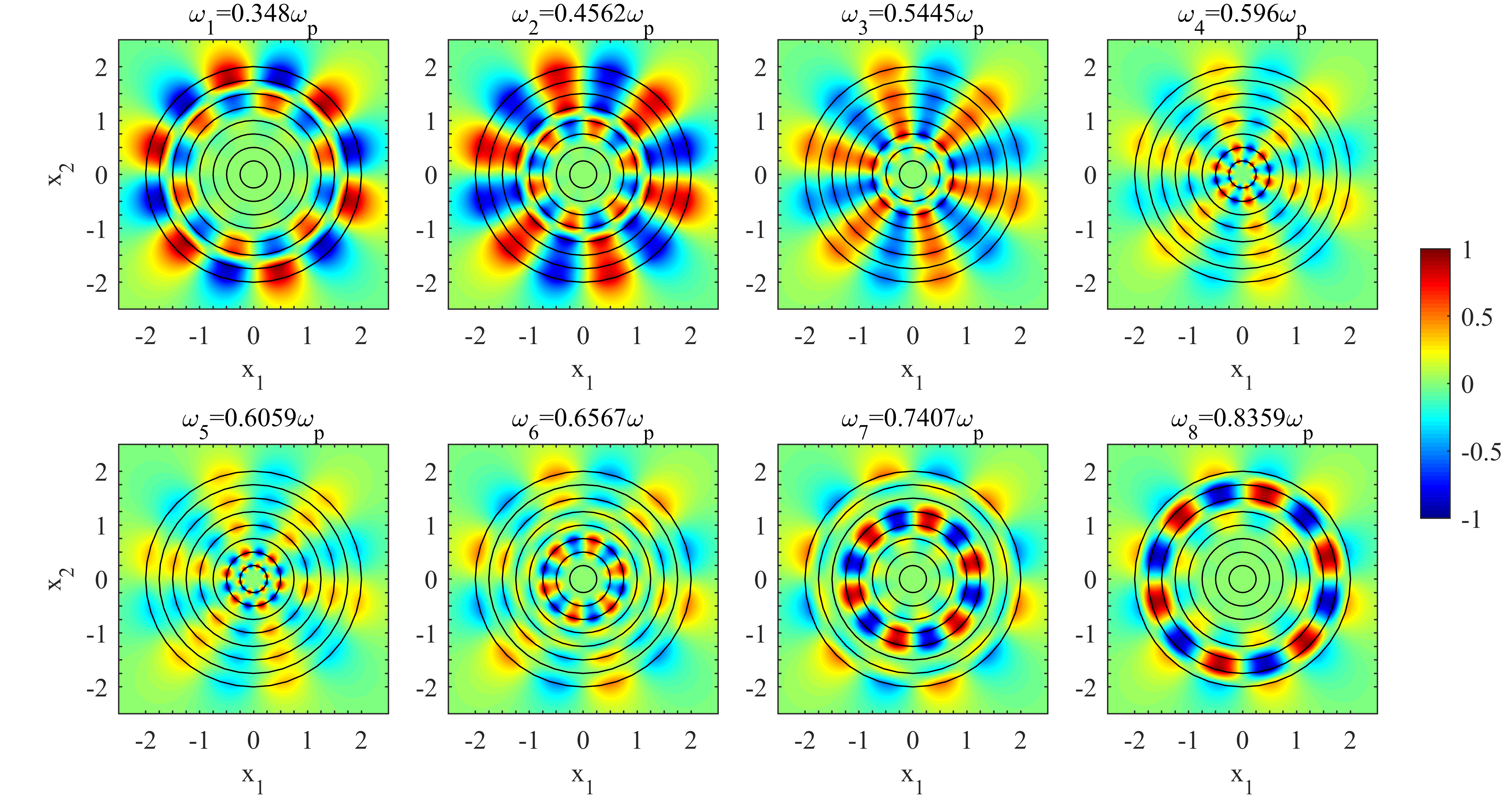}
%	\captionsetup{skip=0.1pt}
	\caption{The real part of the  perturbed electric fields  $u_1,u_2,\ldots,u_8$ with $n=6$ for the eight-layer concentric disks represented as solid black lines. Each plot corresponds to one of the eight plasmon resonance frequencies.}\label{fig:11}
\end{figure}

We consider the multi-layered confocal ellipses with equidistant semi-major axis. For $N$-layer structure, set $b_i = R\cosh\xi_i$ satisfying
\begin{equation}\label{SMED}
b_{i-1} = b_{i}+T,  \quad i=1, 2, \ldots N.
\end{equation}
Let $R = 0.9$, $b_N = 1$ ,and $T=0.2$. The eight plasmon resonant modes in the four-layer confocal ellipses designed by \eqref{SMED}  are shown in Figure \ref{fig:10}.
We fix the scaling to be such that the electric field is normalized to the interval $[-1, 1]$.
 The general splitting of the modes results from a breaking of rotational symmetry; however  the modes specifically split into twice the number of layers because the multi-layered confocal ellipses still possesses two mirror planes. For the single-layer $m$-fold rotationally symmetric, we refer to \cite{JKMA23}.

For the multi-layered structures with rotationally symmetric geometry: concentric disks, the mode splitting in it is not owing to symmetry breaking.  The reason
for the mode splitting in it is owing to truncation of materials \cite{FangdengMMA23,KZDF}, and  becomes evident from the eight plasmon  resonant modes in the eight-layer concentric disks in Figure \ref{fig:11}.  We can see that there are localized surface plasmons excited at almost each surface in multi-layered concentric disks. Because of the finite penetration depth of the light, some inner plasmon is
not excited.  There is thus a clear connection between the breaking of rotational symmetry in the $N$-layer confocal ellipses and plasmon hybridization in the $2N$-layer concentric disks.

As mentioned above, the plasmon resonance occurs locally near the high-curvature point of the plasmonic inclusion, as the curvature of that boundary point increases to a certain degree. To verify this numerically, we let $R = 1$, $b_N = 1.0001$, and $T = 0.00045$ in \eqref{SMED}.
The six plasmon resonant modes in the three-layer confocal ellipses  are shown in Figure \ref{fig:12}.
It is clearly shown that the gradient fields behave much stronger near the high curvature parts of the structure.

\begin{figure}[h]
	\centering
	\includegraphics[scale=0.35]{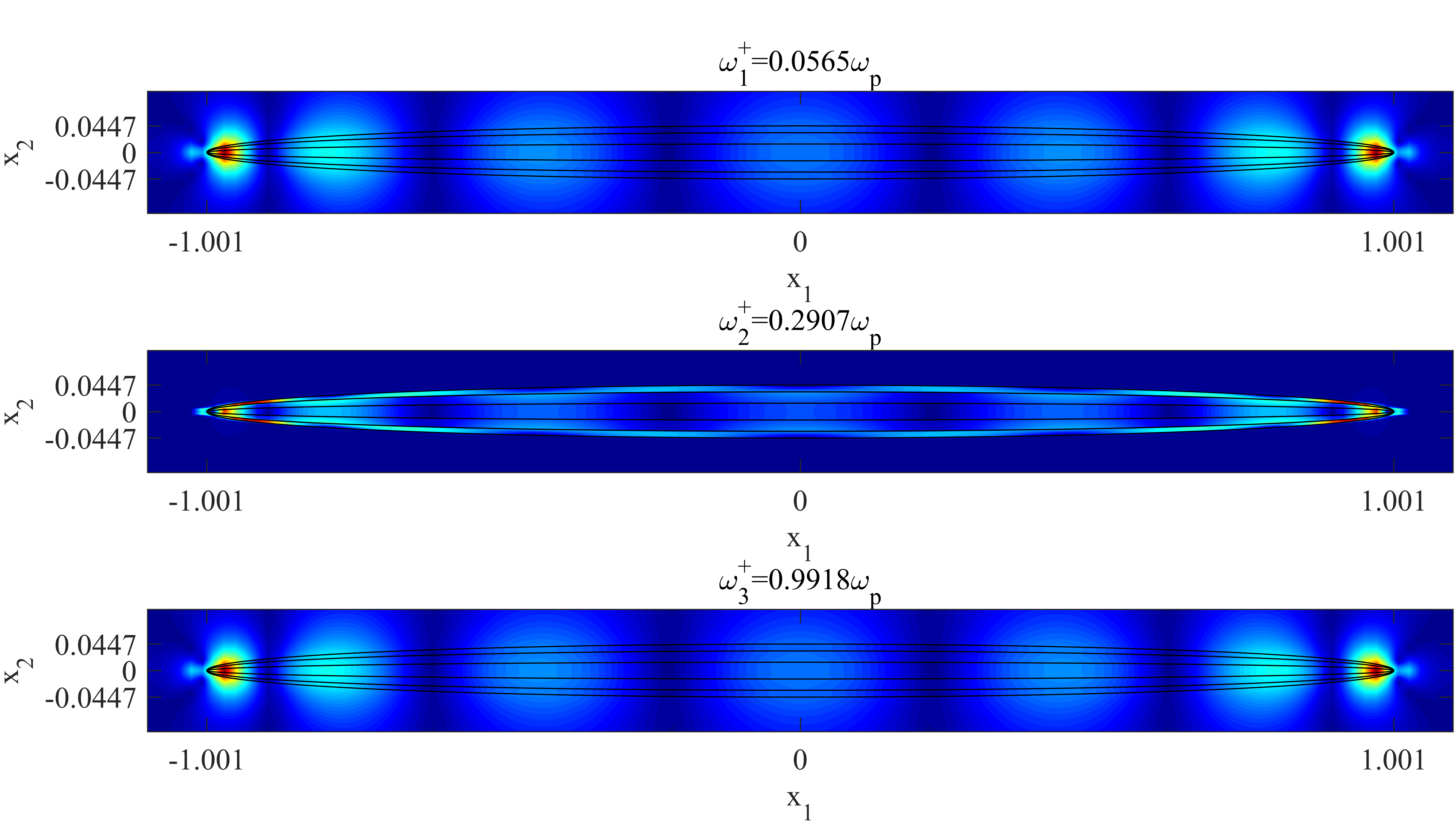}
	%\hspace{0.001in}
	\includegraphics[scale=0.35]{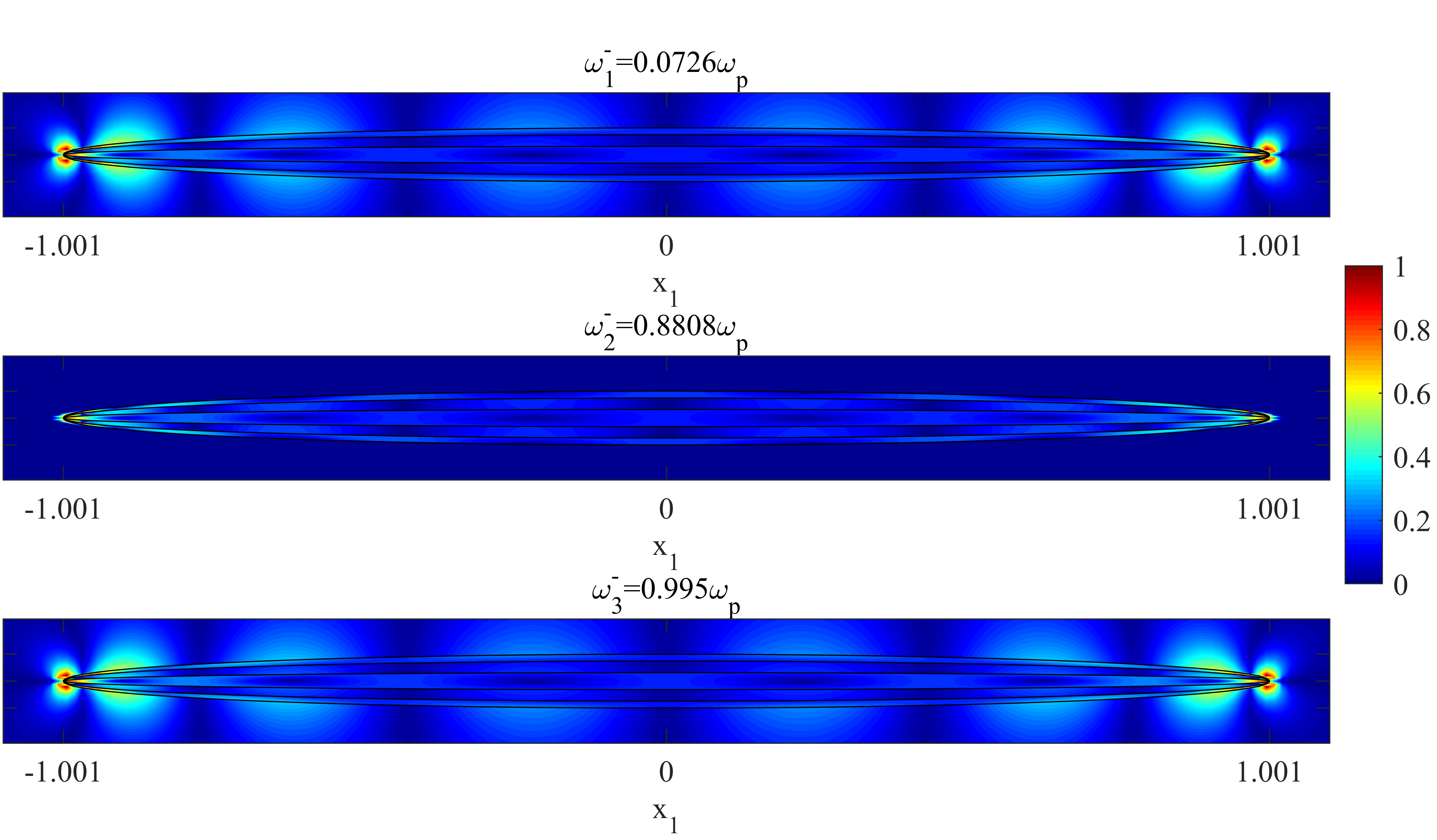}
%	\captionsetup{skip=0.1pt}
	\caption{ The real part of the perturbed gradient fields  $|\nabla u_1^{\pm}|$, $|\nabla u_2^{\pm}|$ and $|\nabla u_3^{\pm}|$ with $n=7$ for the three-layer confocal ellipses designed by \eqref{SMED} with $R = 1$, $b_N = 1.0001$, and $T = 0.00045$, represented as solid black lines. Each plot corresponds to one of the six plasmon resonance frequencies. The left plot displays the even,  the right plot the odd	plasmon modes.
	}\label{fig:12}
\end{figure}

\section{Concluding remarks}\label{sec8}
In this paper, we developed a comprehensive mathematical framework to study electrostatics within multi-layered structures. We use layer potentials and symmetrization techniques to derive a perturbation formula in terms of the NP operator applicable to general multi-layered structures. We have characterized the spectral properties of the NP operator, revealing that the number of plasmon modes increases with the number of layers. The Fourier series-based matrix representation provides precise insights into plasmon mode behavior in complex geometries, such as multi-layered confocal ellipses. Our framework facilitates a deeper understanding of how material truncation and rotational symmetry breaking influence plasmon mode splitting, paving the way for the development of tailored metamaterial applications with enhanced functionalities. The computations of this paper do not apply
to the dilation structure since different elliptic coordinates are required. We shall study this extension in our forthcoming works.

\section*{Data Availability Statement}
No data was used for the research described in the article.

\section*{Acknowledgment}
%The work of L. Kong was supported by the Fundamental Research Funds for the Central Universities of Central South University  (No. 2023ZZTS0354).
The work of Y. Deng was supported by NSFC-RGC Joint Research Grant No. 12161160314 and the Natural Science Foundation Innovation Research Team Project of Guangxi (Grant No. 2025GXNSFGA069001). The work of Z. Peng was supported by NNSF of China grant No.12461023.


\begin{thebibliography}{99}
	
\bibitem{Alu2008}
{\sc A. Al\`{u} and N. Engheta},
{\em Multifrequency optical invisibility cloak with layered plasmonic shells},
{\sl Phys. Rev. Lett.}, 72 (2008), 113901.


%\bibitem{ADAPRB2012}
%F.~Alpeggiani, S.~D'Agostino, and L.~C. Andreani.
%\newblock Surface plasmons and strong light-matter coupling in metallic
%nanoshells.
%\newblock {\em Phys. Rev. B}, 86(3):035421, 2012.

\bibitem{ACLJFA23}
{\sc H. Ammari, Y. Chow, and H. Liu},
{\em Quantum ergodicity and localization of plasmon resonances}, {\sl J. Funct. Anal.}, 285 (2023), 109976.

%\bibitem{ACLS_JEMS2025}
%{\sc H. Ammari, Y. Chow, H. Liu, and M. Sunkula},
%{\em Quantum integrable systems and concentration of plasmon resonance}, {\sl J. Eur. Math. Soc.},  27 (2025) 3407--3445.

\bibitem{ACKLM2}
{\sc H.~Ammari, G.~Ciraolo, H.~Kang, H.~Lee, and G.~W. Milton},
{\em Anomalous localized resonance using a folded geometry in three	dimensions},
{\sl  Proc. R. Soc. A}, 469 (2013), 20130048.

\bibitem{Ammari2013}
{\sc H.~Ammari, G.~Ciraolo, H.~Kang, H.~Lee, and G.~W. Milton}, {\em Spectral theory of a Neumann-Poincar{\'{e}}-type operator and analysis of cloaking due to anomalous localized resonance}, {\sl  Arch. Ration. Mech. Anal.}, 208 (2013), 667--692.

%\bibitem{AmmariPRSA29}
%H.~Ammari and B.~Davies.
%\newblock {A fully coupled subwavelength resonance approach to filtering
%	auditory signals}.
%\newblock {\em Proc. R. Soc. A}, 475(2228):20190049, 2019.


\bibitem{Ammari2016}
{\sc H.~Ammari, Y.~Deng, and P.~Millien}, {\em Surface plasmon resonance of nanoparticles and applications in	imaging}, {\sl Arch. Ration. Mech. Anal.}, 220 (2016), 109--153.


\bibitem{Ammari2007}
{\sc H. Ammari and H. Kang},
{\em Polarization and Moment Tensors with Applications to Inverse Problems and Effective Medium Theory}, (Springer-Verlag, New York, 2007).

%\bibitem{AmaariPRSA17}
%H.~Ammari, B.~Fitzpatrick, D.~Gontier, H.~Lee, and H.~Zhang.
%\newblock Sub-wavelength focusing of acoustic waves in bubbly media.
%\newblock {\em Proc. R. Soc. A}, 473(2208):20170469, 2017.

\bibitem{AKLJCM07}
{\sc H. Ammari, H. Kang, and H. Lee},
{\em A boundary integral method for computing elastic moment
tensors for ellipses and ellipsoids}, {\sl J. Comput. Math.}, 25 (2007),  2--12.

%\bibitem{AFLYZQAM19}
%H.~Ammari, B.~Fitzpatrick, H.~Lee, S.~Yu, and H.~Zhang.
%\newblock Double-negative acoustic metamaterials.
%\newblock {\em Quart. Appl. Math.}, 77(4):767--791, 2019.

\bibitem{AMRZARMA17}
{\sc H. Ammari, P. Millien, M. Ruiz, and H. Zhang},
{\em Mathematical analysis of plasmonic nanoparticles: The scalar case},
{\sl Arch. Ration. Mech. Anal.}, 224 (2017), 597--658.

\bibitem{AmmariJMA17}
{\sc H.~Ammari and H.~Zhang}.
{\em Effective medium theory for acoustic waves in bubbly fluids near minnaert resonant frequency},
{\sl SIAM J. Math. Anal.}, 49 (2017), 3252--3276.

\bibitem{AKKY18}
{\sc K. Ando, Y. Ji, H. Kang, K. Kim, and S. Yu},
{\em Spectral properties of the Neumann-Poincar\'e operator and cloaking by anomalous localized resonance for the elasto-static system},
{\sl European J. Appl. Math.,} 29 (2018), 189--225.

\bibitem{AKKY17}
{\sc K. Ando, Y. Ji, H. Kang, K. Kim, and S. Yu},
{\em Spectrum of Neumann-Poincar\'e operator on annuli and cloaking by anomalous localized resonance for linear elasticity},
{\sl SIAM J. Math. Anal.}, 49 (2017), 4232--4250.

\bibitem{Ando2016}
{\sc K.~Ando and H.~Kang},
{\em Analysis of plasmon resonance on smooth domains using spectral properties of the Neumann-Poincar{\'{e}} operator},
{\sl J. Math. Anal. Appl.}, 435 (2016), 162--178.

%\bibitem{ALMP_PRL10}
%A.~Aubry, D.~Y. Lei, S.~A. Maier, and J.~B. Pendry.
%\newblock Interaction between plasmonic nanoparticles revisited with
%transformation optics.
%\newblock {\em Phys. Rev. Lett.}, 105(23):233901, 2010.

%\bibitem{ALMPACS11}
%A.~Aubry, D.~Y. Lei, S.~A. Maier, and J.~B. Pendry.
%\newblock Plasmonic hybridization between nanowires and a metallic surface: A
%transformation optics approach.
%\newblock {\em ACS Nano}, 5(4):3293--3308,  2011.

%\bibitem{BMML_JPCC2010}
%{\sc R.~Bardhan, S.~Mukherjee, N.~A. Mirin, S.~D. Levit, P.~Nordlander, and N.~J.
%Halas},
%{\em Nanosphere-in-a-nanoshell: A simple nanomatryushka}.
%{\sl J. Phys. Chem. C}, 114 (2010), 7378--7383.

\bibitem{BLLW_ESAIM2020}
{\sc E. Bl\aa sten, H. Li, H. Liu, and Y. Wang}, {\em Localization and geometrization in plasmon resonances and geometric structures of Neumann-Poincar\'e eigenfunctions},
{\sl ESAIM Math. Model. Numer. Anal.}, 54 (2020), 957--976.


\bibitem{EBFTARMA13}
{\sc E. Bonnetier and F. Triki},
{\em On the spectrum of the Poincar\'e variational problem for two close-to-touching inclusions in 2D},
{\sl Arch. Ration. Mech. Anal.}, 209 (2013), 541--567.


\bibitem{BS11}
{\sc G.~Bouchitt\'e and B.~Schweizer},
{\em Cloaking of small objects by anomalous localized resonance},
{\sl Quart. J. Mech. Appl. Math.}, 63 (2010), 437--463.

\bibitem{CGSJLMS23}
{\sc X.~Cao, A.~Ghandriche, and M.~Sini},
{\em The electromagnetic waves generated by a cluster of nanoparticles with high refractive indices},
{\sl J. Lond. Math. Soc.}, 108 (2023), 1531--1616.

\bibitem{Chen2012}
{\sc P. Chen, J. Soric, and A. Al\`{u}},
{\em Invisibility and cloaking based on scattering cancellation}, {\sl Adv. Mater.}, 24 (2012), 281--304.

\bibitem{Chung2014}
{\sc D. Chung, H. Kang, K. Kim, and H. Lee},
{\em Cloaking due to anomalous localized resonance in plasmonic structures of confocal ellipses}, {\sl SIAM J. Appl. Math.}, 74 (2014), 1691--1707.

\bibitem{DKLZ24}
{\sc Y. Deng,  L. Kong, H. Liu, and L. Zhu},
{\em Elastostatics within multi-layer metamaterial structures and an algebraic framework for polariton resonances},
{\sl  ESAIM Math. Model. Numer. Anal.},  58 (2024), 1413--1440.

\bibitem{DLL201}
{\sc Y. Deng,  H. Li, and H. Liu},
{\em Analysis of surface polariton resonance for nanoparticles in elastic system},
{\sl SIAM J. Math. Anal.}, 52 (2020),  1786--1805.

%\bibitem{DFLMMS22}
%Y.~Deng, X.~Fang, and H.~Liu.
%\newblock Gradient estimates for electric fields with multiscale inclusions in
%the quasi-static regime.
%\newblock {\em SIAM Multiscale Model. Simul.}, 20(2):641--656, 2022.

\bibitem{DLbook2024}
{\sc Y. Deng and H. Liu}, \emph{Spectral Theory of Localized Resonances and Applications}, (Springer, Singapore, 2024).

\bibitem{DLZJMPA21}
{\sc Y. Deng, H. Liu, and G.-H. Zheng},
 {\em Mathematical analysis of plasmon resonances for curved nanorods},
 {\sl J. Math. Pure Appl.}, 153 (2021), 248--280.

\bibitem{DLZJDE22}
{\sc Y. Deng, H. Liu, and G.-H. Zheng},
{\em Plasmon resonances of nanorods in transverse electromagnetic scattering},
{\sl J. Differential Equations}, 318 (2022), 502--536.

\bibitem{FangdengMMA23}
{\sc X. Fang and Y. Deng}, {\em On plasmon modes in multi-layer structures}, {\sl Math. Methods Appl. Sci.}, 46 (2023), 18075--18095.

\bibitem{FDLMMA15}
{\sc X.~Fang, Y.~Deng, and J.~Li},
{\em Plasmon resonance and heat generation in nanostructures},
{\sl Math. Methods Appl. Sci.}, 38 (2015), 4663--4672.

%\bibitem{Gaponenko:10}
%S.~Gaponenko.
%\newblock {\em {Introduction to Nanophotonics}}.
%\newblock Cambridge University Press, Cambridge, 2010.


\bibitem{JKMA23}
{\sc Y.-G. Ji and H. Kang},
{\em Spectral properties of the Neumann-Poincar\'e operator on
rotationally symmetric domains}, {\sl Math. Ann.}, 387 (2023), 1105--1123.

%\bibitem{JPBCRWPRB72}
%P.~B. Johnson and R.~W. Christy.
%\newblock Optical constants of the noble metals.
%\newblock {\em Phys. Rev. B}, 6(12):4370--4379,  1972.

\bibitem{KKLSY_JLMS2016}
{\sc H. Kang, K. Kim,  H. Lee, J. Shin, and S. Yu},
{\em Spectral properties of the Neumann-Poincar\'e operator and uniformity of estimates for the conductivity equation with complex coefficients},
{\sl J. Lond. Math. Soc.},  93 (2016),  519--545.

\bibitem{KPSARMA2007}
{\sc D. Khavinson, M. Putinar, and H. S. Shapiro},
{\em Poincar\'e's variational problem in potential theory}, {\sl Arch. Ration. Mech. Anal.}, 185 (2007), 143--184.


\bibitem{KDZIP24}
{\sc L.~Kong, Y.~Deng, and L.~Zhu}, {\em Inverse conductivity problem with one measurement: uniqueness of multi-layer structures}, {\sl Inverse Problems}, 40 (2024), 085005.

\bibitem{KZDF}
{\sc L. Kong, L. Zhu, Y. Deng, and X. Fang}, {\em Enlargement of the localized resonant band gap by using multi-layer structures}, {\sl J. Comput. Phys.}, 518 (2024), 113308.

\bibitem{KPLM_PRB2014}
{\sc M. Kraft, J. B. Pendry, S. A. Maier, and Y. Luo}, {\em Transformation optics and hidden symmetries},
{\sl Phys. Rev. B}, 89 (2014), 245125.

\bibitem{KPN_NA13}
{\sc V.~Kulkarni, E.~Prodan, and P.~Nordlander},
{\em Quantum plasmonics: Optical properties of a nanomatryushka},
{\sl Nano Lett.}, 13 (2013), 5873--5879.

%\bibitem{LVAPL09}
%V.~Leroy, A.~Bretagne, M.~Fink, H.~Willaime, P.~Tabeling, and A.~Tourin.
%\newblock {Design and characterization of bubble phononic crystals}.
%\newblock {\em Appl. Phys. Lett.}, 95(17):171904,  2009.

\bibitem{LLLWESAIMM2AN19}
{\sc H.~Li, S.~Li, H.~Liu, and X.~Wang},
{\em Analysis of electromagnetic scattering from plasmonic inclusions
beyond the quasi-static approximation and applications},
{\sl ESAIM Math. Model. Numer. Anal.}, 53 (2019), 1351--1371.



\bibitem{LLPRSA18}
{\sc H.~Li and H.~Liu},
{\em On anomalous localized resonance and plasmonic cloaking beyond the
quasi-static limit},
{\sl Proc. R. Soc. A}, 474 (2018), 20180165.

\bibitem{LL16}
{\sc H. Li and H. Liu},
{\em On anomalous localized resonance for the elastostatic system},
{\sl SIAM J. Math. Anal.}, 48 (2016),  3322--3344.

\bibitem{LLL16}
{\sc H. Li, J. Li, and H. Liu},
{\em On novel elastic structures inducing polariton resonances with finite frequencies and cloaking due to anomalous localized resonance},
{\sl J. Math. Pures Appl.}, 120 (2018), 195--219.

%\bibitem{PhysRevE.70.055602}
%J.~Li and C.~T. Chan.
%\newblock Double-negative acoustic metamaterial.
%\newblock {\em Phys. Rev. E}, 70(5):055602,  2004.

\bibitem{Lim_IJM_01}
{\sc M. Lim}, {\em Symmetry of a boundary integral operator and a
characterization of balls}, {\sl Illinois J. Math.}, 45 (2001), 537--543.

%\bibitem{Zhang2008}
%{\sc S. Zhang, D. Genov, C. Sun, and X. Zhang},
%{\em Cloaking of Matter Waves}, {\sl Phys. Rev. Lett.}, 100 (2008), 123002.

\bibitem{Maier07}
{\sc S.~A. Maier},
{\em Plasmonics: Fundamentals and Applications}, (Springer, New York, 2007).

%\bibitem{MFZ2005}
%I.~D. Mayergoyz, D.~R. Fredkin, and Z.~Zhang.
%\newblock Electrostatic (plasmon) resonances in nanoparticles.
%\newblock {\em Phys. Rev. B}, 72(15):155412,  2005.

\bibitem{MGWNNAP_PRSA06}
{\sc G.~W. Milton and N.-A.~P. Nicorovici},
{\em On the cloaking effects associated with anomalous localized resonance},
{\sl Proc. R. Soc. A}, 2074 (2006), 3027--3059.

%\bibitem{MAJPCS08}
%A.~Moradi.
%\newblock Plasmon hybridization in metallic nanotubes.
%\newblock {\em J. Phys. Chem. Sol}, 69(11):2936--2938, 2008.

%\bibitem{Nedelec2001}
%J.-C. N\'ed\'elec.
%\newblock {\em {Acoustic and Electromagnetic Equations: Integral
%		Representations for Harmonic Problems}}.
%\newblock Springer, New York, 2001.

\bibitem{PNJCP2004}
{\sc E.~Prodan and P.~Nordlander},
{\em Plasmon hybridization in spherical nanoparticles},
{\sl J. Chem. Phys.}, 120 (2004), 5444--5454.

\bibitem{PRHN2003SCI}
{\sc E.~Prodan, C.~Radloff, N.~Halas, and P.~Nordlander},
{\em A hybridization model for the plasmon response of complex nanostructures},
{\sl Science}, 302(2003), 419--422.


\bibitem{RMSOPRB22}
{\sc M.~Ruiz and O.~Schnitzer},
{\em Plasmonic resonances of slender nanometallic rings},
{\sl Phys. Rev. B}, 105 (2022), 125412.


\bibitem{PMSOPRSA19}
{\sc M.~Ruiz and O.~Schnitzer},
{\em Slender-body theory for plasmonic resonance},
{\sl Proc. R. Soc. A}, 475 (2019), 20190294.


\bibitem{sarid_challener_2010}
{\sc D.~Sarid and W.~A. Challener},
{\em Modern Introduction to Surface Plasmons: Theory, Mathematica	Modeling, and Applications}, (Cambridge University Press, Cambridge, 2010).

%\bibitem{SKIEEE2011}
%G.~Sun and J.~B. Khurgin.
%\newblock Plasmon enhancement of luminescence by metal nanoparticles.
%\newblock {\em IEEE J. Sel. Topics Quantum Electron}, 17(1):110--118, 2011.

\bibitem{JPT2001}
{\sc J.-P. Tignol}.
 {\em {Galois' Theory of Algebraic Equations}},
(World Scientific Publishing, River Edge, NJ, 2001).

\bibitem{YA_SIAMREV18}
{\sc S.~Yu and H.~Ammari},
{\em Plasmonic interaction between nanospheres},
{\sl SIAM Review}, 60 (2018), 356--385.

\bibitem{YA_PNASU18}
{\sc S.~Yu and H.~Ammari},
{\em Hybridization of singular plasmons via transformation optics},
{\sl Proc. Natl. Acad. Sci. USA}, 116 (2019), 13785--13790.

\end {thebibliography}
\end{document}